\documentclass[12pt]{article}

\usepackage{amssymb}
\usepackage{epsfig}
\usepackage{amsmath}
\usepackage{amsthm}
\usepackage{subfigure}
\usepackage{url}
\usepackage{color}
\usepackage{float}
\usepackage{graphicx}
\graphicspath{{figures/}}
\usepackage[]{caption}
\usepackage[margin=1in]{geometry}
\usepackage{wrapfig}
\DeclareMathOperator{\transverse}{\cap\kern-7.75pt\top}

\newcommand{\T}{\operatorname{T}}

\newcommand{\supp}{\operatorname{supp}}

\newcommand{\Lip}{\operatorname{Lip}}
\newcommand{\argmin}{\operatornamewithlimits{argmin}}

\newcommand{\Env}{{\bf{Env}}(\{T_i\}_{i=1}^\nCur)}
\newcommand{\Proj}{\operatorname{Proj}}
\newcommand{\diam}{\operatorname{diam}}
\newcommand{\haus}{\operatorname{H}}
  
\newcommand{\R}{\Bbb{R}}
\newcommand{\Hd}{\mathcal{H}}
\newcommand{\F}{\Bbb{F}}

\newcommand{\bk}{\mathcal{B}}
\newcommand{\es}{\mathcal{E}}
\newtheorem{thm}{Theorem}[subsection]
\newtheorem{defn}[thm]{Definition}

\newtheorem{cor}[thm]{Corollary}
\newtheorem{lemma}[thm]{Lemma}
\newtheorem{claim}[thm]{Claim}
\newtheorem{rem}[thm]{Remark}

\definecolor{titlered}{rgb}{0.6,0,0}
\newcommand{\red}{\color{red}}
\newcommand{\green}{\color{green}}
\newcommand{\blue}{\color{blue}}

\usepackage{float}
\usepackage[bookmarks=true, % true means bookmarks
            bookmarksnumbered=false, % left window are numbered
            bookmarksopen=false, % true means only level 1 are displayed. 
            colorlinks=true, 
            linkcolor=red]{hyperref}

\newcommand{\myell}{\,\hbox{\vrule height 6pt depth 0pt
\vrule height 0.4pt depth 0pt width 5pt}\,}

\usepackage{bm}      
\usepackage{commath} 
\usepackage[nameinlink,noabbrev,capitalize]{cleveref} 
\crefalias{subequation}{equation}
\crefalias{thm}{theorem}

\let\oldlemma\lemma
\renewcommand{\lemma}{%
  \crefalias{thm}{lemma}% Theorem counter now looks like Lemma
  \oldlemma}
\Crefname{lemma}{Lemma}{Lemmas}

\let\olddefn\defn
\renewcommand{\defn}{%
  \crefalias{thm}{defn}% Theorem counter now looks like Definition
  \olddefn}
\Crefname{defn}{Definition}{Definitions}

\let\oldrem\rem
\renewcommand{\rem}{%
  \crefalias{thm}{rem}% Theorem counter now looks like Remark
  \oldrem}
\Crefname{rem}{Remark}{Remarks}

\let\oldcor\cor
\renewcommand{\cor}{%
  \crefalias{thm}{cor}% Theorem counter now looks like Corollary
  \oldcor}
\Crefname{cor}{Corollary}{Corollaries}

\let\oldclaim\claim
\renewcommand{\claim}{%
  \crefalias{thm}{claim}% Theorem counter now looks like Claim
  \oldclaim}
\Crefname{claim}{Claim}{Claims}

\usepackage{tikz}
\newcommand{\circled}[1]{\tikz[baseline=(char.base)]{
    \node[shape=circle,draw,inner sep=1pt] (char) {#1};}}
\newcommand{\Isum}{\circled{I}}

\newcommand{\m}{\hspace*{-0.08in}-}

\newcommand{\add}[1]{\textcolor{blue}{#1}}

\definecolor{darkgrn}{rgb}{0, 0.8, 0}

\newcommand{\Z}{ {\mathbb Z} }

\newcommand{\vb}{ \mathbf{b} }
\newcommand{\vc}{ \mathbf{c} }
\newcommand{\vr}{ \mathbf{r} }
\newcommand{\vs}{ \mathbf{s} }
\newcommand{\vt}{ \mathbf{t} }
\newcommand{\vv}{ \mathbf{v} }
\newcommand{\vw}{ \mathbf{w} }
\newcommand{\vx}{ \mathbf{x} }
\newcommand{\vy}{ \mathbf{y} }
\newcommand{\vzero}{ \mathbf{0} }

\newcommand{\valpha}{ \bm{\alpha} }
\newcommand{\Eu}{\mathcal{E}_U}

\newcommand{\boundary}{\partial}
\newcommand{\vol}{\operatorname{V}}
\newcommand{\mass}{\operatorname{M}}

\newcommand{\Closure}{\operatorname{Closure}}
\newcommand{\hdist}{\operatorname{hdist}}
\newcommand{\grid}{\operatorname{grid}}
\newcommand{\Comp}{\operatorname{Comp}}
\newcommand{\Cone}{\operatorname{Cone}}

\newcommand{\mn}[1]{\bar{#1}}
\newcommand{\mdn}[1]{\hat{#1}} % used \tilde previously

\newcommand{\mrwsmsp}{{\sc MRWSMSP}}
\newcommand{\dmrwsmsp}{{\sc DMRWSMSP}}

\theoremstyle{definition}
\newtheorem{definition}[thm]{Definition}
\theoremstyle{remark}
\newtheorem{remark}[thm]{Remark}

\newcommand{\nCur}{N}           % # currents
\newcommand{\ec}[1]{\kappa(#1)} % edge connectivity
\newcommand{\dimsp}{d}          % dimension of space in which currents sit
\newcommand{\dimcur}{p}         % dimension of a current in general

\usepackage{authblk}

\begin{document}

\title{{\bf Median Shapes}} 
\author[1]{Yunfeng Hu\thanks{Yunfeng.Hu90@gmail.com}}
\author[1]{Matthew Hudelson\thanks{mhudelson@wsu.edu}} 
\author[1]{Bala Krishnamoorthy\thanks{bkrishna@math.wsu.edu}}
\author[1]{Altansuren~Tumurbaatar\thanks{altaamgl@gmail.com}}
\author[1]{Kevin R. Vixie\thanks{vixie@speakeasy.net}} 
\affil[1]{Mathematics and Statistics, Washington State University}

\maketitle

\thispagestyle{empty}
\begin{abstract}
  %In this paper
  We introduce and begin to explore the mean and median of finite sets of shapes represented as integral currents.
  The median can be computed efficiently in practice, and we focus most of our theoretical and computational attention on medians.
  We consider questions on the existence and regularity of medians.
  While the median might not exist in all cases, we show that a mass-regularized median is guaranteed to exist.
  When the input shapes are modeled by integral currents with shared boundaries in codimension $1$, we show that the median is guaranteed to exist, and is contained in the \emph{envelope} of the input currents.
  On the other hand, we show that medians can be \emph{wild} in this setting, and smooth inputs can generate non-smooth medians.

  For higher codimensions, we show that \emph{books} are minimizing for a finite set of $1$-currents in $\R^3$ with shared boundaries.
  As part of this proof, we present a new result in graph theory---that \emph{cozy} graphs are \emph{comfortable}---which should be of independent interest.
  Further, we show that regular points on the median have book-like tangent cones in this case.

  From the point of view of computation, we study the median shape in the settings of a finite simplicial complex.
  When the input shapes are represented by chains of the simplicial complex, we show that the problem of finding the median shape can be formulated as an integer linear program.
  This optimization problem can be solved as a linear program in practice, thus allowing one to compute median shapes efficiently.

  We provide open source code implementing our methods, which could also be used by anyone to experiment with ideas of their own.
  The software could be accessed at \href{https://github.com/tbtraltaa/medianshape}{https://github.com/tbtraltaa/medianshape}.
\end{abstract}

\clearpage
\small{\tableofcontents}

\clearpage
\section{Introduction}
\label{sec:intro}

Our goal is to study shapes and statistics in shape spaces.
The results of any such study depend critically on how we represent shapes, and on what distance we use in that representational space.
Given that statistics in shape spaces is not a new endeavor, there have been a variety of choices for both representations of as well as distances between shapes, leading to an equally diverse set of results.
For instance, see \cite{beg-2005-computing, charpiat-2006-distance, cremers-2003-shape, goodall-1991-procrustes, haker-2004-optimal, kendall-1981-statistics, krim-2006-statistics, rohlf-1999-shape, rumpf-2011-variational}, as well as the references they contain.

In this paper, we take the (mostly) new approach of representing shapes as \emph{currents}.
This approach leads very naturally to the use of \emph{flat norm} as a distance between shapes.
Previous work on related approaches include~\cite{aouada-2009-geometric, aouada-2010-squigraphs, morgan-2007-l1tv, vixie-2010-multiscale}, earlier work by Glaunes and collaborators who used currents to represent 2-dimensional surfaces in $\R^3$ and a distance similar to the flat norm~\cite{glaunes-2007-1, glaunes-2005-1, vaillant-2005-1}, as well as the more recent work from the same group by Charon et al.~\cite{charon-2013-analysis, charon-2013-varifold, charon-2014-functional} and Kaltenmark \cite{kaltenmark-2016-geometrical}.
Perhaps the closest previous results to our work is the paper by Berkels, Linkmann, and Rumpf~\cite{berkels-2010-an-sl2}.

We work with variational definitions of means and medians, which naturally lead to optimization problems that are easy to state.
On the theoretical side, we prove several results on existence and regularity of medians.
On the computational side, the optimization problem to find the median turns out to be quite tractable (solvable as a linear program in practice).
In fact, the computational tractability also motivated in part our efforts toward the theoretical characterization of the median (as opposed to the mean).
We begin by recalling some facts about means and medians.

\subsection{Means and Medians in $\R^\dimsp$}
\label{sec:medians_in_Rn}

While the mean in the context of a set of numbers $\{x_i\}_{i=1}^\nCur\subset\R$ or, more generally, a set of points in $\R^\dimsp$, $\{\vx_i\}_{i=1}^\nCur\subset\R^\dimsp$, is most often thought of as
\[ \bar{\vx} = \frac{1}{\nCur}\sum_{i=1}^\nCur \vx_i \, ,\]
the variational definition:
\[ \bar{\vx} = \argmin_\vx\left\{\sum_{i=1}^\nCur ||\vx_i-\vx||^2\right\}\]
gives us the same result when $||\cdot||$ is the usual Euclidean norm in $\R^\dimsp$.
Analogously, the median is commonly defined as a ``middle number'' for a set of numbers: 

\medskip

\noindent \emph{A median $\hat{x}$ of a set of numbers  $\{x_i\}_{i=1}^\nCur\subset\R$ is any $\hat{x}\in\R$ such that $x_i\geq \hat{x}$ for at least half of the $i$'s and $x_i\leq  \hat{x}$ for at least half of the $i$'s}.\footnote{Sometimes the definition is modified slightly so as to produce a unique number:
sort the $x_i$ and take the middle number if $\nCur$ is odd, or take the middle two and average them if $\nCur$ is even.}

\medskip
\noindent Again, there is a variational version which gives this result when $||\cdot||$ is the Euclidean norm, which in this case (for numbers in $\R$) is also equal to the $1$-norm:
\[ \hat{x} = \argmin_x\left\{\sum_{i=1}^\nCur ||x_i-x||\right\}.\]

\medskip
\noindent In the case that $\{\vx_i\}_{i=1}^\nCur\subset\R^\dimsp$, we arrive at the following characterization of their median:

\medskip

\noindent \emph{If there is a point $\hat{\vx} \not\in\{\vx_i\}_{i=1}^\nCur$ such that
  \[ \sum_i \frac{\vx_i - \hat{\vx}}{||\vx_i - \hat{\vx}||} = \vzero \]
  then $\hat{\vx}$ is the median, otherwise $\hat{\vx}=\vx_i$ for some  $1\leq i \leq \nCur$.
}

\subsection{Shapes as Currents}
\label{sec:currents}

We represent shapes as \emph{currents}.
One can gain much of the intuition for what $\dimcur$-dimensional currents are, as well as for how they behave, by thinking of a current $T$ as a union of a finite number of pieces of oriented $\dimcur$-dimensional smooth submanifolds in $\R^\dimsp$, together with an orienting $\dimcur$-vector field on these submanifolds.

\begin{figure}[htp!]
  \centering
  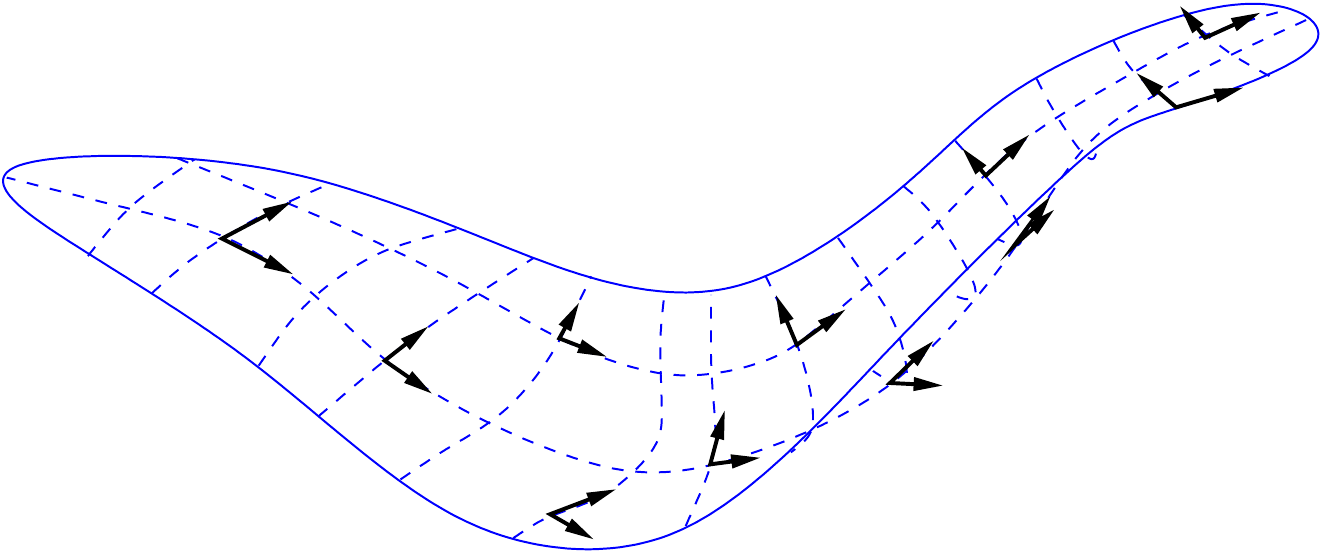
  \caption{An oriented $2$-dimensional submanifold is a $2$-current when it is used to turn $2$-forms into numbers through integration. \label{fig:submanifold}}
\end{figure}

More precisely, a $\dimcur$-current in $\R^\dimsp$ is any element of the dual space of smooth, compactly supported $\dimcur$-forms in $\R^\dimsp$.
Notice that while we can easily identify the finite union of elements from the dual space as mentioned above with the current defined by integration of a form over that finite union, there is no reason to believe that all possible currents are of this form.
In fact there is a very large zoo of currents: see for instance Chapter 4 of the book \emph{Geometric Measure Theory: A Beginners   Guide} by Frank Morgan~\cite{morgan-2008-geometric}.
(This book offers the best first look at geometric measure theory, and is written to both introduce the subject of geometric measure theory as well as to act as an interface to the authoritative reference on the subject by Federer~\cite{federer-1969-geometric}.
See also ~\cite{evans-2015-measure, krantz-2008-geometric, lin-2002-geometric, maggi-2012-sets, mattila-1999-geometry, simon-1984-lectures}.)

We work with \emph{integral currents}.
To define them, we need the notion of rectifiable sets.
For the sake of completeness, we list the definition of Hausdorff measure first.
\begin{rem}
  Hausdorff measure of a set $E \subset \R^n$ is defined using efficient covers of $E$.
  Intuitively, $\Hd^\dimcur(E)$ is the $\dimcur$-dimensional volume of $E$; we compute it as
  \[ \Hd^\dimcur(E) = \lim_{\delta\rightarrow 0}
  \inf_{\mathcal{C}_\delta} \sum
  \alpha(\dimcur)\left(\frac{\diam(C_i)}{2}\right)^\dimcur\,,  \]
  where  the $\mathcal{C}_{\delta}$'s are the collections of sets $\{C_i\}_i^\infty$ such that $E\subset \bigcup_i C_i$ and $\diam C_i < \delta$, and $\alpha(\dimcur)$ is the volume of the unit ball in $\R^\dimcur$.
  (This definition works for any real $\dimcur > 0$ in  which case $\alpha(\dimcur)$ is extended to non-integer $\dimcur$  using the $\Gamma$ function.)
\end{rem}

\begin{rem}
  Note that in $\R^d$, $\Hd^d = \mathcal{L}^d$: $d$-dimensonal
  Hausdorff measure equals Lebesgue measure in $\R^d$.
\end{rem}

\begin{defn}[Rectifiable Sets]
  A set $E$ is a $\dimcur$-rectifiable subset of $\R^\dimsp$ if
  \[ E \subset \left\{ \bigcup_i f_i(\R^\dimcur) \cup N_0 \right\}, \]
  where each of the $f_i:\R^\dimcur\rightarrow\R^\dimsp$ are Lipschitz, and the $\dimcur$-Hausdorff measure $\Hd^\dimcur(N_0) = 0$.
\end{defn}

\noindent In \cref{fig:rectifiable} we show a simple rectifiable set.
This rectifiable set can be considered perfectly nice, insofar as rectifiable sets are concerned.
That is, the singularities, when considered from a smooth perspective, where the curves cross, do not make this rectifiable curve unusual or special from the perspective of rectifiable sets.
\begin{figure}[htp!]
  \centering
  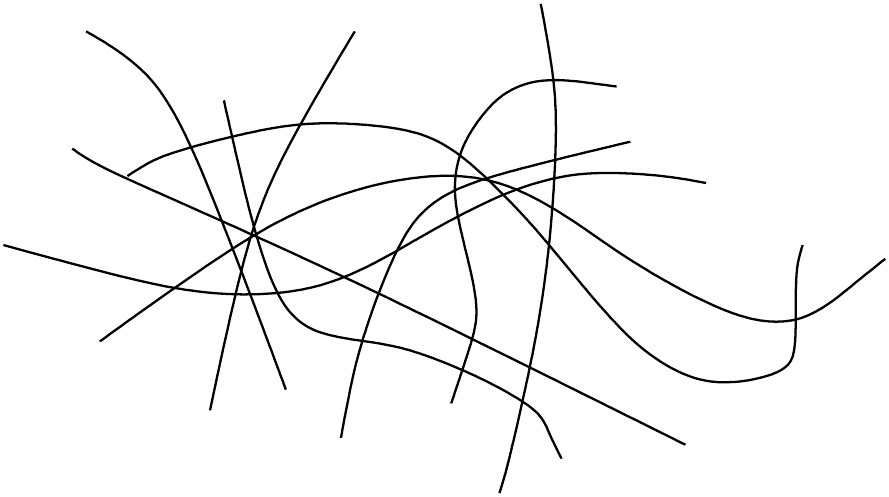 
  \caption{Example of a $1$-dimensional rectifiable set. \label{fig:rectifiable}}
\end{figure}

% ** Rectifiable set has already been defined above!! **
%\begin{defn}[Rectifiable Set]\label{def-rec-set}
%  A subset of $K\subset\R^\dimsp$ is a $k$-rectifiable set if, outside of a set of $\Hd^{k}$-measure $0$, it is a subset of the image of a countable family of Lipschitz maps into $\R^\dimsp$:
%  \[ K \subset \bigcup_{i=1}^\infty f_i(\R^k) \cup \mathcal{N}\]
%  where $\Hd^{k}(\mathcal{N}) = 0$ and $f_i:\R^k\rightarrow\R^\dimsp$ are all Lipschitz.
%  Sometimes we will also require that $\Hd^{k}(K)<\infty$.
%\end{defn}

In preparation for the definition of a current, we need the definition of $\dimcur$-vector and $\dimcur$-covector.
For a more complete, yet still accessible introduction to $\dimcur$-vector and $\dimcur$-covector (as well as currents and other ideas) see the book by Frank Morgan~\cite{morgan-2008-geometric}.
\begin{defn}[$\dimcur$-vector]\label{def-p-vec}
  Informally, but not inaccurately, one can think of a $\dimcur$-vector as the $\dimcur$-plane spanned by $\dimcur$ vectors.
  It has a magnitude equal to the $\dimcur$-volume of the parallelepiped defined by those vectors and it also has a sign, known as the orientation.
\end{defn}
\begin{defn}[$\dimcur$-covector]\label{d-covec}
  A $\dimcur$-covector is a member of the dual space to the vector space of $\dimcur$-vectors.
  In other words, it is a continuous linear functional mapping the space of $\dimcur$-vector to the real numbers.
\end{defn}
\begin{rem}
  $\dimcur$-vector fields and $\dimcur$-covector fields are simply smooth functions that assign to every point in space a $\dimcur$-vector or a $\dimcur$-covector.
  Another name for $\dimcur$-covector fields is $\dimcur$-forms.
\end{rem}

\begin{defn}[Rectifiable Currents] \label{def-recCurrent}
  We say $R$  is a rectifiable current if there is a $\dimcur$-vector field $\vec{\zeta}(\vx)$ in $\R^\dimsp$, a integer valued function  $m:\R^\dimsp\rightarrow\Z$, and a rectifiable set $E$ with $\int_E  |m(\vx)| d\Hd^{\dimcur}\vx < \infty$ such that, for any  $\dimcur$-form $\omega$,
  \[R(\omega) = \int_E m(\vx)\omega(\vec{\zeta}(\vx)) \; d\Hd^{\dimcur}\vx \, . \]
\end{defn}

\noindent In \cref{fig:current}, we show a current built using the rectifiable set shown in \cref{fig:rectifiable}.
\begin{figure}[htp!]
  \centering
  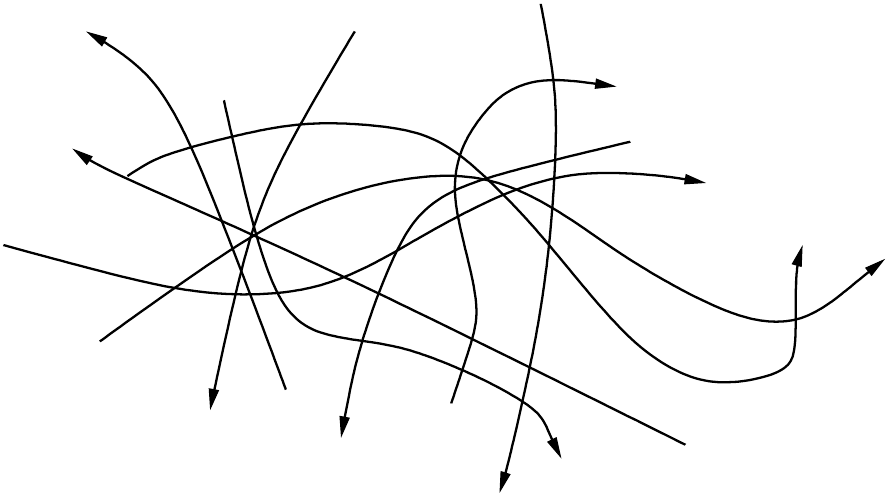
  \caption{Orienting the rectifiable set in \cref{fig:rectifiable} gives a $1$-current. \label{fig:current}}
\end{figure}

\begin{defn}[Boundary of a Current] \label{def-bdyCurrent}
  We define the boundary of a $\dimcur$-current $T$ to be the $(\dimcur-1)$-current $\boundary T$ specified by
  \[\boundary T(\omega) \equiv T(d\omega),\]
  where $d\omega$ denotes the exterior derivative of the $(\dimcur-1)$-form $\omega$.
\end{defn}

\noindent In \cref{fig:boundary}, we show the boundary of the current
shown in \cref{fig:current}.
\begin{figure}[htp!]
  \centering
  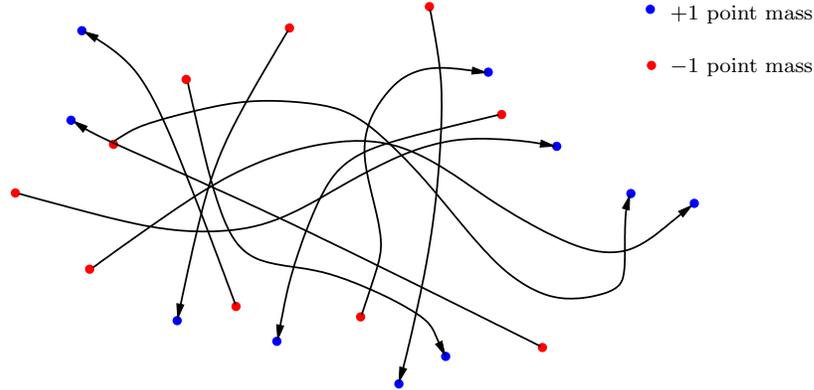
  \caption{The boundary of a $1$-dimensional current is a $0$-dimensional current.
    The boundary is the union of the red and blue point masses here.
    \label{fig:boundary}}
\end{figure}

\begin{defn}[Mass of a Current] \label{def-massCurrent}
  Let $|\cdot|$ denote the norm on the space of $\dimcur$-covectors. Then
  \[ \mass(T) = \sup_\omega \{T(\omega) : |\omega| \leq 1 \;\Hd^\dimcur \; \text{\rm almost everywhere}\}. \]
\end{defn}

\begin{rem}[Mass of Rectifiable Current]
  If $T$ is rectifiable, then $\mass(T) = \int_E |m(\vx)| d\Hd^{\dimcur}\vx < \infty$.
\end{rem}

\begin{defn}[Integral Current] \label[defn]{def-intCurrent}
  A current $I$ is an integral current if both $I$ and $\partial I$ are rectifiable currents, implying that both have finite mass, i.e.,
  \[ \mass(I) + \mass(\boundary I) < \infty. \]
\end{defn}

\begin{rem}[Integral Currents, intuitively]
  We now revisit the intuitive picture introduced in the first part of this subsection:
  one can go a long way toward understanding integral currents by thinking of a finite union of pieces of smooth, oriented  $\dimcur$-submanifolds of $\R^\dimsp$.
  While one needs to allow  infinite unions to get an arbitrary integral $\dimcur$-current in  $\R^\dimsp$ (which certainly adds another level of complication), a lot of ground can be covered with just finite unions.
\end{rem}

We work with integral currents as the representation of shapes.
While we will use more of the technology of integral currents than what we outlined above, this short introduction will help the reader to begin building an intuition for integral currents.

\subsection{The Multiscale Flat Norm}
\label{sec:flatnorm}

The flat norm, introduced by Whitney in the 1950's~\cite{whitney-1957-geometric}, turned out to be the right norm for the space of currents.
It was central to the seminal work of Federer and Fleming in 1961, in which they established the existence of minimal surfaces for a broad class of boundaries.
Under this norm, bounded sets of integral currents possess finite $\epsilon$-nets, leading to a compactness theorem and the existence of minimal surfaces.

The motivation for the flat norm can be illustrated using the following example:
consider the current $T$ defined by the unit circle centered at the origin, oriented in the counterclockwise direction and the current $T_{\epsilon}$, also a unit circle, oriented counterclockwise, but centered at $(\epsilon,0)$.
If one attempts to measure the size of the difference $T-T_{\epsilon}$ using the mass of the difference $\mass(T-T_{\epsilon})$, one finds that $\mass(T-T_{\epsilon}) = \mass(T) + \mass(T_{\epsilon})$ for all $\epsilon \neq 0$, which makes it unsuitable as a measure of distance between currents.
Instead we would like a distance that behaves more smoothly, matching the intuitive sense that this distance between $T$ and $T_{\epsilon}$ goes to $0$ as $\epsilon\rightarrow 0$.

Such a distance could be defined by decomposing the difference $T-T_{\epsilon}$ into two pieces which we measure differently.
More explicitly, we can decompose a $\dimcur$-current $H$, (for example, $H = T-T_{\epsilon}$), into two components: $H = (H - \boundary S) + (\boundary S)$, where $S$ is any $(\dimcur+1)$-current.
Now, instead of defining the size of $H$ to be $\mass(H-\partial S) + \mass(\partial S)$, we define the size of $H$---the flat norm of $H$---as the infimum:
\[
\F(H) = \inf_{S\in\mathcal{D}^{\dimcur+1}} \mass(H-\boundary S) + \mass(S),\]
where $\mathcal{D}^{\dimcur+1}$ is the space of $(\dimcur+1)$-currents.

Returning to the case of $T$ and $T_\epsilon$ above, we find that for small enough $\epsilon$, $\F(T-T_{\epsilon}) = 2\add{\pi}\epsilon + O(\epsilon^2)$, the area of the set whose boundary is $T-T_{\epsilon}$.
See \cref{fig:flatnorm} for an illustration of the flat norm for a more general instance with $T_1, T_2$ being general closed curves (rather than unit circles).

With the aid of the Hahn-Banach theorem, one can prove this infimum is always attained.
On the other hand, this result is guaranteed only if we minimize over all currents.
In the case in which we minimize over integral currents, the minimum need not be attained in all cases~\cite{ibrahim-2016-flat}.

The \emph{multiscale flat norm}, a simple yet useful generalization of the flat norm introduced by Morgan and Vixie \cite{morgan-2007-l1tv},
is given by
\[ \F_{\lambda}(H) = \inf_{S\in\mathcal{D}^{\dimcur+1}} \mass(H-\boundary S) + \lambda \mass(S),~~\mbox{ for } \lambda \geq 0.\]

\begin{figure}[htp!]
  \centering
  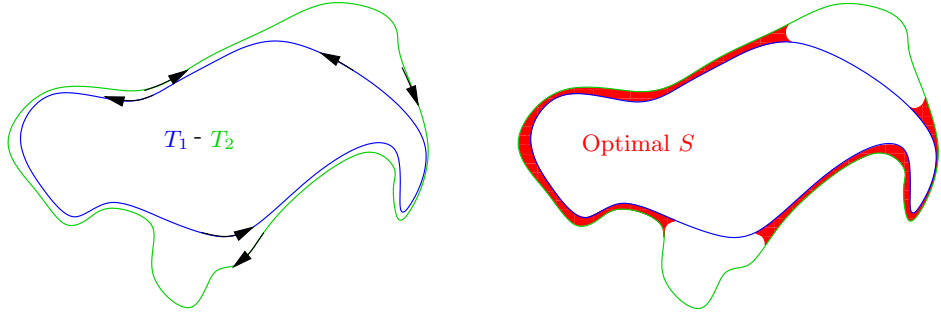  
  \caption{The optimal Flat norm decomposition of two curves $T_1$ and $T_2$.}
  \label{fig:flatnorm}
\end{figure}

\subsection{Means and Medians in the Space of Integral Currents}
\label{sec:motive}

Suppose we have a set of integral $\dimcur$-currents $\{T_i\}$.
We define their \emph{mean} as 
\begin{equation} \label{eq:defmean}
  \mn{T} = \argmin_{T\in\mathcal{I}^\dimcur} \sum_i \F_{\lambda}(T-T_i)^2 \, ,
\end{equation}
and their \emph{median} as
\begin{equation} \label{eq:defmdn}
  \mdn{T} = \argmin_{T\in\mathcal{I}^\dimcur} \sum_i \F_{\lambda}(T-T_i).
\end{equation}
Notice that we have used the variational definitions of the mean and median, and replaced $\R^\dimsp$ with the space of integral $\dimcur$-currents $\mathcal{I}^\dimcur$, and the Euclidean norm with the multiscale flat norm.

\medskip
\noindent We will study also the \emph{mass regularized} versions of the mean and median: 
\begin{equation} \label{eq:defmsregmean}
  \mn{T}_{\mu} = \argmin_{T\in\mathcal{I}^\dimcur} \sum_i \F_{\lambda}(T-T_i)^2 + \mu \mass(T) \mbox{ for } \mu \geq 0,
\end{equation}
and
\begin{equation} \label{eq:defmsregmdn}
  \mdn{T}_{\mu} = \argmin_{T\in\mathcal{I}^\dimcur} \sum_i \F_{\lambda}(T-T_i) + \mu \mass(T) \mbox{ for } \mu \geq 0. 
\end{equation}

While the mean $\mn{T}$ leads to a difficult optimization problem, the median $\mdn{T}$ computation can be cast as a linear optimization problem in practice, which can be solved efficiently.
Because of our interest in both theory and computation, we will focus on the median.
% in the remainder of this paper.

\begin{rem}
  Since a minimizer is guaranteed to exist only if we minimize over \emph{all} currents, our restriction to integral currents implies  that we will need to establish existence of a minimizer in each of our cases.
\end{rem}

\subsection{Comment on our Perspectives and Goals}
\label{sec:comments-perspective}

Geometric measure theory is, in general, rather underexploited for its potential to a wide range of application areas.
As a result, these application areas have yet to offer up their rich trove of inspirations to geometric measure theory and geometric analysis.
One serious impediment to changing this situation is the rather large investment in the effort required to master the techniques and ideas in geometric measure theory, due partly to the optimal conciseness of references like Federer's famous tome~\cite{federer-1969-geometric}.
While Frank Morgan's excellent reference~\cite{morgan-2008-geometric} has begun to address this issue, there is much more to do in this regard.

In this paper, we are attempting to span the rather large gap between those who know some geometric measure theory and those who are interested in applications in shape analysis.
Because of this setting, there are some details we include that, while not quite old hat to those who know geometric measure theory or geometric analysis well, would be considered an exercise in things ``everyone knows'', and would therefore (probably) not be written down.
The proof that regular medians have ``books'' as tangent cones (see \cref{sec:regular-books}) is one such (rather involved) exercise.
Because we feel such exercises are valuable for the uninitiated, they are included, and in great detail as well.

In fact, we believe these sorts of detailed expositions should be included more often so as to facilitate a broader impact of a wide range of mathematical works.
This is especially true in this new mathematical age in which the true symbiosis between applications and
pure theory is being seen and exploited more frequently.
While this perspective would not surprise the scientists from the past---theory and applications lived in close proximity to each other before the 20th century---it is our opinion that the happy comingling and
collaboration of the pure and the applied (across STEM fields) is still far from common enough.
In the case of this paper, we readily admit that there are pieces we do not explain in enough detail for the paper to be completely self-contained across the broad readership we think may be interested in the contents.
Nevertheless, we hope that the interested, mathematically inclined scientist-reader, willing to occasionally consult Morgan's introduction~\cite{morgan-2008-geometric} (perhaps with a mathematician friend on call), will find all the ideas accessible and understandable even if a detail or two remains a bit obscure.

It is also the case that this paper is not an attempt to solve all the problems that the developments we introduce suggest.
Rather, we hope what we write will prompt others to explore and advance the ideas we have merely begun to explore.
There are other problems and challenges, some rather low hanging---especially when we include the computational arena---that we are not trying to stake out as \emph{our} discoveries.
Indeed, we would very much like others to dig in and contribute as well.
To that end, we outline some of those problems and challenges in the discussion section at the end of the
paper.

\subsection{Outline of Paper}
\label{sec:outline}

\cref{sec:unreg} begins the remainder of the paper by showing that without further assumptions, the family of medians can, in some cases, be too big, including highly irregular currents.
Regularizing the problem with a term penalizing the mass of the median, we get existence very easily.

In \cref{sec:shared1} we move to (unregularized) median for families of codimension $1$ currents that share a common boundary, and in this context we prove an existence theorem and a theorem stating that even in the case of smooth input families, we can end up with families of medians, none of which are smooth.

Next, we turn in \cref{sec:shared2} to the case of codimension $2$ input currents.
We prove that one family of surfaces which we call \emph{books} are indeed minimizers of the implicit ensemble minimal surface problem, and are in fact minimal varifolds under Lipschitz deformations in which multiplicities are counted.
This particular proof, as well as the proof showing that regular inputs can give nonsmooth medians, relies on new results from graph theory.
We also show that in the case that the medians and the resulting minimal surfaces generated by the flat norm minimization are smooth, these \emph{books} are the tangent cones at every point on the interior of the median.

\cref{sec:simp-prelim} and \cref{sec:simplicial} introduce \emph{simplicial} currents and the simplicial multiscale flat norm, and we explain how we compute medians using simplicial representations of currents (as chains) and linear programming.
This work is motivated by previous results showing that the implicit integer optimization problem for computing the simplicial flat norm can be relaxed to a real optimization problem in many important cases.
Computational examples are explored in \cref{sec:compute}, including an illustration of the fact that these calculations can be used to interpolate smoothly between shapes.

We close with discussion of the results in \cref{sec:discuss}, along with open problems and ideas concerning where these results might be useful.

\subsection{Acknowledgments}
\label{sec:acknow}

Hu acknowledges helpful conversations with Enrique Alvarado.
Krishnamoorthy acknowledges partial funding from NSF via grant CCF-1064600.
Vixie acknowledges helpful conversations with Bill Allard and Beata Vixie.

% All notation, collected in one place
% to be input at the end of the Introduction section
\subsection{Notation} \label{ssec:notn}

We collect here all notation used throughout the (rest of the) paper.

\begin{table}[ht!]
  \centering
  \begin{tabular}{ll}
    \hline
    symbol/notation & definition/interpretation \\ \hline
    $\vx,\vr,\vs,\vt,$... &  vectors (bold lower case letters) \\
    $\mass(\cdot)$ & mass (of a current) \\
    $\mathcal{D}^\dimcur$ & space of $\dimcur$-currents in $\R^\dimsp$\\
    $\mathcal{I}^\dimcur$ & space of integral $\dimcur$-currents in $\R^\dimsp$\\
    $\F, \F_\lambda$ & flat norm, multiscale flat norm \\
    $\partial E$, $\partial^* E$ &  topological and reduced boundaries of the set E\\
    $\Hd^d$ & d-dimensional Hausdorff Measure\\
    $\mathcal{L}^d$ & Lebesgue measure in $\R^d$\\
    $B(\vx,r)$ & Euclidean open ball of radius $r$ centered at $\vx$ \\
    $\alpha(d)$ & d-Volume of unit ball in $\R^d$: $\mathcal{L}^d(B(x,r)) = \alpha(d)r^d$\\
    $\nCur$, $\{T_i\}_{i=1}^\nCur$ & number of input currents, set of input currents.\\
    $\mn{T}, \mdn{T}$, and $\mdn{T}_{\lambda,\mu}$ & mean, median, and mass-regularized median current \\
    $\supp(T)$ & support of current $T$ \\
    $[[E]]$ & integral current defined by the $d$-dimensional set $E\subset\R^d$ \\
    $\eta\myell[[E]]$ & integral current on set $E$ with integer multiplicity function $\eta$ \\
    $\mathcal{E}_U$ & a special set of $\dimcur$-currents in $\R^{\dimcur+1}$: See \cref{def-E_U} \\
    $\Env$ & envelope of input currents $\{T_i\}_{i=1}^{\nCur}$ \\
    $T\myell U$ & current $T$ with restriction to the set $U$ \\
    $T^{\pi}_i(\epsilon_s)$ & projection of current $T_i$ onto cubical grid of size $2\epsilon_s$ (\cref{thm-mdnexist})\\
    $T_{\vx} S$ & tangent space of $S$ at point $\vx$ \\
    $Cyl(r,\delta)$ & Cylinder with bottom (or top) radius $r$ and height $\delta$ \\
    $grid(\epsilon_s)$ & grid of cubes with side length $2\epsilon < R$ \\
    $\Cone(h,\theta)$ & symmetric cone with height $h$ and angle $\theta$; See \cref{fig-tgtcone,fig-tgtconeS}
  \end{tabular}

\end{table}

\clearpage

\section{Theorems and Examples for Arbitrary Integral Inputs}
\label{sec:unreg}

It appears challenging to prove results about the (unregularized)
median of a set of arbitrary integral currents.  Even existence can be
challenging, since the quantity we are minimizing does not directly
control the mass of the candidate median.  Indeed, in the next
section, we see an example where the family of medians contains
sequences of currents whose masses diverge. On the other hand, for the
regularized version of the median, we get existence using tools from
geometric measure theory developed to solve minimal surface problems.

\subsection{Mass Regularized Medians Exist} 
For the mass regularized median, we easily get existence using the compactness theorem for integral currents.

\subsubsection{Existence Theorem}

\begin{thm}[Existence of $\mdn{T}_{\lambda,\mu}$]
  Let $\{T_i\}_{i=1}^\nCur \subset \mathcal{I}^\dimcur$, and suppose further that for all $i$, the support of $T_i$ lies within a finite ball:  $\supp(T_i) \subset B(\vzero,r)$ for some $r < \infty$.
  Then there exists a $\mdn{T}_{\lambda,\mu} \in \mathcal{I}^\dimcur$ such that
%  \begin{align*}
\[    
\mdn{T}_{\lambda, \mu} = \argmin_{T \in \mathcal{I}^p} \, \sum_{i=1}^\nCur \mathbb{F}_\lambda (T-T_i) + \mu \mass(T),
\]
%\end{align*}
and we call $\mdn{T}_{\lambda,\mu}$ a \emph{mass-regularized median}.
\end{thm}

\begin{proof}
  We choose $\{P_j\}\in \mathcal{I}^\dimcur$ such that
  \[\lim_{j\rightarrow\infty}\left(\sum_{i=1}^\nCur \F_\lambda
    (P_j-T_i) +\mu \mass(P_j)\right) = \inf_{T \in
    \mathcal{I}^\dimcur} \, \sum_{i=1}^\nCur \F_\lambda (T-T_i) +\mu
  \mass(T).\] Because of the regularization term $\mu \mass(T)$, it is
  guaranteed there exists a $C < \infty$ such that $\sup_j \mass(P_j)
  < C$.  Notice that for each $i$ and $j$, there is an optimal $S_i^j
  \in\mathcal{I}^{\dimcur+1}$ such that $\F_{\lambda}(P_j - T_i) =
  \mass(P_j - T_i - \boundary S_i^j) + \lambda \mass(S_i^j)$.  Because
  none of the $T_i$'s go outside the ball $B(\vzero,r)$, we can
  radially project the minimal $S_i^j$'s and the $P_j$'s onto the ball
  $B(\vzero,r)$ and obtain a decomposition that is possibly better (if
  $P_j$ and the $S_i^j$ intersect $\R^\dimsp\setminus B(\vzero,r)$
  nontrivially).  This result implies that $P_j$ (and $S_i^j$) are
  also supported in the ball $B(\vzero,r)$.  Now we invoke the
  compactness theorem (Chapter 5 of ~\cite{morgan-2008-geometric}) to
  get a limit $\hat{P}$ of the $P_j$ that is also supported in
  $B(\vzero,r)$.

  It remains to show that this current is a median, i.e., that
  \[\sum_{i=1}^N \mathbb{F}_\lambda (\hat{P}-T_i) +\mu \mass(\hat{P})
  \rightarrow \inf_{T\in \mathcal{I}^\dimcur} \sum_{i=1}^\nCur
  \F_\lambda (T-T_i) +\mu \mass(T).\] But the flat norm is (of course)
  continuous under the flat norm, and the mass $\mass$ is lower
  semicontinuous under the flat norm.  Therefore the regularized
  median functional is lower semicontinuous under the flat norm,
  implying the result.
\end{proof}

\subsection{Medians Can Be Trivial}
We proved that mass regularized medians always exist.  However, this
result does not imply the median has to be nontrivial.  In fact, in
some cases, it can only to be trivial.  In \cref{exm:0-current}, we
show that the unique, unregularized median for a particular set of
three input currents is the trivial (or empty) $0$-current.
Furthermore, we explain that the unique, regularized median is also
the trivial $0$-current in this case (in \cref{rem:0-current}).

%\begin{exm}\label{exm:0-current}
\begin{lemma}[Medians can be trivial]\label{exm:0-current}
  Let $\lambda =1$ and let $T_1$, $T_2$ and $T_3$ be three
  $0$-currents (signed masses), each with mass $1$ and positive
  orientation $+1$, which are more than $4$ units away from each
  other. Then the unique median for $T_1,T_2$ and $T_3$ is the trivial
  $0$-current.
  %\end{exm}
\end{lemma}

\begin{proof}
  Notice first that the objective function for the median (the
  functional we minimize to find median in \cref{eq:defmdn}) has value
  $3$ when $\mdn{T} = 0$.  Let $T$ be a nontrivial candidate median.
  Since it is an integral current, $\mdn{T}$ is a finite number of
  point masses, each with sign $+1$ or $-1$ -- note that we can get
  points with other integer multiplicities by just having some of the
  points coincide.  We consider two cases based on the cardinality of,
  i.e., number of (possibly non-distinct) points in, $T$.
  \begin{enumerate}
  \item {\bf $\mass(T)$ is even:} For each input current $T_i$, $\F_{1}(T -
    T_i) \geq 1$.  This follows because $\mass(T-T_i)$ is odd, $\mass(\partial
    S_i)$ of any integral 1-current $S_i$ is an even integer, and
    \[\mass(T-T_i - \partial S_i) \geq |\mass(T-T_i) - \mass(\partial S_i)|.\] A
    little more slowly, if we take the absolute values of the
    multiplicities of all the points in $T-T_i$ and sum them up, we get
    an odd integer.  Any $1$-current $S_i$ has boundary made up of
    pairs of points with equal multiplicity. Thus $\mass(\partial S_i)$ is
    even. Now because
     \[\mass(T-T_i - \partial S_i) \geq |\mass(T-T_i) - \mass(\partial S_i)|,\]
    we conclude that 
    \begin{eqnarray*}
      \F_{1}(T-T_i) &=& \inf_{S_i} \mass(T-T_i - \partial S_i) + \mass(S_i) \\
                  &\geq& 1 + \mass(S_i) \\
                  &\geq& 1
    \end{eqnarray*}
    Note that if any of the minimizing $S_i$'s are nonempty, then this
    also shows that $\F_{1}(T-T_i) > 1$ and, for that $T$, we have
    that the sum of the flat norms is strictly greater than $3$.  

    If all the $S_i$ are empty, then we have that either $\mass(T) = 0$ and
    $T$ is the empty $0$-current, {\bf or} $\mass(T) \geq 2$ and $\F_{1}(T
    - T_i ) > 1$ for some $i$.
  \item {\bf $\mass(T)$ is odd:} 
    \begin{enumerate}
    \item    Define $R_i$ to be the $1$-current of minimal length such, as sets
    of points (i.e. ignoring orientation) $T - T_i$ and $\boundary
    R_i$ are equal.
    \item Now consider the sign assignments to the points in each
    $T-T_i$.  Notice that either the numbers of $+1$ and $-1$ points
    are always equal for all $i$, or always not equal for all $i$.  
    \item If the number of $+1$ points does not equal the number of $-1$
    points in $T-T_i$, then $\F_1(T-T_i) = \mass(T - T_i -\boundary
    S_i) + \mass(S_i) \geq 2$, and in this case, the sum of the flat
    norms (over all $i$) is at least $6$, and we are done. Hence we
    assume we have matching numbers of $+1$ and $-1$ points in $T-T_i$
    for all $i$.
    \item If the number of $+1$ points equals the number of $-1$ points in
    $T-T_i$, then $\F_1(T-T_i) = \mass(T - T_i -\boundary S_i) +
    \mass(S_i) \geq \mass(R_i)$ for any $S_i$ that spans $T-T_i$,
    i.e., with $\boundary S_i = T-T_i$.
  \item If there are two or more $i$ where the optimal $S_i$ given by
    the flat norm decomposition does not span $T - T_i$, then the sum
    of the flat norms is at last $4$, and we are done.  Hence we
    assume at least two of the $i$ have optimal $S_i$ that span $T -
    T_i$. Without loss of generality, assume that $S_1$ and $S_2$ span
    $T - T_1$ and $T - T_2$
  \item Then we get 
    \begin{eqnarray*}
     \F_{1}(T - T_1) &+& \F_{1}(T - T_2) \\
             &=& \mass(T-T_1 -\boundary S_1) + \mass(S_1) + \mass(T -
    T_1 -\boundary S_1) + \mass(S_1)\\ 
                                     &\geq& \mass(R_1) + \mass(R_2).      
    \end{eqnarray*}

  \item We claim $R_1 \cup R_2$ ``spans'' $T_1$ and $T_2$ in the sense
    that there is a path in $R_1 \cup R_2$ from $\supp(T_1)$ to
    $\supp(T_2)$.  If this result holds, we are done because the
    distance between the point supports of $T_1$ and $T_2$ exceeds
    $4$.
  \item To see that this claim image that the line segments that make
    up $R_1$ and $R_2$ are colored red and blue, respectively.
  \item Notice that we allow the case in which these line segments have
    length equal to zero, which happens when $T_1$ and or $T_2$
    coincide with a point of $T$ of the opposite orientation.
  \item Imagine drawing both $R_1$ and $R_2$ at the same time, with
    the different colors.
  \item Now begin at $T_1$ and move along the red edge to an element
    of $T$. Now move along the blue edge that must end on that element
    of $T$ to another node in $(T-T_1) \cup (T-T_2)$. This node will
    not be $T_1$. we keep moving from node to node until we end on
    $T_2$. See Figure~\ref{fig:zero-current-graph-1}
    \begin{figure}[htp!]
      \centering
     \scalebox{1.0}{
     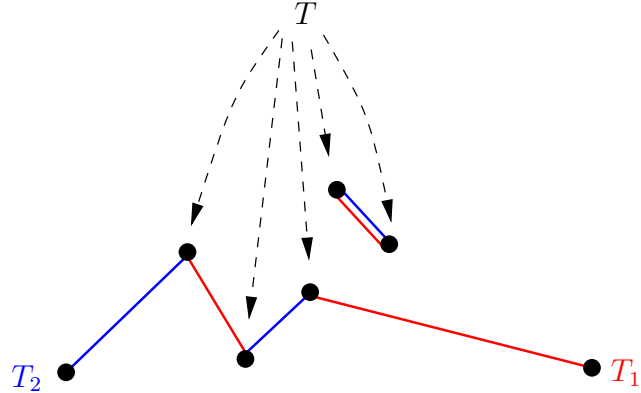}      
      \caption{$R_1 \cup R_2$ contains a path from $T_1$ to $T_2$}
      \label{fig:zero-current-graph-1}
    \end{figure}
  \item Once we leave a node in this path, we never return since to do
    so would mean that three edges end on that node. Since there is
    only one other node with degree $1$ (in the graph theoretic sense),
    $T_2$, the path must end there.
  \item Notice that the argument works even if one of the beginning red
    or ending blue (or both) shrink to a length of zero, i.e. if nodes
    in $T$ coincide with $T_1$ or $T_2$ or both.
  \item This completes the proof.
    \end{enumerate}
  \end{enumerate}
\vspace*{-0.3in}
\end{proof}

\begin{rem}
\label{rem:0-current}
The above example shows that for particular input
  $0$-currents $T_1$, $T_2$ and $T_3$, the unique unregularized
median is the trivial $0$-current.  If we regularize the
objective function of the median (as in \cref{eq:defmsregmdn}), then
we still get the trivial $0$-current as the unique median for
the same $3$ input currents.  This result follows from the fact
  that the regularized functional still equals $3$ when evaluated on
the trivial $0$-current, and it always increases in value for
all other nontrivial $T$.
\end{rem}

% \begin{exm}
%   There is another example for 1-currents in $\mathbb{R}^2$. Suppose
%   $T_1,T_2\in \mathfrak{B}^2_{R,E}$ with length 1, and $T_1$ can be
%   seen as a vertical parallel shift of $T_2$ by 2 units. Then
%   $\mdn{T}_{\lambda,\mu}$ has to be trivial based on the
%   following observations, assuming $\lambda=\mu=1$,
% \begin{itemize}
% \item[(a)] When $\mdn{T}_{\lambda,\mu}=\emptyset$, 
% \begin{align*}
%   \sum_{i=1}^2 \mathbb{F}_\lambda (\mdn{T}_{\lambda,\mu}-T_i)+\mu
%   \mass(\mdn{T}_{\lambda,\mu})=2
% \end{align*}
% \item[(b)] When $\mdn{T}_{\lambda,\mu}\neq \emptyset$, then at
%   least one of $\sum_{i=1}^2 \mathbb{F}_\lambda
%   (\mdn{T}_{\lambda,\mu}-T_i)$ and $\mu
%   \mass(\mdn{T}_{\lambda,\mu})$ is greater or equal to 2.
% \end{itemize}
% Hence, $\mdn{T}_{\lambda.\mu}=\emptyset$.   
% \end{exm}

%\input{input_shared_boundaries}

\section{Shared Boundaries: Co-dimension 1 Results}
\label{sec:shared1}
%\subsection{Co-dimension 1 results}

\subsection{Point of View and Definitions}
\label{sec:definitions}

As we have just seen, the median need not be non-trivial for every
collection of integral currents as inputs.  Therefore, we now restrict
ourselves to input currents $\{T_i\}_{i=1}^\nCur$ which share
(non-empty) boundaries, and we seek medians over all currents $T$ that
share the same boundary.  This set up guarantees that $T-T_i$ is a
boundary for each $i$, and that there is a $\lambda$ small enough such
that the implicit minimization in each of the flat norm distances
$\F_{\lambda}(T-T_i)$ yields a minimal surface $S_{i,\lambda}$.  This
result follows from the intuitive observation that when $\lambda$ is
small enough, it is cheaper to ``fill in'' a boundary than pay for its
length (see Lemma 4.1 in our previous paper \cite{IbKrVi2016}). This
result could be understood fr follows from Thus we are left with the
problem of choosing a $T$ such that the sum of the volumes of the
minimal surfaces $S_{i,\lambda}$ (bound by $T-T_i$) is minimal.  Under
this setting, we obtain the particularly nice result of finding a
median $\mdn{T}$ such that the corresponding collection of minimal
surfaces $\{S_{i,\lambda}\}_{i=1}^\nCur$ is a stationary (under
Lipschitz maps) varifold with boundary $\{T_i\}_{i=1}^\nCur$.

In this section, we restrict our attention to the case in which all
the input currents $T_i$ are codimension-$1$ currents
($\dimcur$-dimensional currents in $\R^\dimsp$ for $\dimsp=\dimcur+1$)
that are themselves pieces of boundaries of multiplicity-$1$
$(\dimcur+1)$-dimensional currents.  Additionally, $\partial T_i
= \partial T_j$ for all $i$ and $j$, i.e., all the input currents have
the same, shared boundary.

\subsubsection{Definitions}
\label{ssec:envelope}

We begin by recalling the definition of top dimensional currents and
then define a special class of integral currents
(\cref{def-recCurrent,def-intCurrent}) we will use in this section.

\begin{defn}[Integral $(\dimcur+1)$-currents in
  $\R^{\dimcur+1}$] \label{def-codim1Current} Suppose $E \subset
  \R^{\dimcur+1}$ and $\mathcal{L}^{\dimcur+1}(E) < \infty$. We define
  the $(\dimcur+1)$-current $[[E]]$ to be the current $[[E]](\omega) =
  \int_E \omega(\vec{x}) d\mathcal{L}^{\dimcur+1}$ where
  $\vec{x}$ is the standard orienting $(\dimcur+1)$-vector in
  $\R^{\dimcur+1}$.  If we have a multiplicity function
  $\eta:\R^{\dimcur+1}\rightarrow\Z$, we define $\eta\myell[[E]]$ to
  be the current $\eta\myell[[E]](\omega) = \int_{E}
  \eta(x)\omega(\vec{x}) d\mathcal{L}^{\dimcur+1}$.  If $\mass(\partial
  \eta\myell[[E]]) <\infty$, then $\eta\myell[[E]]$ is a
  $(\dimcur+1)$-dimensional integral current.
\end{defn}

\begin{defn}[Sets of Finite Perimeter]
  $E\subset\R^{\dimcur+1}$ is a set of finite perimeter if $[[E]]$ is
  an integral current, i.e., if $\mass(\partial [[E]]) < \infty$.
\end{defn}

In section~\ref{sec:epsilon-envelope} we will use the reduced boundary. We need the idea of \emph{Approximate Normal}.
\begin{defn}[Approximate Normal]
  A set $E\subset\R^d$ is said to have an \emph{apprximate (outward) normal} $\vec{n}_x$, at a point $x\in\partial E$ if:
\[      \lim_{r\rightarrow 0} \frac{\mathcal{L}^d(B(x,r)\cap E\cap \{y\; | \; (y-x)\cdot \vec{n}_x>0 \})}{\alpha(d)r^d} \rightarrow 0 \]
and
\[      \lim_{r\rightarrow 0} \frac{\mathcal{L}^d(B(x,r)\cap E^c\cap \{y\; | \; (y-x)\cdot \vec{n}_x<0 \})}{\alpha(d)r^d} \rightarrow 0 \]
\end{defn}
\begin{defn}[Reduced Boundary]
  If $E\subset\R^d$ is a set of finite perimeter, then its reduced
  bounary $\partial^* E$ is the set of points $x\in\partial E$ where
  the approximate normals exist. 
\end{defn}

\begin{rem}
  The reduced boundary of $E$ and approximate normals are a part of
  the theory of sets of finite perimeter. These points are the points
  where, as we zoom in, except for a set with density $0$, $E$ looks
  like a half-space. The defining hyperplane of the half space is the
  measure-theoretic tangent plane of the set. See Chapter 5 of Evans
  and Gariepy~\cite{evans-1992-1} for all the details.
\end{rem}

\begin{defn}[$\Eu$] \label{def-E_U} Let $E\subset\R^{\dimcur+1}$ be a
  set of finite perimeter and $U\subset\R^{\dimcur+1}$ be a
  bounded open set such that $\mass(\partial(\partial [[E]] \myell U))<
  \infty$.  We
  define $\Eu\subset\mathcal{I}^{\dimcur}$ to be the collection of all
  integral $\dimcur$-currents $S$ such that
  \begin{enumerate}
    \item $S = \partial [[F]] \myell U$ for some set of finite perimeter $F$, and
    \item For some open $U'$ compactly supported in $U$,
      $U'\subset\subset U$, we have $E\setminus U' = F\setminus U'$.
  \end{enumerate}
  Note that this implies that $\partial(\partial [[F]] \myell U)
  = \partial(\partial [[E]] \myell U)$.
  See \cref{fig-Eu} for an illustration.
\end{defn}

\begin{figure}[htp!]
  \centering
  \includegraphics[scale=0.2]{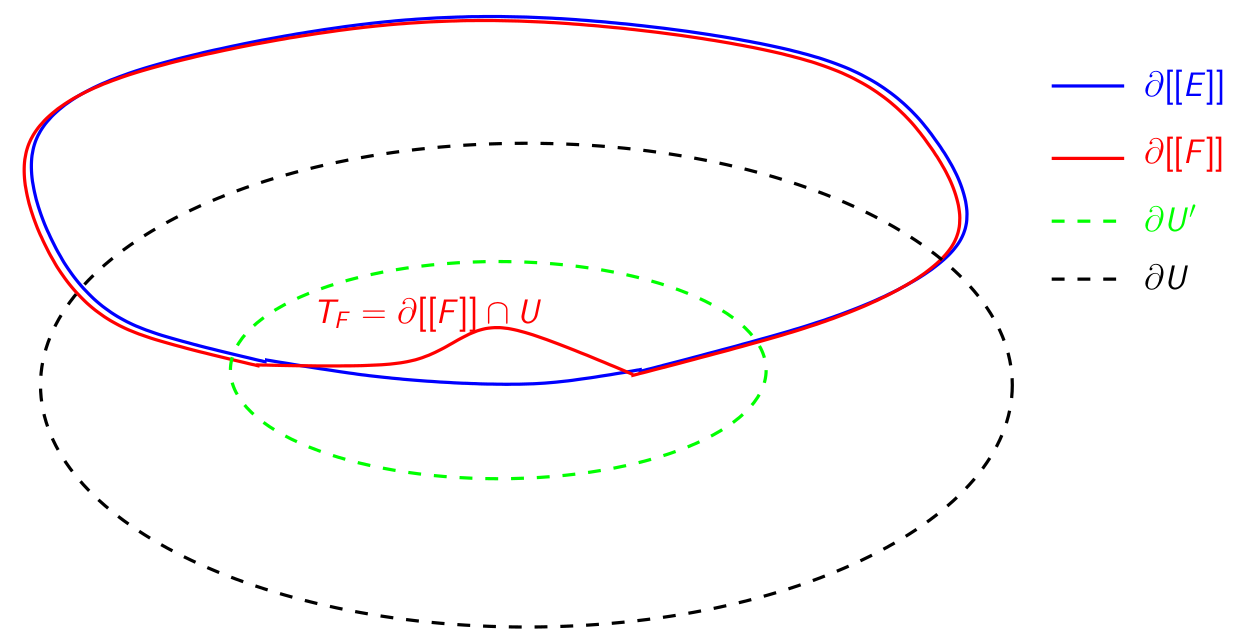}
  \caption{The reason for set $E_U$ is to guarantee there exists a cubical cover of the difference $[[F]] -[[E]]$ such that it is supported in some $U'\subset \subset U$, then we can apply compactness theorem. }
  \label{fig-Eu}
\end{figure}

\begin{rem}[Shared Boundaries]
  We say that a set of currents in $\{T_i\}_{i=1}^{\nCur}\subset\Eu$
  have shared boundaries when $\partial T_i = \partial T_j$ for all
  $i\neq j$. By design, every subset of currents in $\Eu$ has shared
  boundaries.
\end{rem}

\begin{defn}[Precise Representative of $f$. \cite{evans-1992-1}]\label{preciserepresentative}
Assume $f\in L_{loc}^1(\R^n)$, then 
\begin{align*}
f^*(x) = \left\{\begin{array}{ll}
\lim\limits_{r\rightarrow 0} \frac{1}{\alpha(n)r^n} \int_{B(x,r)} f(y) dy, &\quad \mbox{if this limit exist}\\
0, &\quad otherwise. 
\end{array}\right.
\end{align*}
\end{defn}

\begin{defn}[Precise Representative of a set $E$]\label{preciserepresentativeset}
Let $E\in \R^\dimsp$ be a bounded set with finite perimeter, and $f = \chi_E$. Define 
\begin{align*}
E^* = \{x| f^*(x) = 1\},
\end{align*}
to be the precise representative $E$. 
\end{defn}

\begin{rem}
Since Hausdorff measure is a Radon measure, by Lebesgue Besicovitch differentiation theorem, the limit defined in \ref{preciserepresentative} exists almost everywhere, i.e. $\Hd^\dimsp(E^*- E) = 0.$ Compared to $E$, $E^*$ removed the subset from $E$ that cannot be seen under measure $\Hd^\dimsp.$
\end{rem}

\begin{defn}[Envelope] \label{def-env} The envelope $\Env$ of a set of
  integral currents with shared boundaries,
  $\{T_i\}_{i=1}^{\nCur}\subset \Eu$, is defined as the union of $E^*_{i,j}$, $i < j$, such that
  $\partial\left(m\myell[[E_{i,j}]]\right) = T_i - T_j$, where
  $|m(x)|=1$ for all $x\in E$. $\Env$ is the
  union of all the precise representatives of regions that lie between any two of the input
  currents.
\end{defn}

\begin{rem}[Compact support]
  We note that for any finite collection of currents in $\Eu$,
  $\{T_i\}_{i=1}^{\nCur}\subset\Eu$, we have that $\Env\subset\subset U$. Moreover, $\partial [[E^*_{ij}]] = \partial [[E]]$ as $\Hd^\dimsp(E^*- E) = 0.$
\end{rem}

%\subsubsection{Subclass we will minimize over}
\paragraph{Subclass we will minimize over:}
\label{sec:class}

In this section, we always work with $p$-currents in $\Eu$, and in
particular, with sets of input currents
$\{T_i\}_{i=1}^{\nCur}\subset\Eu$.  We will
also assume that $\lambda$ is always small enough that the flat norm
decomposition implicit in $\F_\lambda(T-T_i)$ chooses an $S$ such that
$T-T_i = \partial S$.  Under this setting, we specialize the median
functional (introduced in \cref{eq:defmdn}) to the following one:
\begin{defn}[Median] \label{def-mdncd1shrdbdy}
  Let $\{T_i\}_{i=1}^\nCur \subset \Eu$.
  Then the median $\mdn{T}_\lambda$ is defined to be 
  \begin{align*}
    \mdn{T}_{\lambda} =\argmin_{T\in\Eu} \sum_{i=1}^\nCur \mathbb{F}_\lambda (T-T_i).
  \end{align*}
\end{defn}

\begin{rem}[ $\mdn{T}_\lambda\in \Eu$]
  We need to prove that the integral current we get in the
  existence theorem is in fact also in $\Eu$, but we will get this
  fairly easily using the compactness theorem for sets of finite
  perimeter.
\end{rem}

\subsubsection{Outline of the section}
\label{sec:outline-shared}

We begin by showing that the difference current between the
support of the median and the support of any input current, is a
subset of the envelope we defined above.  That is, if $T_i
= \partial[[E_i]]\myell U$ and $\hat{T} = \partial[[\hat{E}]]\myell U$
then $[[E_i]] - [[\hat{E}]]$ is supported in $\Env$.  Then,
using the deformation theorem, we show that medians exist. This turns
out to be a non-trivial result because there can indeed be minimizing
sequences with unbounded mass.  Next we demonstrate that for the case
we are considering in this section---the codimension $1$ case---nice,
smooth input currents can generate families of medians, all of which
are non-smooth.  Finally, we study the case of the mass-regularized
median (as defined in \cref{eq:defmsregmdn}), and show that the
difference set for this median lives in an $\epsilon$-neighborhood of
the envelope of the input currents and that $\epsilon \rightarrow 0$
as $\mu/ \lambda \rightarrow 0$.

\subsection{Medians Are In The Envelope}
\label{sec:in-envelope}

\begin{thm}[Medians are in the envelope] \label{thm:medianinenvelope}
  Let $\{T_i\}_{i=1}^\nCur \subset \mathcal{E}_U$.
  The support any median, $\mdn{T}_\lambda$, satisfies $\supp(\mdn{T}_\lambda - T_i)\subset Closure(\Env)$ and $$\mdn{T}_\lambda\myell \Closure(\Env)^c = T_i \myell \Closure(\Env)^c\ \forall i.$$.
\end{thm}

\begin{proof}
  It is obvious that $\mdn{T}_\lambda\myell \Closure(\Env)^c = T_i \myell \Closure(\Env)^c$ for all $i$ since all $T_i$'s agree outside $\Env^c.$ Now by way of contradiction, suppose $\supp(\mdn{T}_\lambda - T_i) \nsubseteq \Closure(\Env)$, then the Hausdorff distance between $\supp(\mdn{T}_\lambda - T_i)$ and $\Closure(\Env)$ is positive, i.e. $d_H(\mdn{T}_\lambda - T_i, \Env)=c>0$ for any $i$.  For any $i$, define $[[S_i]]$ to be the unique bounded integral current that spans $\mdn{T}_\lambda - T_i$.
  Note that because $\mdn{T}_\lambda - T_i$ is codimension $1$ and bounded, it divides the space into two components, one of which is bounded and the other unbounded.
  The bounded component is the unique minimal current spanning $\mdn{T}_\lambda - T_i$.
  In other words, $\partial [[S_i]] = \mdn{T}_\lambda - T_i$ and 
\begin{align*}
\partial ([[S_i]] - [[S_j]]) = \partial [[S_i]] - \partial [[S_j]] = T_j - T_i. 
\end{align*}
\noindent This implies $[[S_i]] - [[S_j]]$ spans $T_i - T_j$. Recall that in the definition of $\Env$, $\partial [[E^*_{ij}]] = T_i - T_j\subset \Env$. This tells us $[[S_i]]$ and $[[S_j]]$ agree outside $\Env$ almost everywhere, i.e. 
\begin{align*}
\Hd^{\dimcur+1}((S_i\backslash \Env) \triangle (S_j\backslash \Env)) = 0,\ \forall i,j.
\end{align*}
Define 
\begin{align*}
S'_i &= S_i \myell \Env,\\
S &= S_i\myell \Env^c,
\end{align*}
\noindent where the orientation of $[[S'_i]]$ and $[[S]]$ are induced by $[[S_i]]$. Notice that even though it is possible for $S_i\myell \Env^c \neq S_j\myell \Env^c$ on a set of $\Hd^{\dimcur+1}-$measure 0 for $i\neq j$, $[[S]]$ as a current for any $i,j$ will be the same. Define the new median to be $\mdn{T}'_\lambda = \mdn{T}_\lambda - \partial [[S]].$

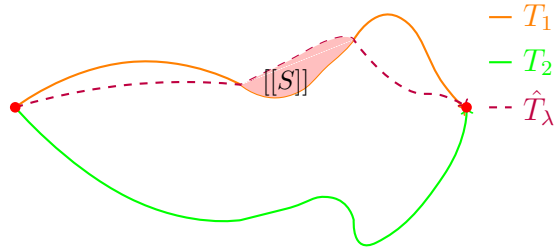
\begin{figure}[H]
\begin{center}
\begin{tikzpicture}
\def \r{3}
\draw[orange, thick, ->] plot [smooth, tension=1, ->] coordinates {(0,0) (0.5*\r,0.2*\r) (1*\r,0.1*\r)} -- plot [smooth, tension=1] coordinates {(1*\r,0.1*\r) (1.2*\r, 0.05*\r) (1.4*\r, 0.2*\r) (1.5*\r, 0.3*\r)} -- plot [smooth, tension=1] coordinates {(1.5*\r, 0.3*\r) (1.7*\r, 0.4*\r) (1.9*\r, 0.1*\r) (2*\r,0)};
\draw[green, thick, ->] plot [smooth, tension=1] coordinates {(0,0) (0.5*\r,-0.4*\r) (1*\r,-0.5*\r)} -- plot [smooth, tension=1] coordinates {(1*\r,-0.5*\r) (1.2*\r, -0.45*\r) (1.4*\r, -0.4*\r) (1.5*\r, -0.5*\r)} -- plot [smooth, tension=1] coordinates {(1.5*\r, -0.5*\r) (1.6*\r, -0.6*\r) (1.9*\r, -0.3*\r) (2*\r,0)};
\draw[purple, thick, dashed, ->] plot [smooth, tension=1] coordinates {(0,0) (0.5*\r,0.1*\r) (1*\r,0.1*\r)} -- plot [smooth, tension=1] coordinates {(1*\r,0.1*\r) (1.2*\r, 0.2*\r) (1.4*\r, 0.3*\r) (1.5*\r, 0.3*\r)} -- plot [smooth, tension=1] coordinates {(1.5*\r, 0.3*\r) (1.7*\r, 0.1*\r) (1.9*\r, 0.05*\r) (2*\r,0)};
\draw[thick,orange] (2.1*\r, 0.4*\r) -- (2.2*\r, 0.4*\r)node[right]{$T_1$};
\draw[thick,green] (2.1*\r, 0.2*\r) -- (2.2*\r, 0.2*\r)node[right]{$T_2$};
\draw[thick,purple, dashed] (2.1*\r, 0*\r) -- (2.2*\r, 0*\r)node[right]{$\hat{T}_\lambda$};
\fill[red] (0,0) circle[radius=2pt];
\fill[red] (2*\r,0) circle[radius=2pt];
%\node at (1.2*\r,0.11*\r) {$S$};
\fill[pink] plot [smooth, tension=1] coordinates {(1*\r,0.1*\r) (1.2*\r, 0.05*\r) (1.4*\r, 0.2*\r) (1.5*\r, 0.3*\r)} -- plot [smooth, tension=1] coordinates {(1*\r,0.1*\r) (1.2*\r, 0.2*\r) (1.4*\r, 0.3*\r) (1.5*\r, 0.3*\r)};
\node at (1.2*\r,0.11*\r) {\footnotesize{$[[S]]$}};
%\draw[thick, dashed] plot [smooth, tension=1] coordinates {(0,0) (0.5*\r,0.1*\r) (1*\r,0.1*\r)} -- plot [smooth, tension=1] coordinates {(1*\r,0.1*\r) (1.2*\r, 0.05*\r) (1.4*\r, 0.2*\r) (1.5*\r, 0.3*\r)} -- plot [smooth, tension=1] coordinates {(1.5*\r, 0.3*\r) (1.7*\r, 0.1*\r) (1.9*\r, 0.05*\r) (2*\r,0)};}
\end{tikzpicture}
\caption{The region outside the envelope is invariant with respect to input currents}
\end{center}
\end{figure}

\begin{figure}[H]
\begin{center}
\begin{tikzpicture}
\def \r{3}
\fill[pink] plot [smooth, tension=1] coordinates {(1*\r,0.1*\r) (1.2*\r, 0.05*\r) (1.4*\r, 0.2*\r) (1.5*\r, 0.3*\r)} -- plot [smooth, tension=1] coordinates {(1*\r,0.1*\r) (1.2*\r, 0.2*\r) (1.4*\r, 0.3*\r) (1.5*\r, 0.3*\r)};
\draw[orange, thick, ->] plot [smooth, tension=1, ->] coordinates {(0,0) (0.5*\r,0.2*\r) (1*\r,0.1*\r)} -- plot [smooth, tension=1] coordinates {(1*\r,0.1*\r) (1.2*\r, 0.05*\r) (1.4*\r, 0.2*\r) (1.5*\r, 0.3*\r)} -- plot [smooth, tension=1] coordinates {(1.5*\r, 0.3*\r) (1.7*\r, 0.4*\r) (1.9*\r, 0.1*\r) (2*\r,0)};
\draw[green, thick, ->] plot [smooth, tension=1] coordinates {(0,0) (0.5*\r,-0.4*\r) (1*\r,-0.5*\r)} -- plot [smooth, tension=1] coordinates {(1*\r,-0.5*\r) (1.2*\r, -0.45*\r) (1.4*\r, -0.4*\r) (1.5*\r, -0.5*\r)} -- plot [smooth, tension=1] coordinates {(1.5*\r, -0.5*\r) (1.6*\r, -0.6*\r) (1.9*\r, -0.3*\r) (2*\r,0)};
\draw[purple, thick, dashed, ->] plot [smooth, tension=1] coordinates {(0,0) (0.5*\r,0.1*\r) (1*\r,0.1*\r)} -- plot [smooth, tension=1] coordinates {(1*\r,0.1*\r) (1.2*\r, 0.2*\r) (1.4*\r, 0.3*\r) (1.5*\r, 0.3*\r)} -- plot [smooth, tension=1] coordinates {(1.5*\r, 0.3*\r) (1.7*\r, 0.1*\r) (1.9*\r, 0.05*\r) (2*\r,0)};
\draw[thick,orange] (2.1*\r, 0.4*\r) -- (2.2*\r, 0.4*\r)node[right]{$T_1$};
\draw[thick,green] (2.1*\r, 0.2*\r) -- (2.2*\r, 0.2*\r)node[right]{$T_2$};
\draw[thick,purple, dashed] (2.1*\r, 0*\r) -- (2.2*\r, 0*\r)node[right]{$\hat{T}_\lambda$};
\draw[thick, dashed] (2.1*\r, -0.2*\r) -- (2.2*\r, -0.2*\r)node[right]{$\hat{T}'_\lambda$};
\fill[red] (0,0) circle[radius=2pt];
\fill[red] (2*\r,0) circle[radius=2pt];
\node at (1.2*\r,0.11*\r) {\footnotesize{$[[S]]$}};
\draw[thick, dashed] plot [smooth, tension=1] coordinates {(0,0) (0.5*\r,0.1*\r) (1*\r,0.1*\r)} -- plot [smooth, tension=1] coordinates {(1*\r,0.1*\r) (1.2*\r, 0.05*\r) (1.4*\r, 0.2*\r) (1.5*\r, 0.3*\r)} -- plot [smooth, tension=1] coordinates {(1.5*\r, 0.3*\r) (1.7*\r, 0.1*\r) (1.9*\r, 0.05*\r) (2*\r,0)};
\end{tikzpicture}
\caption{Project $\hat{T}_\lambda$ out of the envelope back to the boundary of the envelope.}
\end{center}
\end{figure}
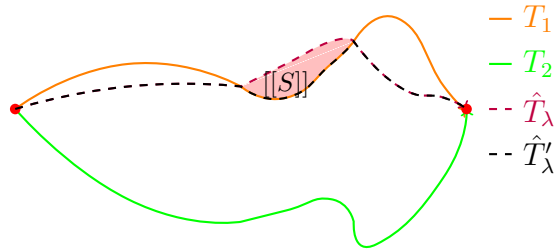

Then 
\begin{align*}
\F_\lambda(\mdn{T}_\lambda - T_i) &= \lambda \mass([[S_i]])\\
\F_\lambda(\mdn{T}'_\lambda - T_i) &= \lambda \mass([[S'_i]]),
\end{align*}

\noindent and $\mass([[S_i]]) - \mass([[S'_i]]) = \mass([[S]]) \geq 0$ for each $i$. Therefore 
\begin{align*}
\sum_{i=1}^\nCur \F_\lambda(\mdn{T}_\lambda - T_i) - \F_\lambda(\mdn{T}'_\lambda - T_i) = \nCur \lambda \mass([[S]]) > 0,
\end{align*}

\noindent which contradicts the fact of $\mdn{T}_\lambda$ being the median. So $\supp(\mdn{T}_\lambda - T_i) \subset \Closure(\Env).$

\end{proof}

\subsection{Medians Exist}
\label{sec:exist}

\begin{thm}[Medians exists] \label{thm-mdnexist}
  Let $\{T_i\}_{i=1}^\nCur \subset \Eu$, where $\Eu$ is specified in \cref{def-E_U}.
  Then $\mdn{T}_\lambda$ exists, and $\mdn{T}_\lambda \in \Eu$.
\end{thm}

\begin{proof}
The proof will be divided into the following steps:

\begin{enumerate}
\item \emph{Construct a sequence of cubical grids, $\{\grid\{\epsilon_s\}\}_{s=1}^\infty$, with side length $2\epsilon_s$ for each cube,  such that 
\begin{align*}
\Env \subset \grid(\epsilon_s) \subset U,
\end{align*}
where $\epsilon_s \rightarrow 0$ as $s\rightarrow \infty$ and $\bigcap_{s=1}^\infty \grid(\epsilon_s) = \Closure(\Env).$}\\

Since there are finite number of input currents, there exists a $U'\subset \subset U$ such that the difference of $T_i$'s only occurs in $U'$. Let $R = \hdist(U',U)$, where $\hdist$ is the Hausdorff distance. Define a sequence of cubical grid with side length $2\epsilon_s<R$, denoted as $\{grid(\epsilon_s)\}_{s=1}^\infty$, such that 
\begin{align*}
\Env \subset \grid(\epsilon_s) \subset U.
\end{align*}
Moreover each cube in $\grid(\epsilon_s)$ has nonempty intersection with $\Env$. Therefore
\begin{align*}
\lim\limits_{s\rightarrow \infty} \grid(\epsilon_s) = \bigcap_{s=1}^\infty \grid(\epsilon_s) = \Env. 
\end{align*}

By the definition of $\Env$, the differences between input currents lie within $\Env$, i.e., 
$$T_i\myell \Env^c = T_j\myell \Env^c, \forall i,j.$$ 

And $\Env \subset \grid(\epsilon_s)$, so $T_i$'s also agrees outside $\grid(\epsilon_s)$ for all $s$.

\item \emph{Push each $T_i$ to $\grid(\epsilon_s)$.}  \\

 Since all $T_i$'s agrees outside the $\Env$ and $\Env\subset \grid(\epsilon_s)$, we only need to push $T_i\myell \grid(\epsilon_s)$ to the grid.
 Hence we do not have to decide how $\partial T_i$ gets pushed.

By the deformation theorem \cite[Theorem 5.1]{morgan-2008-geometric}, each $T_i\myell \grid(\epsilon_s)$ can be decomposed into 
\begin{align*}
T_i\myell \grid(\epsilon_s) = T^{\pi}_i(\epsilon_s)\myell \grid(\epsilon_s) + \partial S^{\pi}_i(\epsilon_s)
\end{align*}
where $T^{\pi}_i(\epsilon_s)\myell \grid(\epsilon_s) \in \mathcal{P}_\dimcur \R^{\dimcur+1}$, the space of polyhedral $\dimcur$-currents in $\R^{\dimcur+1}$, and $S^{\pi}_i(\epsilon_s)\in \mathcal{I}_{\dimcur+1}\R^{\dimcur+1}$, the space of integral $(\dimcur+1)$-currents in $\R^{\dimcur+1}$.
In addition, 

\begin{align*}
\mass(T_i^{\pi}(\epsilon_k)\myell \grid(\epsilon_s)) &\leq \gamma \mass(T_i\myell \grid(\epsilon_s)),\\
T^{\pi}_i(\epsilon_s)\myell \grid^c(\epsilon_s) &= T_i\myell \grid^c(\epsilon_s),
\end{align*}
where $\gamma = 2(p+1)^{2p+2}$.

Define 

\begin{align*}
T^{\pi}_i(\epsilon_s) = T^{\pi}_i(\epsilon_s)\myell \grid(\epsilon_s) + T_i\myell \grid^c(\epsilon_s).
\end{align*}

As a consequence,  
\begin{align}
\begin{aligned}
\mass(T_i^{\pi}(\epsilon_s)) &= \mass(T^{\pi}_i\myell \grid(\epsilon_s)) + \mass(T^{\pi}_i\myell \grid^c(\epsilon_s))\\
&\leq \gamma \mass(T_i\myell \grid(\epsilon_s)) + \mass(T_i\myell \grid^c(\epsilon_s))\\
&\leq (\gamma + 1) \mass(T_i),~~~\mbox{ and }\\
\end{aligned}\\
\begin{aligned}
\F_\lambda (T_i - T_i^{\pi}(\epsilon_s))&= \F_\lambda(T_i \myell \grid(\epsilon_s) - T^{\pi}_i(\epsilon_s))\myell \grid(\epsilon_s)) \\
&= \F_\lambda (\partial S_i^{\pi}(\epsilon_s))\\
&\leq \epsilon_s\gamma \mass(T_i\myell \grid(\epsilon_s))\\
& \leq \epsilon_s\gamma \mass(T_i) \,. \label{pushedInputDiff}
\end{aligned}
\end{align}

%where $\gamma = 2(p+1)^{2p+2}$.

\item \emph{Construct pushed minimizing sequence for medians.} \\

Let $\{\mdn{T}_{\lambda,j}\}\subset \Eu$ be a minimizing sequence for the median objective function. Since all $T_i$'s agree outside $\grid(\epsilon_s)$, we can restrict $\{\mdn{T}_{\lambda,j}\}$ to satisfy
$$\mdn{T}_{\lambda,j}\myell \grid^c(\epsilon_s) = T_i\myell \grid^c(\epsilon_s),\ \forall i,j.$$

Next we first push each $\mdn{T}_{\lambda,j}$ to $grid(\epsilon_s)$, denoted as $\mdn{T}^{\pi}_{\lambda,j}(\epsilon_s)\myell \grid(\epsilon_s)$ and then extend it to $U$ as
\begin{align*}
\mdn{T}^{\pi}_{\lambda,j}(\epsilon_s) = \mdn{T}^{\pi}_{\lambda,j}(\epsilon_s)\myell \grid(\epsilon_s) + T_i\myell \grid^c(\epsilon_s).
\end{align*}

Note 
\begin{align*}
\mdn{T}_{\lambda,j}\myell \grid^c(\epsilon_s) = \mdn{T}^{\pi}_{\lambda,j}(\epsilon_s)\myell \grid^c(\epsilon_s) = T_i\myell \grid^c(\epsilon_s) = T_i^{\pi}(\epsilon_s)\myell \grid^c(\epsilon_s).
\end{align*}

In particular, we will pick $\epsilon_s = \frac{1}{2^s w_s}$, where $\omega_s = \mass^\dimcur(\mdn{T}_{\lambda,j}).$

\item \emph{Modify $\mdn{T}^{\pi}_{\lambda, j}(\epsilon_s)$ to $\mdn{T}^{adj}_{\lambda,j}$.} \label{step-modify}\\

After pushing everything to the grid, we can treat all $\{T_i^{\pi}(\epsilon_s)\}$ and $\{\mdn{T}^{\pi}_{\lambda,j}(\epsilon_s)\}$ as the boundaries of sets $\{E^{\pi}_i(\epsilon_s) \cap U\}$ and $\{\mdn{E}^{\pi}_{\lambda, j}(\epsilon_s) \cap U\}$, and the flat norm between $T^{\pi}_i(\epsilon_s)$ and $\mdn{T}^{\pi}_{\lambda,j}(\epsilon_s)$ is 
\begin{align}\label{pushedMedianDiff}
\begin{aligned}
\F_\lambda (\mdn{T}^{\pi}_{\lambda,j}(\epsilon_s) - T^{\pi}_i(\epsilon_s)) &= \Hd^{\dimcur +1}((\mdn{E}^{\pi}_{\lambda, j}(\epsilon_s)\cap U) \triangle (E^{\pi}_i(\epsilon_s)\cap U))\\
&= \Hd^{\dimcur+1}(\mbox{union of cubes in $(
\mdn{E}^{\pi}_{\lambda, j}(\epsilon_s)\cap U) \triangle (E^{\pi}_i(\epsilon_s)\cap U)$})\\
&= (2\epsilon_s)^{\dimcur + 1} (\mbox{union of cubes in $(
\mdn{E}^{\pi}_{\lambda, j}(\epsilon_s)\cap U) \triangle (E^{\pi}_i(\epsilon_s)\cap U)$}).
\end{aligned}
\end{align}

For each $\mdn{T}^{\pi}_{\lambda,j}(\epsilon_s)$, we can modify $\mdn{T}^{\pi}_{\lambda,j}(\epsilon_s)$ by adding cubes $\mathcal{C}$ to $\mdn{E}^{\pi}_{\lambda, j}(\epsilon_s)$ or subtracting cubes $\mathcal{C}$ from $\mdn{E}^{\pi}_{\lambda, j}(\epsilon_s)$ and replace the old $\mdn{T}^{\pi}_{\lambda,j}(\epsilon_s)$ with $\mdn{T}^{\pi}_{\lambda, j} + \sum_{C\in \mathcal{C}} \partial C$, denoted as $\mdn{T}^{adj}_{\lambda,j}(\epsilon_s)$, until it is the union of pieces from $\{T^{\pi}_i(\epsilon_s)\}$. 

Now in more detail: the intersections of the $E^{\pi}_i(\epsilon_s) \cap \grid(\epsilon_s)$ partition $\grid(\epsilon_s)$ into a finite number of components that sometimes share boundaries. For each component $\Comp_l(\epsilon_s)$, 

\begin{enumerate}
\item If $\Comp_l(\epsilon_s)\cap \mdn{E}^{\pi}_{\lambda, j}(\epsilon_s) = \emptyset$, do nothing;
\item If $\Comp_l(\epsilon_s)\cap \mdn{E}^{\pi}_{\lambda, j}(\epsilon_s) \neq \emptyset$, we will update $\mdn{T}^{\pi}_{\lambda,j}(\epsilon_s)$ in the following way: 

Define
\begin{itemize}
\item $\Comp_l(\epsilon_s) \cap \mdn{E}_{\lambda, j}(\epsilon_s) = F_l(\epsilon_s)$, 
\item $\Comp_l(\epsilon_s) \cap \mdn{E}^c_{\lambda, j}(\epsilon_s) = K_l(\epsilon_s).$
\end{itemize}
Note that $F_l(\epsilon_s)\cup K_l(\epsilon_s) = \Comp_l(\epsilon_s)$ and either $E^{\pi}_i(\epsilon_s) \cap \Comp_l(\epsilon_s) = \emptyset$ or $\# E^{\pi}_i(\epsilon_s) \cap \Comp_l(\epsilon_s) = \# \Comp_l(\epsilon_s)$. The second condition means if $\Comp_l(\epsilon_s)$ contains one of the cubes from $E^{\pi}_i(\epsilon_s)$, then all the cubes in $\Comp_l(\epsilon_s)$ are contained $E^{\pi}_i(\epsilon_s)$. As a result, 
\begin{align*}
\mdn{E}^{\pi}_{\lambda,j} \triangle E^{\pi}_i(\epsilon_s) = F_l(\epsilon_s)\ or\ K_l(\epsilon_s).
\end{align*}

Now for each cube $C$ in $F_l(\epsilon_s)$ or $K_l(\epsilon_s)$, denote 
\begin{enumerate}
\item $N_C^{F_l(\epsilon_s)} = \# \{E^{\pi}_i(\epsilon)| C\in E^{\pi}_i(\epsilon)\}$ if $C\in F_l(\epsilon_s)$,
\item $N_C^{K_l(\epsilon_s)} = \# \{E^{\pi}_i(\epsilon)| C\in E^{\pi}_i(\epsilon)\}$ if $C\in K_l(\epsilon_s)$.
\end{enumerate} 

There are two cases:
\begin{enumerate}
\item If $\sum_{C\in F_l(\epsilon_s)}N_C^{F_l(\epsilon_s)}\geq \sum_{C\in K_l(\epsilon_s)}N_C^{K_l(\epsilon_s)}$, then subtracting $\Comp_l(\epsilon_s)$ will decrease the sum of flat norms between $\mdn{T}_{\lambda, j}(\epsilon_s)$ and $T^{\pi}_i(\epsilon_s)$'s by $(2\epsilon_s)^{\dimcur+1}\sum_{C\in F_l(\epsilon_s)}N_C^{F_l(\epsilon_s)}$ and increase the sum by $(2\epsilon_s)^{\dimcur+1}\sum_{C\in K_l(\epsilon_s)}N_C^{K_l(\epsilon_s)}$. Therefore, the sum of flat norms will decrease by $(2\epsilon_s)^{\dimcur+1}(\sum_{C\in F_l(\epsilon_s)} N_C^{F_l(\epsilon_s)} - \sum_{C\in K_l(\epsilon_s)}N_C^{K_l(\epsilon_s)})$. So $\mdn{E}^{adj}_{\lambda, j}(\epsilon_s) = \mdn{E}^{\pi}_{\lambda, j}(\epsilon_s)\backslash \Comp_l(\epsilon_s)$ and $\mdn{T}^{adj}_{\lambda, j}(\epsilon_s) = \mdn{E}_{\lambda, j}(\epsilon_s) - \sum_{C\in \mdn{E}_{\lambda, j}(\epsilon_s)\backslash \Comp_l(\epsilon_s)}$ $\partial C$,
\item If $\sum_{C\in F_l(\epsilon_s)} N_C^{F_l(\epsilon_s)}< \sum_{C\in K_l(\epsilon_s)}N_C^{K_l(\epsilon_s)}$, then adding $\Comp_l(\epsilon_s)$ will decrease the sum of flat norms between $\mdn{T}_{\lambda, j}(\epsilon_s)$ and $T^{\pi}_i(\epsilon_s)$'s by $(2\epsilon_s)^{\dimcur+1}\sum_{C\in K_l(\epsilon_s)}N_C^{K_l(\epsilon_s)}$ and increase the sum by $(2\epsilon_s)^{\dimcur+1}\sum_{C\in F_l(\epsilon_s)}N_C^{F_l(\epsilon_s)}$. Therefore, the sum of flat norms will decrease by $(2\epsilon_s)^{\dimcur+1}(\sum_{C\in K_l(\epsilon_s)}N_C^{K_l(\epsilon_s)} - \sum_{C\in F_l(\epsilon_s)}N_C^{F_l(\epsilon_s)})$. So $\mdn{E}^{adj}_{\lambda, j}(\epsilon_s)=\mdn{E}^{\pi}_{\lambda, j}(\epsilon_s)\cup \Comp_l(\epsilon_s)$ and $\mdn{T}^{adj}_{\lambda, j}(\epsilon_s) = \mdn{E}^{\pi}_{\lambda, j}(\epsilon_s) + \sum_{C\in \mdn{E}_{\lambda, j}(\epsilon_s)\backslash \Comp_l(\epsilon_s)} \partial C$.
\end{enumerate}
The process will end in finite steps since there are only finite number of $\Comp_l(\epsilon_s)$'s. And when it finishes, $\mdn{E}_{\lambda, j}(\epsilon_s)\myell \grid(\epsilon_s)$ will be the union of pieces from $T^{\pi}_i(\epsilon_s)\myell \grid(\epsilon_s)$.
\end{enumerate}

\begin{figure}[H]
  \begin{center}
    \scalebox{0.3}{
      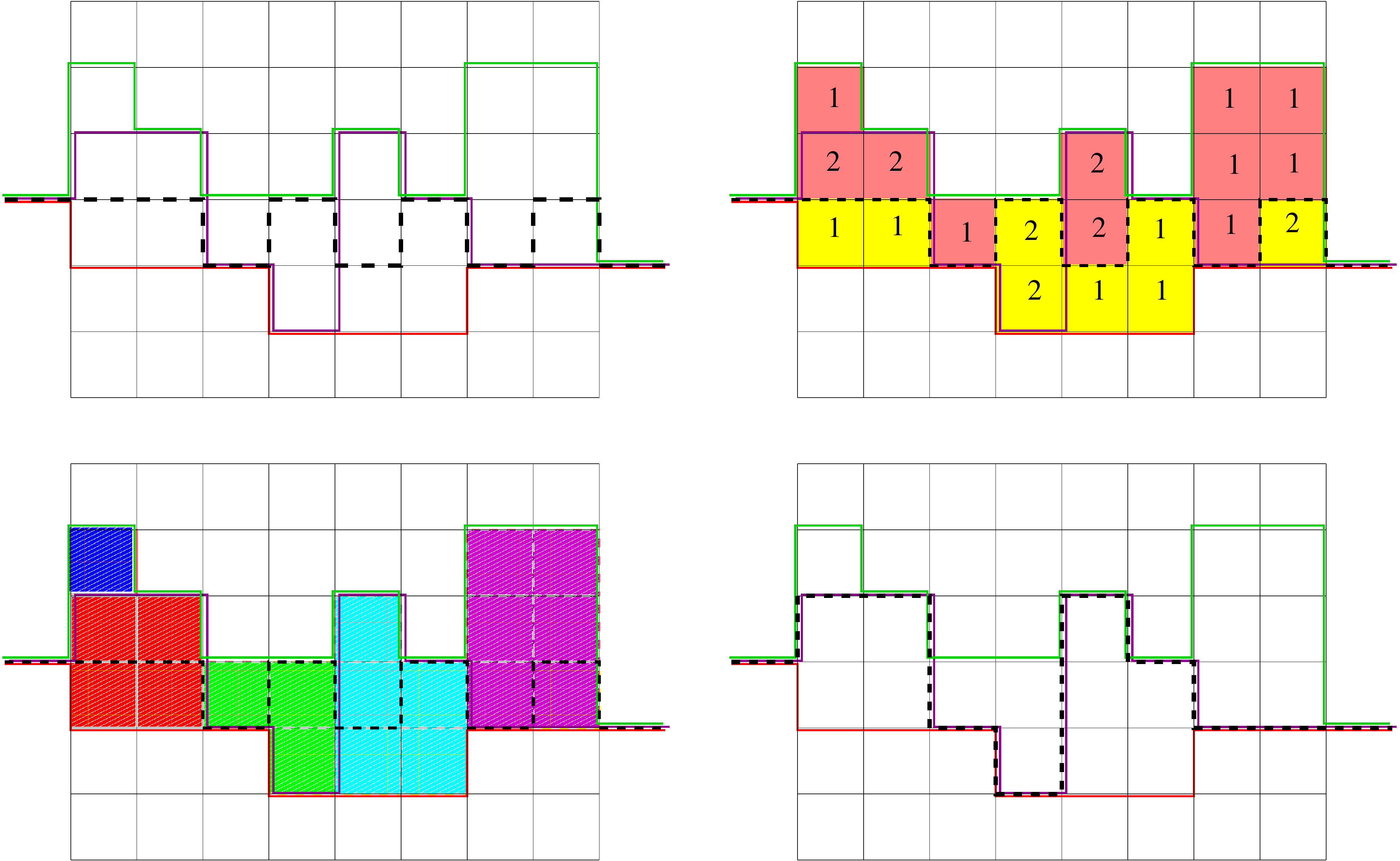}\\
%    \scalebox{0.3}{
%      \input{figures/pushedMedian1.pdf_t}}
    \caption{An example in the case of $\dimcur+1 =2$ of how to adjust the pushed median. }
    \label{fig:pushedMedian}
  \end{center}
\end{figure}

In the top-left picture of \cref{fig:pushedMedian}, there are 3 pushed input currents represented as solid green, red and purple lines. The black dashed line is the original pushed median $\mdn{T}_{\lambda,j}(\epsilon_s)$. In the top-right picture, pink regions represent the regions outside $\mdn{E}^{\pi}_{\lambda, j}(\epsilon_s)$ while yellow regions are the opposite. The number in each cube $C$ equals $\# \{E^{\pi}_i(\epsilon)| C\in E^{\pi}_i(\epsilon)\}$. In the bottom-left picture, different color represents different connected components. For the blue component, it does not intersects with $\mdn{E}^{\pi}_{\lambda,j}(\epsilon_s)$, we leave it alone. For the red component, $\sum_{C\in F_l(\epsilon_s)}N_C^{F_l(\epsilon_s)}< \sum_{C\in K_l(\epsilon_s)}N_C^{K_l(\epsilon_s)}$, so we added the entire yellow component to  $\mdn{E}^{\pi}_{\lambda, j}(\epsilon_s)$. For the green component, $\sum_{C\in F_l(\epsilon_s)}N_C^{F_l(\epsilon_s)}\geq \sum_{C\in K_l(\epsilon_s)}N_C^{K_l(\epsilon_s)}$, so we subtract the green component from $\mdn{E}^{\pi}_{\lambda, j}(\epsilon_s)$. We can continue the same process to cyan and purple components. In bottom-right picture, the black dashed line is the updated pushed median.

\item \emph{$\mass(\mdn{T}^{adj}_{\lambda,j}(\epsilon_s))$ is bounded uniformly.}\\

Each $\mdn{T}^{adj}_{\lambda,j}(\epsilon_s)\myell \grid(\epsilon_s)$ is the union of pieces from $T_i^{\pi}(\epsilon_s)\myell \grid(\epsilon_s)$ and $\mdn{T}^{adj}_{\lambda,j}(\epsilon_s)\myell \grid^c(\epsilon_s) = T_i^{\pi}(\epsilon_s)\myell \grid^c(\epsilon_s)$, so
\begin{align*}
\mass(\mdn{T}^{adj}_{\lambda,j}(\epsilon_s)) &\leq \sum_{i=1}^\nCur \mass(T_i^{\pi}(\epsilon_s))\leq \sum_{i=1}^\nCur (\gamma + 1) \mass(T_i),
\end{align*}

and $\mdn{T}^{adj}_{\lambda}(\epsilon_s)\in \Eu$. $\{\mdn{T}^{adj}_{\lambda,j}(\epsilon_s)\} \subset U.$

\item \emph{Apply triangle inequality and prove that $\mdn{T}^{adj}_{\lambda}(\epsilon_s)$ converges to the median $\mdn{T}_\lambda$ as $s\rightarrow \infty$.} \\

By diagonal argument, the sequence $\{\mdn{T}^{adj}_{\lambda, s}(\epsilon_s)\}$ converges to some $\mdn{T}_{\lambda}.$ 

Note that 
\begin{align}
\F_\lambda (\mdn{T}^{\pi}_{\lambda,s} - \mdn{T}_{\lambda,s}) &\leq \epsilon_s \gamma \mass(\mdn{T}_{\lambda,s}) \leq \frac{\gamma}{2^s}, \label{eq-TriIneq1} \\
\sum_{i=1}^\nCur \F_\lambda (\mdn{T}^{adj}_{\lambda,s}(\epsilon_s) - T_i) &\leq \sum_{i=1}^\nCur \F_\lambda (\mdn{T}^{\pi}_{\lambda,s}(\epsilon_s) - T_i). \label{eq-TriIneq2} 
\end{align}

This inequality follows from the actions described in Step \ref{step-modify} for the construction of $\mdn{T}^{adj}_{\lambda,s}(\epsilon_s)$, where the adjustment process decreases the sum of flat norms between all $T_i$'s.

Using the triangle inequality with the bounds in \cref{eq-TriIneq1,eq-TriIneq2} we get

\begin{align*}
\lim_{s\rightarrow \infty}\sum_{i=1}^\nCur \F_\lambda (\mdn{T}_{\lambda,s} - T_i) &\leq \sum_{i=1}^\nCur \F_\lambda (\mdn{T}_\lambda - T_i) \\
&\leq \lim_{s\rightarrow \infty}\sum_{i=1}^\nCur \F_\lambda (\mdn{T}^{adj}_{\lambda,s}(\epsilon_s) - T_i) \\
&\leq \lim_{s\rightarrow \infty}\sum_{i=1}^\nCur [(\F_\lambda (\mdn{T}^{adj}_{\lambda,s}(\epsilon_s) - T^{\pi}_i(\epsilon_s)) + (\F_\lambda (T^{\pi}_i(\epsilon_s) - T_i))]\\
&\leq \lim_{s\rightarrow \infty}\sum_{i=1}^\nCur [(\F_\lambda (\mdn{T}^{\pi}_{\lambda,s}(\epsilon_s) - T^{\pi}_i(\epsilon_s)) + (\F_\lambda (T^{\pi}_i(\epsilon_s) - T_i))]\\
&\leq \lim_{s\rightarrow \infty}\sum_{i=1}^\nCur [\F_\lambda(\mdn{T}^{\pi}_{\lambda,s}(\epsilon_s) - \mdn{T}_{\lambda,s})+ \F_\lambda(\mdn{T}_{\lambda,s} - T_i) + 2(\F_\lambda (T^{\pi}_i(\epsilon_s) - T_i))]\\
&\leq \lim_{s\rightarrow \infty} \left[\sum_{i=1}^\nCur \F_\lambda (\mdn{T}_{\lambda,s} - T_i) + \frac{\gamma\nCur}{2^s} + \epsilon_s \gamma \sum_{i=1}^\nCur \mass(T_i)\right]\\
&=\lim_{s\rightarrow \infty}\sum_{i=1}^\nCur \F_\lambda (\mdn{T}_{\lambda,s} - T_i) \, .
\end{align*}

\noindent Therefore $\mdn{T}_\lambda$ is a median and by step 5, $\mass(\mdn{T}_\lambda) \leq \sum_{i=1}^\nCur (\gamma+1) \mass(T_i)$ and $\mdn{T}_\lambda\in \Eu$.
\end{enumerate}
\end{proof}

\subsection{Medians Can Be Wild}
\label{sec:wild}

As we proved in \cref{sec:exist}, the median $\mdn{T}_\lambda$ for
$\{T_i\}_{i=1}^\nCur \subset \mathcal{E}_U$ always exists with a mass
bounded by $\sum_{i=1}^\nCur (\gamma+1) \mass(T_i)$.  However, it is not
guaranteed that all the medians are bounded.  In fact, there exist
sequences of medians whose masses diverge.  For example, take two
input currents $T_1$ and $T_2$ to be the upper and lower half of the
boundary of a rectangle. The median $\mdn{T}_\lambda$ can be any non
self-intersecting curve of finite length inside the square. We can,
for example use any graph that represents a random walk in the vertical
direction vesus time, represented by the horizontal axis, under the
constriant that the walk must stay in the rectangle. Of course the
lengths (i.e. mass) of these random walks are not bounded since the
speed of the walk (the slope of the graph) is not bounded.

\begin{figure}[H]
\begin{center}
\includegraphics[scale=0.5]{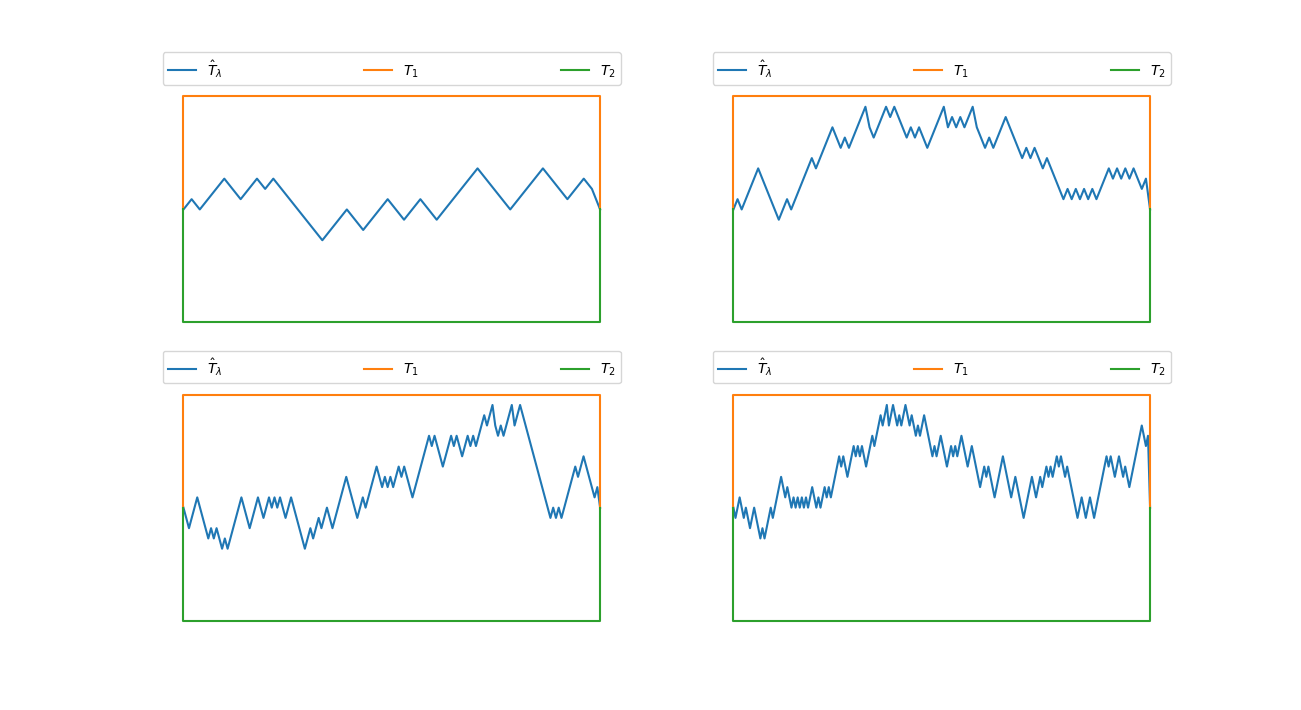}
\caption{Random walk medians can have arbitrarily high mass.}
\end{center}
\end{figure}

\subsection{Smooth Inputs Can Generate Non-smooth Medians}

Even if the input currents are regular, the median need not be
regular.  We present an example in $\R^2$ showing the median can fail
to be regular.  We will be looking for medians which are pieces of
boundaries of sets, as we did in the proof of existence for the
codimension-$1$ shared boundary case above.

\begin{thm}(Regularity of inputs does not imply regularity of
  median) \label{thm-nonregmdn} Suppose that each of the $T_i$'s are
  smooth, with shared boundaries, and that we minimize over $T$ that
  are pieces of boundaries of sets.  Then the entire set of medians
  might consist only of currents that lack smoothness somewhere.
\end{thm}

\begin{proof}
  Consider the case of the input $T_i$'s being oriented graphs of
  smooth functions $f_i$, where the \{set, orienting vector field\} pairs are given by:
  \[ \left(\{(x,f_i(x))| x\in [0,1]\},\frac{(-1,-f_i'(x))}{\sqrt{1 +
        (f_i'(x))^2}}\right)\] and each $f_i$ satisfies $f_i(0) = 1 =
  f_i(1)$. An example is shown in Figure~\ref{fig:non-smooth-median}.

  The next lemma is more than we actually need, but is included
  because part of its proof anticipates a later proof.
  \begin{lemma}(Graphs are Good) When the $T_i$'s are one-dimensional
    graphs in $\R^2$ sharing the same two boundary points, the infimum
    of the median objective functional over piece-wise smooth
    non-graphs is not smaller than the infimum over graphs.  Hence the
    infimum in the median problem can be restricted to graphs.
    \label{lem:graphs-good}
  \end{lemma}

  \begin{figure}[htp!]
    \centering
    \scalebox{0.5}{
    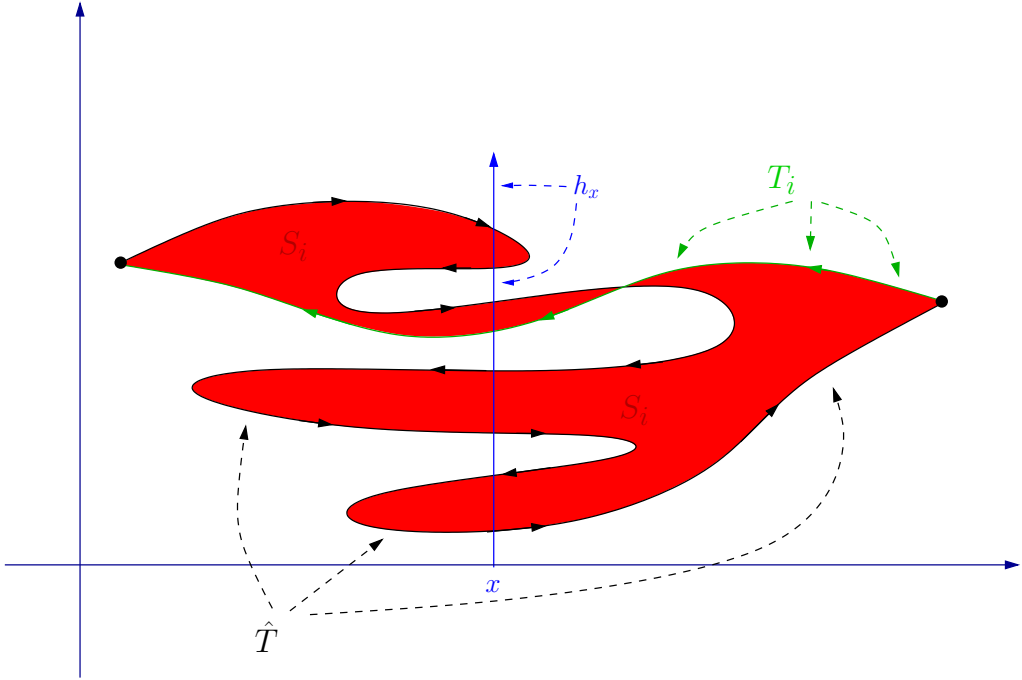}
    \caption{Slicing $T_i$, $\mdn{T}$ and $S_i$ with a vertical line, $h_x$.}
    \label{fig:slicing-h-x}
  \end{figure}

  \begin{figure}[htp!]
    \centering
    \scalebox{0.7}{
    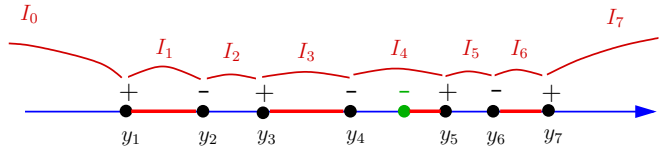}
  \caption{Structure of the slice. The intervals generated by the
    intersection of $h_x$ and the $S_i$ are shown in red.}
    \label{fig:slicing-s-i}
  \end{figure}

  \begin{figure}[htp!]
    \centering
    \scalebox{0.7}{
    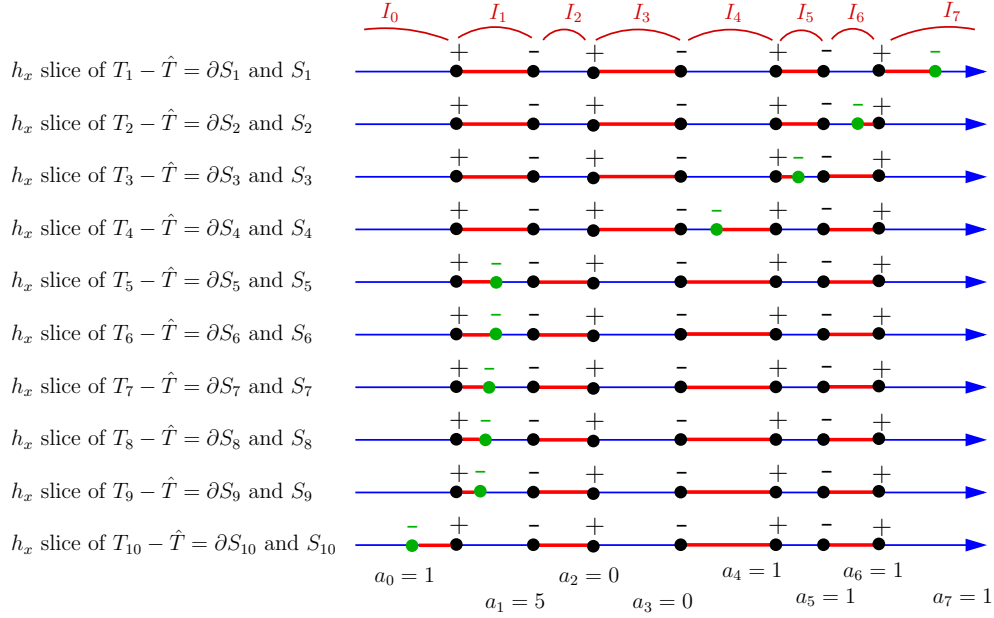}
    \caption{All the slices and intervals they generate.}
    \label{fig:slicing-all-ses}
  \end{figure}

  \begin{proof}[Proof of Lemma~\ref{lem:graphs-good}] {\bf Outline:}
    We slice $T_i$, $\mdn{T}$, and the $2$-dimensional current $S_i$
    bounded by $T_i$ and $\mdn{T}$ vertically to get positive and
    negative oriented $0$-currents and oriented intervals. Sum of 
    the integrals of the lengths of those intervals over the $x$-axis
    equals the median objective function for $\mdn{T}$:
    \[ \sum_{i=1}^{\nCur}\Bbb{F}_\lambda (T_i-\mdn{T}) = \int \left( \sum_{i=1}^{\nCur} \Hd^1(h_x\cap
    S_i) \right)dx\] where we are interpreting $S_i$ as the set we integrate over
    to get the current $S_i$.  (We can also express this as
    $\sum_{i=1}^{\nCur}\Bbb{F}_\lambda (T_i-\mdn{T}) =  \sum_{i=1}^{\nCur} \mass(\hat{h}_x(S_i))$
    where $\hat{h}_x(S_i)$ is the slice of the current $S_i$ (which is
    itself a current) by the line $h_x$). The strategy of the proof
    shows that every sum equals or exceeds the sum generated by a
    graph that stays in the median interval of each slice. (Recall
    that in 1 dimension, the set of medians is either a point or an
    interval. If we minimize the median objective function over
    graphs, we get that the graph has to live in the median interval
    generated by the slicing of the $T_i$.) We see that if every slice
    of the median generates points not in the median interval of that
    slice, the cost exceeds the minimal cost.  Since the minimal cost
    is attained by any graph that stays in the median intervals, and
    any such interval is forced to stay in a cone with a kink, we are
    done.

    {\bf Now, the details:} 

    \begin{enumerate}
    \item We assume that the median intersects vertical slices
      transversely almost everywhere.
      \begin{enumerate}
      \item This can be shown by assuming not -- that the
      measure of $E$, defined to be the $x$'s such that $h_x$
      intersects $\mdn{T}$ tangentially at some point, has positive
      measure. Choose any (big) $C > 0$.
    \item We can cover $E$ with intervals
      $F_w$, $w\in E$ which are in one to one correspondence with
      disjoint pieces of $\mdn{T}$, $G_w$, such that $\Hd^1(G_w) > C
      \Hd^1(F_w)$.
    \item This implies that the length of $\mdn{T}$ is
      unbounded.
    \item Note: We use the fact that $\mdn{T}$ is piecewise smooth,
      so the projection of the points on $\mdn{T}$ where it is not
      smooth has measure zero on the $x$ axis.
      \end{enumerate}
    \item We assume that the $T_i$ all intersect the vertical slices transversely.
    \item Since the mass of $T_i$ and $\mdn{T}$ is finite, the slice
      $h_x$ generates a finite number of intersections with $\mdn{T}$
      for almost every $x$.
    \item As a result, for almost every $x$, $\Hd^0(h_x \cap \mdn{T})
      = 2m(x)+1$, each intersection with multiplicity and orientation of
      either $+1$ or $-1$. We will denote the intersection points
      $y_i$ $i=1,...,2m(x)+1$ This is shown in
      Figures~\ref{fig:slicing-h-x} and \ref{fig:slicing-s-i}.
    \item Note: we will abbreviate $m(x)$ to $m$ from this point on in
      this proof.
    \item The $2m+1$ points generate the $2m+2$ intervals $I_0 =
      (-\infty,y_0)$, $I_1= [y_1,y_2)$, $I_2 = [y_2,y_3)$,
      ..., $I_{2m}=[y_{2m},y_{2m+1})$, $I_{2m+1}= [y_{2m+1},\infty)$.
    \item Define $a_j \equiv|\{i | h_x\cap T_i\in I_j|$.
    \item Figure~\ref{fig:slicing-all-ses} shows all slices generated
      at one x for a case in which there are $10$ input currents
      $T_1$,...,$T_{10}$.
    \item Observe that there are two types of intervals: those with a
      positive left endpoint - $I_1$,$I_3$,...,$I_{2m-1}$, $I_{2m+1}$
      -- and those with a positive right endpoint --
      $I_0$,$I_2$,...,$I_{2m-2}$, $I_{2m}$.
    \item Observe that, when $j\geq 2$ is even, then all the slices
      generate $\sum_{i=0}^{j-1} a_i$ red intervals and $a_j$ partial
      intervals from the right endpoint to the intersections of the
      $T_i$ and $I_j$. Likewise, one can see that when $j < 2m+1$ is
      odd, all the slices generate $\sum_{i=j+1}^{2m+1} a_i$ red
      intervals and $a_j$ partial intervals from the $T_i$ generated
      points in $I_j$ and the positive endpoint of $I_j$.
    \item In order that the intervals generated by the intersection of
      $h_x$ with all the $S_i$'s provide separate paths to every
      (green) intersection point and some fixed positive point in $y_k
      \in\{y_1, y_3, ... y_{2m+1}\}$, four conditions must be met.  We
      must have that in each interval $I_l$, l even and odd, to the
      left and right of $y_k$ a sufficient number of full red
      intervals to create separate paths to the green $T_i$
      intersections that are in that interval (if their connecting
      partial interval connects in the wrong direction) or those
      further away. The four conditions thus generated
      are:
      \begin{enumerate}
      \item For $l = k,k+2,...$ we must have 
            \[\sum_{l+1}^{2m+1} a_i \geq \sum_{l+1}^{2m+1} a_i\]
      \item for $l = k-2,k-4,...$ we must have
            \[ \sum_{l+1}^{2m+1} a_i \geq \sum_{i=0}^{l} a_i\]
      \item for $l= k+1,k+3,...$ we must have 
            \[ \sum_{i=0}^{l-1} a_i \geq \sum_{i=l}^{2m+1} a_i\]
      \item for $l= k-1,k-3,...$ we must have 
            \[ \sum_{i=0}^{l-1} a_i \geq \sum_{i=0}^{l-1} a_i\]
      \end{enumerate}
    \item This reduces to finding $k\in\{1,3,5,...,2m-1,2m+1\}$ such that:
      \begin{enumerate}
      \item in the case $k = 1$:
      \begin{eqnarray*}
        \sum_{i=0}^{k} a_i &\geq&   \sum_{i=k+1}^{2m+1} a_i \\ 
      \end{eqnarray*}
      \item in the case that $k\in\{3,5,...,2m-1\}$:
      \begin{eqnarray*}
        \sum_{i=0}^{k} a_i     &\geq&   \sum_{i=k+1}^{2m+1} a_i \\ 
        \sum_{i=k-1}^{2m+1} a_i &\geq&   \sum_{i=0}^{k-2} a_i 
      \end{eqnarray*}
      \item in the case that $k = 2m+1$:
      \begin{eqnarray*}
        \sum_{i=k-1}^{2m+1} a_i &\geq&   \sum_{i=0}^{k-2} a_i 
      \end{eqnarray*}
      \end{enumerate}
    \item Since it is clear that such a $k$ always exists, we have
      that the cost of piece-wise smooth non-graph always equals or
      exceeds the cost of a graph.
    \item But we have even more: denote the median interval, generated
      by the $\nCur$ intersections $h_x\cap T_i$ on each vertical
      line, by $h_x^m$. Define the \emph{median interval envelope} to
      be the union $\cup_x (x,h_x^m) \subset \R^2$. Our proof implies
      that if, for some $x$, all positive intersections of $\mdn{T}$
      with $h_x$ occur outside the closed median interval on $h_x$,
      the cost of $\mdn{T}$ is strictly greater than a graph that
      lives in $\cup_x (x,h_x^m)$, which is impossible and so we
      conclude that there must be an intersection of {\bf any}
      piece-wise smooth median with each $h_x$ in the median interval
      on $h_x$.
    \end{enumerate}
    \vspace*{-0.2in}
  \end{proof}
  \noindent {\bfseries Back to Proof of \cref{thm-nonregmdn}:}
  Given the median $\mdn{T}$ must be in
  every median interval, it is immediate that, in the example shown in
  Figure~\ref{fig:non-smooth-median}, the medians cannot be differentiable at $p$
  because there is a kink in the median
  interval envelope at $p$.
\end{proof}
  \begin{rem}
    Since we can write the median functional as an integral over
    $1$-dimensional slices, the same proof generalizes easily to any
    codimension $1$ case with smooth input currents with shared boundary
    but non-smooth median interval envelope.
  \end{rem}

  \begin{figure}[ht!]
    \centering
    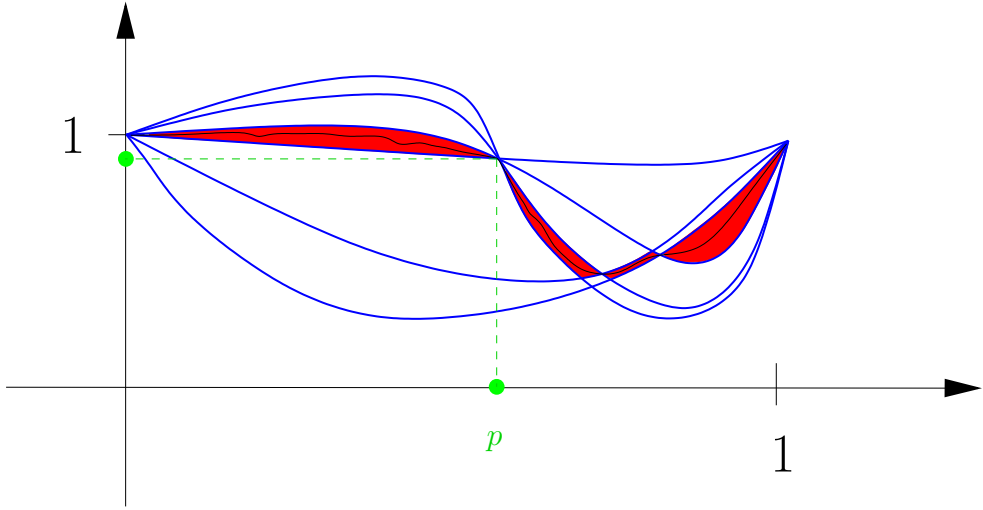    
    \caption{Here is an example of a case in which smooth inputs do
      not imply smooth medians: every median is non-smooth!}
    \label{fig:non-smooth-median}
  \end{figure}

\subsection{Regularized Medians are in the $\epsilon$-Envelope}
\label{sec:epsilon-envelope}
Recall the definition of mass-regularized median given in \cref{eq:defmsregmdn}.
  We specialize the definition to $\mathcal{E}_U$ here, which we specified in \cref{def-E_U}.
  Also recall the definition of \emph{envelope} of a set of input currents (\cref{def-env}).
\begin{defn}[Mass regularized median]
  Let $\{T_i\}_{i=1}^\nCur \subset \mathcal{E}_U$ and $\{T_i\}_{i=1}^\nCur$ agrees outside some $U'\subset U.$ 
  Then the mass regularized median $\mdn{T}_{\lambda,\mu}$ is defined to be 
  \begin{align}\label{eqn-massregularizedmedian}
    \mdn{T}_{\lambda,\mu} = \arg\min_{T} \sum_{i=1}^\nCur \F_\lambda (T-T_i) + \mu \mass(T), \mu > 0.
  \end{align} 
  where we minimize $T$ over $\{T_i \myell (\Env)^c + P \,:\, P\in \mathcal{I}^\dimcur, \supp(P) \subset U', \partial (T_i\myell (\Env)^c)$ $=\partial P\}$.
\end{defn}
\begin{thm}[Mass regularized medians are in closed ${\bf Env}^{\epsilon} (\{T_i\})_{i=1}^\nCur$] \label{thm-mdninenv}
  Let $\{T_i\}_{i=1}^\nCur \subset \mathcal{E}_U$. 
  Then $\supp(P)\subset {\bf Env}^{\epsilon} (\{T_i\})_{i=1}^\nCur$ for some $\epsilon$, where ${\bf Env}^{\epsilon} (\{T_i\})_{i=1}^\nCur$ is the closed $\epsilon$-extension of ${\bf Env}(\{T_i\}_{i=1}^\nCur)$.
  Further, $\epsilon \rightarrow 0$ as $\mu/\lambda \rightarrow 0$. 
\end{thm}
\begin{proof}
We divide the proof into steps: 
\begin{enumerate}
\item \emph{There exists an $\epsilon$ such that $\supp(P) \subset {\bf Env}^{\epsilon} (\{T_i\})_{i=1}^\nCur$.}\\

Take the convex hull of $\Env$, then $\supp(P)$ has to stay inside this convex hull; otherwise, we can use the same argument as in \cref{thm:medianinenvelope} to reach a contradiction. Since $\Env\subset U$ and $U$ is bounded, then convex hull of $\Env$ is also bounded. Define 
$$R = d(\Env, \mbox{\bf{con}}(\Env)),$$
where $d$ is the Hausdorff distance between two sets. Then $\supp(P)\subset {\bf Env}^{R} (\{T_i\})_{i=1}^\nCur$. This implies there exists some $\epsilon \leq R$ such that $\supp(P) \subset {\bf Env}^{\epsilon} (\{T_i\})_{i=1}^\nCur$. In fact, in the following step, we will prove that the smallest $\epsilon$ defining containing ${\bf Env}^{\epsilon} (\{T_i\})_{i=1}^\nCur$ goes to 0 as $\mu/\lambda \rightarrow 0$.
\item \emph{There exists an $\epsilon$ such that $\supp(P) \subset {\bf Env}^{\epsilon} (\{T_i\})_{i=1}^\nCur$ and $\epsilon\rightarrow 0$ as $\mu/\lambda \rightarrow 0.$}\\

If this is not the case, then there exists an $r>0$ and a point $p \in \supp(P)$ such that 
\begin{align*}
d(p, \Env) > r,\ as\ \mu/\lambda \rightarrow 0.
\end{align*}

\begin{enumerate}
\item Next, define $[[S]]$  to be  
 \[ [[S]](\omega) = \int_{S} \omega(\vec{x}), \forall \omega\]
 where
\begin{align*}
S= {\bf Env}(\{T_i\}_{i=1}^\nCur \cup \mdn{T}_{\lambda,\mu}) \cap B(p,r/2),
\end{align*}
and the orientation of $[[S]]$ is induced by $\mdn{T}_{\lambda, \mu}.$ Define a new median to be $\mdn{T}'_{\lambda, \mu}$ 
      \begin{align*}
        \mdn{T}'_{\lambda,\mu} = \mdn{T}_{\lambda,\mu} - \boundary[[S]].
      \end{align*}
Intuitively, $\mdn{T}'_{\lambda,\mu}$ differs from $\mdn{T}_{\lambda,\mu}$ only inside the closure of $B(Q,r/2)$ and it pushes $\mdn{T}_{\lambda,\mu}\myell B(Q,$ $r/2)$ to the boundary of $B(Q, r/2).$
As $\mdn{T}_{\lambda, \mu}$ is the mass regularized median,
and since we are in codimension $1$,  $\mdn{T}'_{\lambda,\mu} = \mdn{T}_{\lambda,\mu} - \boundary[[S]]$, which implies
\[\F_\lambda (\mdn{T}_{\lambda, \mu} - T_i) - \F_\lambda (\mdn{T}'_{\lambda, \mu} - T_i) = \lambda \mass[[S]],\]
we have 
\begin{align}\label{massmediancomparison}
\begin{split}
0 &\geq \left(\sum_{i=1}^\nCur \F_\lambda (\mdn{T}_{\lambda, \mu} - T_i) + \mu \mass(\mdn{T}_{\lambda, \mu})\right) - \left(\sum_{i=1}^\nCur \F_\lambda (\mdn{T}'_{\lambda, \mu} - T_i) + \mu \mass(\mdn{T}'_{\lambda, \mu})\right) \\
& = \nCur \lambda \mass([[S]]) - \mu(\mass(\mdn{T}'_{\lambda, \mu}) - \mass(\mdn{T}_{\lambda, \mu})).
\end{split}
\end{align}
If it were the case that $\mass(\mdn{T}'_{\lambda, \mu}) < \mass(\mdn{T}_{\lambda, \mu})$, then the last row of (\ref{massmediancomparison}) would be greater than 0, which would show that $\mdn{T}_{\lambda,\mu}$ could not be the mass regularized median. Therefore let's assume $\mass(\mdn{T}'_{\lambda, \mu}) \geq \mass(\mdn{T}_{\lambda, \mu})$ and under this assumption, 
\begin{align*}
\mass(\mdn{T}'_{\lambda, \mu}) - \mass(\mdn{T}_{\lambda, \mu}) &= \mass(\mdn{T}'_{\lambda, \mu}\myell \partial B(Q, r/2)) -  \mass(\mdn{T}_{\lambda, \mu}\myell B(Q, r/2))\\
&\leq \mass(\mdn{T}'_{\lambda, \mu}\myell \partial B(Q, r/2)) \\
&\leq (\dimcur +1) \alpha(\dimcur + 1) (r/2)^{\dimcur},
\end{align*}
and hence (\ref{massmediancomparison}) becomes 
\begin{align}\label{massmediancomparison1}
\begin{split}
0 &\geq \left(\sum_{i=1}^\nCur \F_\lambda (\mdn{T}_{\lambda, \mu} - T_i) + \mu \mass(\mdn{T}_{\lambda, \mu})\right) - \left(\sum_{i=1}^\nCur \F_\lambda (\mdn{T}'_{\lambda, \mu} - T_i) + \mu \mass(\mdn{T}'_{\lambda, \mu})\right) \\
& \geq \nCur \lambda \mass([[S]]) - \mu((\dimcur +1) \alpha(\dimcur + 1) (r/2)^{\dimcur}\\
& = \lambda \left(\nCur \mass([[S]]) - \frac{\mu}{\lambda}((\dimcur +1) \alpha(\dimcur + 1) (r/2)^{\dimcur}\right).
\end{split}
\end{align}
Obviously, if $\mass([[S]])$ does not converge to 0 as $\mu/\lambda \rightarrow 0$, then the last row of (\ref{massmediancomparison1}) will eventually be greater than 0, which is a contradiction. 

\item Now let's suppose $\mass([[S]])\rightarrow 0$ as $\mu/\lambda \rightarrow 0$. Define the following sets which are important for the rest of the proof: 

\begin{itemize}
\item[] $h_t = \partial B(Q, t)\cap S,\ t\in (0, r/2)$,
\item[] $g_t = \partial^* S \cap B(Q,t)$,
\item[] $H_t = S\cap B(Q,t).$
\end{itemize}
Note that $g_t\cup h_t = \partial^* H_t$ for $\Hd^{\dimcur}$ almost everywhere and that
\begin{align*}
\Hd^{\dimcur + 1} (H_t) = \int_0^t \Hd^{\dimcur}(h_s) ds\ \mathrm{and}\ \Hd^{\dimcur + 1}(H_t) \leq \Hd^{\dimcur + 1}(S).
\end{align*}

\item {\bf{Claim:}} There exists a $t_0\in (0, r/2)$ such that $\Hd^{\dimcur}(h_{t_0}) \leq \Hd^{\dimcur}(g_{t_0}).$
\begin{enumerate}
\item If claim is true, we are done since this implies $\mass(\mdn{T}'_{\lambda,\mu}) \leq \mass(\mdn{T}_{\lambda,\mu})$, which leads to a contradiction of \cref{massmediancomparison}.
  Assume the claim is false, then $\Hd^{\dimcur}(h_t) > \Hd^{\dimcur}(g_t)$ for all $t\in (0, r/2)$.
\item Choose $\mu/\lambda$ small enough that $\Hd^{\dimcur + 1} (H_{r/2}) < \min\left\{ \frac{\alpha(\dimcur + 1)}{2} \left(\frac{r}{4}\right)^{\dimcur + 1}, \left(\frac{Cr}{4}\right)^{\dimcur +1}\right\},$ where $C$ is the constant in the  relative isoperimetric inequality in a ball (Proposition 12.37 in \cite{maggi-2012-sets}). 
\item This implies that 
\begin{align*}
\frac{\Hd^{\dimcur + 1} (H_{r/4})}{\Hd^{\dimcur+1} B(Q,r/4)} < \frac{1}{2}.
\end{align*}
\item We can then apply relative isoperimetric inequality in a ball (Proposition 12.37 in \cite{maggi-2012-sets}) to say that for $t\in (r/4, r/2)$, 
\begin{align*}
\Hd^{\dimcur}(h_t) = \frac{d}{dt} \Hd^{\dimcur + 1}(H_t) > \Hd^{\dimcur} (g_t) \geq C (\Hd^{\dimcur+1} (H_t))^{\frac{p}{p+1}}.
\end{align*}
Solving the inequality above and integrating from $r/4$ to $r/2$ yields 
\begin{align*}
(p+1)\Hd^{\dimcur + 1}(H_{r/2})^{1/(p+1)} - (p+1) \Hd^{\dimcur + 1}(H_{r/4})^{1/(p+1)} \geq \frac{Cr}{4}. 
\end{align*}
Therefore 
\begin{align*}
\Hd^{\dimcur + 1}(H_{r/2}) \geq \left(\frac{Cr}{4}\right)^{\dimcur + 1}.
\end{align*}
which contradicts (b).
\end{enumerate}
\end{enumerate}
\end{enumerate}
\vspace*{-0.1in}
\end{proof}

\section{Shared Boundaries: Co-dimension $> 1$ results}
\label{sec:shared2}

\subsection{Books are Minimizing}

\label{sec:minimal-books}

\begin{defn}[Books in $\R^3$]
  let $L$ be the vertical axis ($z$-axis or $x_3$-axis) in $\R^3$.
  Let $V \equiv \{\vv_i\}_{i=1}^\nCur$ be $\nCur$ unit vectors in $\R^3$ whose $z$-coordinates are $0$.
  We will call $V$ \emph{stationary} if $\sum_{i=1}^\nCur \vv_i = 0$.
  Let $H_i$ be the half plane containing $\vv_i$ whose boundary is $L$.
  Let $C\subset\R^3$ denote a closed, solid cylinder that is bounded and circular, whose axis is the $z$-axis: $C = \{(x,y,z)\; | \; x^2 + y^2 < r, \; 0 \leq z \leq l\}$.
A set $\bk \subset \R^3$ such that $\bk = C\cap\{\cup_{i=1}^\nCur H_i\}$ for some cylinder $C$ and some stationary $V$ will be called a \emph{book}.
\end{defn}

\begin{defn}[Edge Set, and Pages]
  An \emph{edge set} $\es$ is any set $\es = \partial
  C\cap\{\cup_{i=1}^\nCur H_i\}$ such that the direction set of the
  $H_i$'s is stationary. We can write this more simply as $\es
  = \partial \bk$ for any book $\bk$, if $\partial$ is the varifold
  boundary. We will sometimes write $\es$ as $\es(\bk)$.
  Further, we define the individual \emph{page} $E_i$ to be $H_i \cap \es$.
\end{defn}

Because, in our case, we care about multiplicity, the kind of pinching
that a Lipschitz map can do to reduce the measure of the set
is not of interest to us. Therefore, we will
consider bi-Lipschitz deformations of a book.

\begin{thm}[Books are minimal]
  Given any book $\bk$, the corresponding edge set $\es(\bk)$, and any
  bi-Lipschitz map $f:\R^3\rightarrow\R^3$ such that $f|_{\es(\bk)}$
  is the identity, we have that $\Hd^2(\bk) \leq \Hd^2(f(\bk))$. In
  fact, equality is obtained only in the case that $f(\bk) = \bk$.
\end{thm}

\begin{proof}
  The proof follows from a slicing argument in combination with a new result from graph theory.
  Figure~\ref{fig:bk-deformation-figure} illustrates the details.
  \begin{figure}[hb!]
    \centering
    \scalebox{0.5}{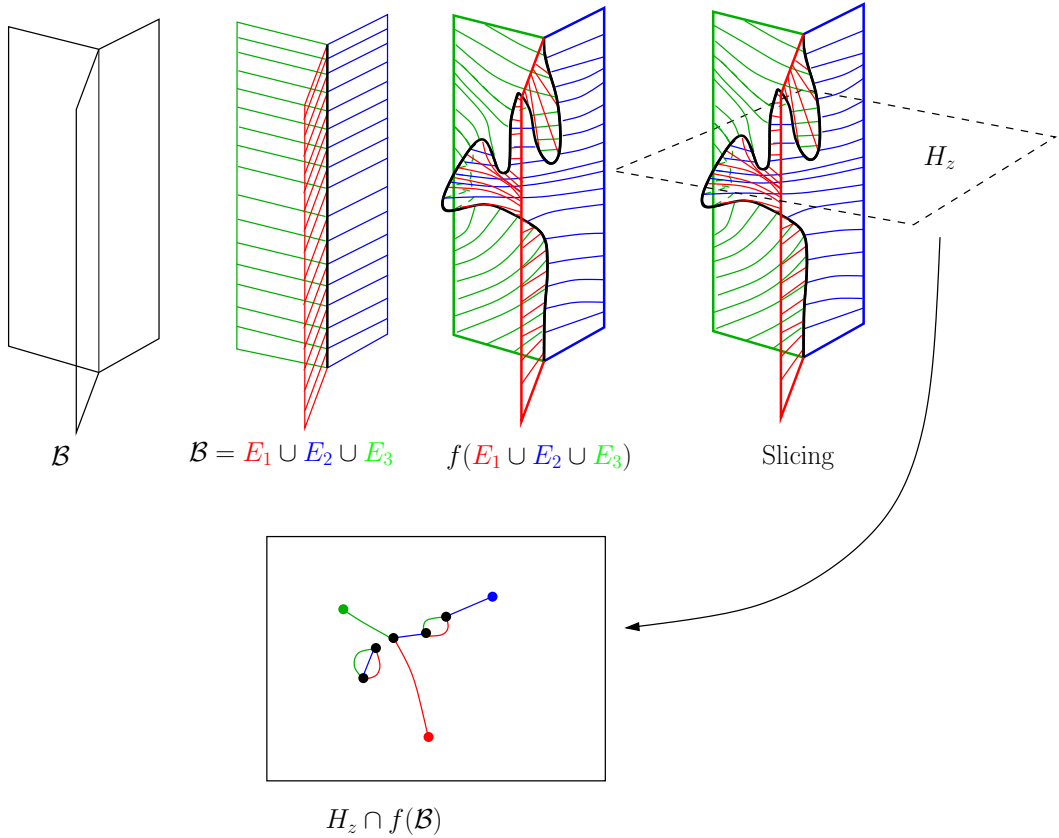}
    \caption{Illustration of slicing in the proof of book
      optimality. ${\red E_1}$, ${\blue E_2}$ and ${\green E_3}$ are
      the leaves of the book, $\bk = {\red E_1}\cup{\blue
        E_2}\cup{\green E_3}$, with colors to help us see what is going
      on.
      $T$ is the ``spine'' of the book $\bk$, and $T_i$'s are the boundaries of the leaves minus $T$.
      \label{fig:bk-deformation-figure}
    }
  \end{figure}

  \begin{enumerate}

    \item
      Denote the $\nCur$ pieces of $\es$ (each ``$C$-shaped'' piece) $T_i$, $i=1,...,\nCur$.
      Define $T = L \cap \bk =  \{0\}\times\{0\}\times [0,l]$ for some $l<\infty$.
      Define $H_{z} = \{(x_1,x_2,x_3) \, | \, x_3 = z\}$.

  \item We can approximate $f(\bk)$ arbitrarily well with a polygonal
      approximation that keeps $f(T_i)=T_i$ fixed.  We outline how
      this is done. By \emph{approximate} here, we mean close in
      measure and in Hausdorff distance: $|\Hd^2(f(\bk)) -
      \Hd^2(\bk_{f,\delta})| + d_H(f(\bk),\bk_{f,\delta}) <
      \delta$, where we are denoting the the approximation of $f(\bk)$
      by $\bk_{f,\delta}$.
      \begin{description}
      \item[$C^1$ Approximation:] The first step is to use the $C^1$
        approximation of Lipschitz maps to find a $C^1$ map $\phi$ approximating $f$
        and a fine enough regular grid $G_i^\epsilon$ on the page of
        the book $E_i$ so that the triangular polygonal surface
        generated by the points $\phi(G_i^\epsilon)$,
        $P_{\phi(G_i^\epsilon)}$, satisfies $|\Hd^2(f(E_i)) -
        \Hd^2(P_{\phi(G_i^\epsilon)})| < \epsilon$ and $d_H(f(E_i),
        P_{\phi(G_i^\epsilon)}) < \epsilon$.
      \item[Fixing up the edges:] We move the boundary points of
        $P_{\phi(G_i^\epsilon)}$ back to the boundary points of
        $P_{f(G_i^\epsilon)}$ paying a penalty of $C'\epsilon$ in the
        $d_H$ distance and $C''\epsilon$ in the difference of measures,
        where $C',C''$ depend only on the size of the Book and not in $f$
        or $\phi$. We denote this new polygonal surface by
        $\hat{P}_{\phi(G_i^\epsilon)}$.
      \item[Perturbing into transverse intersections:] Next we perturb
        the points in $\cup_i \hat{P}_{\phi(G_i^\epsilon)}$ enough
        that none of them coincide and none of the resulting edges and
        faces meet almost every $H_z$ transversely, all without
        introducing more than an additional $\epsilon$ to each of the
        $\Hd^2$ and $d_H$ differences between $P_{\phi(G_i^\epsilon)}$
        and $f(E_i)$. In this case, being transverse boils down to
        none of the edges of faces being horizontal. (This last step
        insuring everything is non-horizontal is not actually
        necessary since we will integrate over the slices and the
        number of slices that could contain horizontal edges of sides
        is finite and therefore ignored in the integration.)
      \item[The Approximation:] We define $\bk_{f,\eta(\epsilon)}$
        to be the resulting perturbed $\cup_i \hat{P}_{\phi(G_i^\epsilon)}$.
      \item[Verifying the approximation:] Doing the book-keeping, we
        find that there is a $C$ not depending on $f$ or $\phi$ such
        that $|\Hd^2(f(\bk)) - \Hd^2(\bk_{f,\eta(\epsilon)})| +
        d_H(f(\bk),\bk_{f,\eta(\epsilon)})< C\epsilon$; i.e.,
        $\eta(\epsilon) = C\epsilon$.
      \item[Conclusion:] Since $\epsilon$ was arbitrary, 
        we can choose $\epsilon = \frac{\delta}{C}$ and we get
        that $\bk_{f,\eta(\epsilon)} = \bk_{f,C\epsilon} = \bk_{f,\delta}$
        which satisfies $|\Hd^2(f(\bk)) -
      \Hd^2(\bk_{f,\delta})| + d_H(f(\bk),\bk_{f,\delta}) <
      \delta$.
      \end{description}
      %  We can (and do) choose the points that define this
      % complex in such a way that all of the $1$- and $2$-dimensional
      % pieces of $\bk_{f,\delta}$ meet every $H_z$ transversely.
      % (\emph{Remarks: In this case, \emph{meeting transversely} means
      %   that the normal to $H_z$ is not a normal to any of the $1$- or
      %   $2$-dimensional components of $\bk_{f,\delta}$.  We use the
      %   $C^1$ approximability of rectifiable sets to get the polygonal
      %   approximation that approximates $f(\bk)$ through a grid on
      %   $\bk$ and mapping that forward. We use the fact that $f$ is
      %   Lipschitz to get the Haussdorff approximation. Because f is
      %   bi-Lipschitz, we know the grid points map forward in a
      %   one-to-one manner. We can perturb things to get the
      %   transversality conditions.})
    \item
      Define $S_z = \bk_{f,\delta}\cap H_z$, $\mu S_z = \Hd^1\myell S_z$, $\mu = \Hd^2\myell\bk$ and $\mu_{f,\delta} =  \Hd^2\myell\bk_{f,\delta}$.

    \item
      Now we slice the measure $\mu_{f,\delta}$ with $H_z$ to  produce another Radon measure $\omega_z$ with the property that its support is $S_z$ and $\mu_{f,\delta}(E) = \int_0^l \omega_z(E)$.
      We also know that $\omega_z = h\myell \mu S_z$ with $h \geq 1$ everywhere.\footnote{In terms of the slicing measures in section 1.9 of the revised edition of the text by Evans and Gariepy~\cite{evans-1992-1}, we would first write $\sigma$ as $g\myell(\Hd^1\myell T)$ which we can do since $\sigma << \Hd^1\myell T$, then we get  that $\omega_z = g(z)\nu_z$.}
    
    \item Define $W_z = \bk \cap H_z$.  Note that all $W_z$ are
      equivalent modulo a vertical translation.  We define vertices
        $v = T \cap H_z$ and $v_i = T_i \cap H_z$ for $i=1,...,\nCur$.
        Note that for all other graphs $G$ with a distinct sequence of
        edges joining each of the $v_i$'s to some other common vertex
        in $H_z$, we will have $\Hd^1(G) > \Hd^1(W_z)$.
    \item
      We denote the corresponding vertices under $f$ (as represented by $\bk_{f,\delta}$) as $v^f = f(T) \cap H_z$ and $v^f_i = f(T_i) \cap H_z$ for $i=1,...,\nCur$.
      Suppose that for almost every $z \in [0,l]$, $S_z$ contains a graph $G$ that has a distinct sequences of edges joining each of the $v^f_i$'s to some other common vertex in $H_z$.
      Then we have  that
      \begin{eqnarray}
               \Hd^2(\bk_{f,\delta}) &   =  &    \int_0^l \omega_z(S_z) \\
                                   & \geq &    \int_0^l \Hd^1(S_z) \\
                                   & \geq &    \int_0^l \Hd^1(W_z) \\
                                   &   =  &    \Hd^2(\bk).
      \end{eqnarray}

    \item Now we need to show that $S_z$ contains the graph $G$
      connecting each of the $v^f_i$'s to some common vertex with
      separate paths. See the bottom box in
      \cref{fig:bk-deformation-figure} for an illustration---even
      though not all black vertices in the middle are connected to the
      red, green, and blue vertices by separate paths here, there does
      exist one such black vertex. The result holds automatically if
      there is only \emph{one} common vertex of the form $v^f = f(T)
      \cap H_z$.  Let $T_{f,\delta}$ represent the approximation of
      $f(T)$ (in the same way $\bk_{f,\delta}$ approximates $f(\bk)$).
      Note that our approximation leaves $f(T_i)$ fixed, and $f$ in
      turn leaves $T_i$ fixed.  We assume $T_{f,\delta} \cap H_z$
      includes more than one common vertex.  We assign $k$ colors
      $\{1,\dots,k\}$ distinctly to each vertex $v^f_i$, as well as to
      the corresponding surface which spans the boundary represented
      by the union of $T^i$ and $T_{f,\delta}$.  Correspondingly, each
      edge in the graph $G$ in $S_z$ is colored with one of the $k$
      colors.  Notice that, by design, each such common vertex (of the
      form $v^f$) in $G$ is connected to $k$ edges, one of each color,
      while each vertex $v^f_i$ is connected to a single edge that is
      colored $i$.  We show the following result: in the subgraph of
      $G$ induced by the common vertices with degree $k$, for every
      vertex there is another vertex to which there exist $k$
      edge-disjoint paths.  In fact, we state and prove this result as
      a new theorem on $k$-colored graphs in the next Subsection (see
      Theorem~\ref{thm:cozyarecomfy}).
    \item Thus we have that \[  \Hd^2(f(\bk)) \geq \Hd^2(\bk). \]
    \item To finish the proof, suppose that $\Hd^2(f(\bk)) =
      \Hd^2(\bk)$ and that some point $p\in\bk$ is not in
      $f(\bk)$. Since $f(\bk)$ is closed, we know that for some
      $\delta > 0$, $B(p,\delta)\cap f(\bk)=\emptyset$.
    \item \label{setK} A set $K$ such that 
      \begin{enumerate}
      \item every slice $K\cap H_z$ contains a separate path from each
        of the $v^f_i$ to a common point and
    \item $B(p,\delta/2)\cap K = \emptyset$ 
      \end{enumerate}
      satisfies \[\Hd^2(K) \geq \epsilon(p,\delta) + \Hd^2(\bk),\] for
      some $\epsilon(p,\delta) > 0$.
    %\item Choose $\bk_{f,\eta}$ such that $\eta <
      % \min(\delta/2,\epsilon(p,\delta)/3)$.
    %\item Then we have that \[\bk_{f,\eta}\cap B(p,\delta/2) =
     % \emptyset\] and \[\Hd^2(\bk_{f,\eta}) <
      %\frac{2}{3}\epsilon(p,\delta) + \Hd^2(\bk)\] which is a contradiction.
    \item Because $f(\bk)$ satisfies the requirements for set $K$ specified in Step \ref{setK} above,
      $\Hd^2(f(\bk)) > \Hd^2(\bk)$.
    \item This implies that $\bk\subset f(\bk)$.
    \item Now suppose that $f(\bk)\setminus\bk$ is not empty and $q\in
      f(\bk)\setminus\bk$. Because (a) $f(\bk)$ is closed and (b) $f$
      is bi-Lipschitz, we have that:
      \begin{enumerate}
      \item there is an $\epsilon > 0$ such that $B(q,\epsilon)\cap\bk = \emptyset$
      \item $\Hd^2(B(q,\epsilon)\cap f(\bk)) > 0$
      \end{enumerate}
      which would imply that $\Hd^2(f(\bk)) > \Hd^2(\bk)$.
\item Thus we conclude that $f(\bk)=\bk$.
  \end{enumerate}
\vspace*{-0.2in}
\end{proof}

\begin{rem}
  One can assume only that f is Lipschitz and, with a minor change in the proof
  get the same result with the exception that one can now only conclude
  that $\Hd^2(f(\bk)\setminus\bk) = 0$.
\end{rem}
\subsubsection{Cozy graphs are comfortable} \label{sssec:cozy}

For the sake of completeness, we start with several definitions on graphs (the basic ones are presented in typical texts on the subject \cite{Harary1969}).
We work with undirected graph $G=(V,E)$, with vertex set $V$ and edge set $E$.

\begin{defn}[$1$-factorization of graph] \label{def:factor}
  A $k$-factor (for $k \in \Z_{>0}$) of $G$ is a $k$-regular (i.e., each vertex has degree $k$) spanning subgraph of $G$.
  A $k$-factorization of $G$ is the partition of $E$ into disjoint $k$-factors.
  The graph $G$ is said to be $k$-factorable if it admits a $k$-factorization.
  In particular, a $1$-factor is a perfect matching of $G$.
  Finally, a $1$-factorization of a $k$-regular graph $G$ is an edge coloring with $k$ colors, i.e., an assignment of one of $k$ colors to each edge in $E$ such that no two edges incident to the same vertex have the same color.
\end{defn}

\noindent We now define the two properties of graphs that are required for our main body of work.
\begin{defn}[cozy graph] \label{def:cozy}
  An undirected graph $G=(V,E)$ is called \emph{$k$-cozy} if it is a $1$-factorable, $k$-regular, connected graph (such that the $k$ edges incident at each vertex $v \in V$ are assigned distinct  colors from $\{1,\dots,k\}$).
\end{defn}

\begin{defn}[comfortable graph] \label{def:comfy}
  An undirected graph $G=(V,E)$ is called \emph{$k$-comfortable} if for every vertex $v \in V$, there is another vertex $v' \in V$ such that there exist $k$ edge-disjoint $v$-$v'$ paths in $E$.
\end{defn}

\noindent We will prove that a $k$-cozy graph is also $k$-comfortable.
To this end, we prove several smaller results, which we use in the proof of the main theorem.
We need two additional definitions first.

\begin{defn}[spine, rib] \label{def:spinerib}
  For a set $U \subseteq V$ of vertices of graph $G=(V,E)$, and for $i=1,2$, let $E_i(U) \subseteq E$ be the set of edges of $G$ incident with exactly $i$ vertices in $U$.
  The edges in $E_1(U)$ are called \emph{spines} of $U$, and the edges in $E_2(U)$ are called \emph{ribs} of $U$.
\end{defn}

\begin{lemma} \label{lem:spineparity}
  For a set $U$ of vertices of a $k$-cozy graph $G$, the number of spines of any given color and $|U|$ have the same parity.
\end{lemma}
\begin{proof}
  For any given color, let $s$ and $r$ be the number of spines and ribs of that color, respectively, for $U$.
  Since $G$ is $k$-cozy, every vertex of $G$ is incident to a unique edge of a given color.
  Hence we must have $|U|=s + 2r$, and the result follows immediately.
\end{proof}

\begin{cor} \label{cor:evenspines}
  Let $U$ be a set of vertices of a $k$-cozy graph $G$ with fewer than $k$ spines.
  Then the number of spines of $U$ of any given color is even.
\end{cor}
\begin{proof}
  If there were an odd number of spines of some color, then $|U|$ must be odd due to Lemma \ref{lem:spineparity}.
  But if $|U|$ is odd, then the number of spines of every color must also be odd, which implies $U$ must have a spine of every color.
  Hence the total number of spines of $U$ is at least $k$, giving a contradiction.
\end{proof}

\begin{defn}[edge connectivity]
  The \emph{edge connectivity} of graph $G$ is the minimum number of edges whose removal disconnects $G$.
  We denote the edge connectivity of $G$ by $\ec{G}$.
\end{defn}

\begin{cor} \label{cor:edgeconeven}
  Let $G$ be a $k$-cozy graph whose edge connectivity $\ec{G} < k$.
  Then $\ec{G}$ is even.
\end{cor}
\begin{proof}
  Let the edge connectivity of $G$ be $\ec{G} = c < k$.
  Let $D \subset E$ be a subset of edges with $|D|=c$ whose removal disconnects $G$ into two components $G_1$ and $G_2$.
  The edges in $D$ are spines of the sets of vertices of $G_i$ for $i=1,2$.
  Since $c < k$, $c$ must be even by Corollary \ref{cor:evenspines}.
\end{proof}

\begin{lemma} \label{lem:cpntedgecon}
  Let $G=(V,E)$ be $k$-cozy and $\ec{G} = 2 \ell$.
  Let $D \subset E$ be a set of edges with $|D| = 2\ell$ whose removal disconnects $G$ into two components $G_1$ and $G_2$.
  Then $\ec{G_1} \geq \ell$ and $\ec{G_2} \geq \ell$.
\end{lemma}

%\begin{proof}
  \noindent\textit{Proof.}
  Let $D'$ be a set of $\ec{G_1}$ edges whose removal disconnects $G_1$ into components $H_1$ and $H_2$.
  
  \begin{wrapfigure}{l}{2in}
    \centering
    \vspace*{-0.1in}
    \includegraphics[scale=0.6]{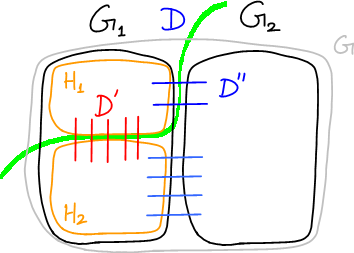}
    \vspace*{-0.3in}
  \end{wrapfigure}
  \noindent
  Also let $D_i \subset D$ for $i=1,2$ be the set of edges in $D$ that join vertices in $H_i$ to vertices in $G_2$.
  Notice that $D_1 \cup D_2 = D$.
  Finally, let $D''$ be the smaller of $D_1$ and $D_2$.
  Since $2 \ell = |D| = |D_1| + |D_2|$, we have $|D''| \leq \ell$.
  Here, $D' \cup D'' \subset E$ is a set of edges whose removal disconnects $G$.
  Hence we have $2 \ell \leq | D' \cup D'' | = |D'| + |D''|  = \ec{G_1} + \ell$, which implies $\ec{G_1} \geq \ell$.
  Reversing the roles of $G_1$ and $G_2$ gives $\ec{G_2} \geq \ell$.
  \qed
  %\end{proof}

\vspace*{0.2in}

\begin{cor} \label{cor:edgdisjpaths}
  Let $G, D, G_1, G_2$ be as defined in Lemma \ref{lem:cpntedgecon}.
  Let $V_1, V_2$ be multisets of vertices of one component (either $G_1$ or $G_2$) with $|V_1| = |V_2| = q \leq \ell$.
  Then there are $q$ edge-disjoint paths in $G$ connecting $V_1$ and $V_2$ with multiplicities preserved, such that every vertex of $V_1$ and of $V_2$ is the end point of some such path.
\end{cor}

\begin{proof}
  We append a source and sink vertex $s$ and $t$, and attach $s$ to each vertex in $V_1$ and $t$ to each vertex in $V_2$, with multiple edges to account for multiplicities of the vertices.
  Let this new pseudograph (due to multiplicities of some nodes and edges) be called $G'$.
  Then $\ec{G'} = q$, and by the max flow-min cut theorem (see \cite[Theorem 6.7]{AMO1993}), $G'$ has $q$ edge-disjoint $s$-$t$ paths.
  Removing $s$ and $t$ from $G'$ provides the $q$ edge-disjoint paths connecting $V_1$ and $V_2$ in $G$.
\end{proof}

\noindent We need one more construction related to spines, which we employ in the proof of the main result in this subsection.
\begin{defn}[special edges and knitting] \label{def:knit}
  Let $U \subset V$ be a set of vertices of a $k$-cozy graph $G$ with fewer than $k$ spines.
  For each of the $k$ colors, there must exist an even number of spines of that color by Corollary \ref{cor:evenspines}.
  We partition into pairs the vertices in $U$ that are end points of spines of a given color.
  In the subgraph $G(U)$ of $G$ induced by $U$, we join each such pair of vertices with a new edge of the same color.
  These new edges are called \emph{special edges}.
  Repeating this process for every color gives us a supergraph of $G(U)$, which we refer to as a \emph{knitting} of $G(U)$.
  Any knitting of $G(U)$ is immediately seen to be $k$-cozy.
\end{defn}

\noindent We now present the main result related to cozy and comfortable graphs.

\begin{thm}[Cozy graphs are comfortable] \label{thm:cozyarecomfy}
  For $k \geq 0$, every $k$-cozy graph is $k$-comfortable.
\end{thm}

\begin{proof}
  We do induction on $k$, noting that a $0$-cozy graph is the trivial graph (a single vertex).
  Also, the only $1$-cozy graph is $K_2$, the complete graph on $2$ vertices, which is a pair of vertices connected by a single edge.
  For $k=2$, we observe that a $2$-cozy graph must be an even cycle (as we cannot assign two colors in an alternating fashion to edges along an odd cycle such that each vertex is incident to two edges of the two colors).
  Hence every pair of vertices in a $2$-cozy graph has two edge-disjoint paths connecting them, showing the result holds for $k=2$.

  We assume the result holds for $k=2r$ and show it must then hold for $k=2r + t$ for $t \in \{1,2\}$.
  Assume every $2r$-cozy graph is also $2r$-comfortable, and let $G$ be a $(2r+t)$-cozy graph.
  If $\ec{G}=2r+t$, then there exist $2r+t$ edge-disjoint paths between every pair of distinct vertices in $G$ (again, by the max flow-min cut theorem, as seen in Corollary \ref{cor:edgdisjpaths}).
  Hence any counterexample must have $\ec{G} < 2r+t$.
  Further, by Corollary \ref{cor:edgeconeven}, any such counterexample must have $\ec{G}=2\ell$ for $\ell \leq r$.

  Let $G=(V,E)$ be such a counterexample with the smallest number of vertices.
  Let $D \subset E$ be a set of $2\ell$ edges whose removal disconnects $G$ into components $G_1=(V_1,E_1)$ and $G_2=(V_2,E_2)$.
  Let $v \in V$ be a vertex such that $v \in V_1$.
  We knit $G(V_1)$, as described in Definition \ref{def:knit}, to create a $(2r+t)$-cozy graph $H$ with fewer vertices than $G$.
  Hence $H$ has a smaller number of vertices than $G$, and is therefore $(2r+t)$-comfortable.

  \medskip
  Let $v'$ be a vertex in $H$ such that there exist $(2r+t)$ edge-disjoint $v$-$v'$ paths $P_1, \dots, P_{2r+t}$ in $H$.
  We partition these paths into special edges $e_1, \dots, e_q$ and maximum length subpaths $\{Q_j\}$ that do not contain special edges.
  Let each special edge be of the form $e_i = \{u_i,w_i\}$, where $u_i$ is encountered first along a $v$-$v'$ path $P_h$.
  Let $\{u_i,y_i\}$ and $\{z_i,w_i\}$ be the spines of $D$ of the same color as $e_i$.
  We consider the multisets of vertices $Y = \{y_i\}_{i=1}^q$ and $Z = \{z_i\}_{i=1}^q$, as defined by these special edges.
  Notice that vertices in $Y$ and $Z$ belong to $V_2$ (i.e., to component $G_2$).
  By Corollary \ref{cor:edgdisjpaths}, there exist $q \leq \ell$ edge-disjoint $Y$-$Z$ paths.
  Observe that all these paths are located within $G_2$, as $\ec{G_2} \geq \ell$ by Lemma \ref{lem:cpntedgecon}.
  We extend each of these $q$ paths in $G_2$ of the form $y_i,\dots,z_{\sigma(i)}$ (for some index function $\sigma(i)$ which takes care of multiplicities) to $v$-$v'$ paths in $G$ of the form $R_i = u_i,y_i,\dots,z_{\sigma(i)},w_{\sigma(i)}$.
  The edge $\{u_i,y_i\}$ is the only spine of $V_1$ of its color, and hence if $\{u_i,y_i\}$ and $\{u_l,y_l\}$ are the same edge (on account of multiplicities), it must hold that $i=l$.
  The same result holds for the spine at the other end $\{z_{\sigma(i)},w_{\sigma(i)}\}$.
  Hence the $\{R_i\}$ paths are edge-disjoint among each other, and also, by definition, from the $\{Q_j\}$ paths.

  \medskip
  Finally, to certify the existence of $2r + t$ edge-disjoint $v$-$v'$ paths in $G$, we form a new graph $F$ whose \emph{vertices} are the $\{Q_j\}$ and $\{R_i\}$ paths.
  Two vertices in $F$ are joined by an edge in $F$ if the end vertex of the first path in $G$ is the start vertex of the other path (in $G$).
  If a path $Q_j$ is already a $v$-$v'$ path in $G$, we add two vertices corresponding to $Q_j$ in $F$.
  Observe that in the new graph $F$, every vertex has degree $2$ except for those which correspond to paths in $G$ whose start vertex is $v$ or whose end vertex is $v'$.
  All vertices in $F$ of the latter type have degree $1$.
  Hence every connected component of $F$ is a path or a cycle.
  Further, we observe there are $2r + t$ vertices in $F$ which correspond to paths in $\{Q_j\} \cup \{R_i\}$ whose start vertex (in $G$) is $v$.
  There are $2r + t$ \emph{additional} vertices in $F$ which correspond to paths in $\{Q_j\} \cup \{R_i\}$ whose end vertex (in $G$) is $v'$.
  Due to the way $F$ is constructed, these two sets of vertices must be end vertices of $2r+t$ vertex-disjoint paths in $F$.
  These $2r+t$ paths in $F$ correspond to the desired $2r+t$ edge-disjoint $v$-$v'$ paths in $G$, whose existence contradicts $G$ being a counterexample to the result in the theorem.
  Hence every $(2r+t)$-cozy graph is also $(2r+t)$-comfortable.
\end{proof}

\subsection{Regular points have Book-like tangent cones}
\label{sec:regular-books}

We now show a nice property of the median $\mdn{T}_\lambda$ of a set of smooth $1$-currents $\{T_i\}_{i=1}^\nCur$ in $\R^3$ with shared boundaries  under the condition that all minimal surfaces $S_i$'s spanned by the median $\mdn{T}_\lambda$ and $T_i$'s are smooth.
Moreover, according to Krummel \cite{BK2014}, as $\mdn{T}_\lambda$ is the intersection of all smooth minimal surfaces, it is also smooth.

Before stating the general result, we will begin with a simple case.
Let $Cyl(r,h)$ be a cylinder with radius $r$ and height $h$ and $\{H_i\}_{i=1}^\nCur$ be $\nCur$ half hyperplanes with shared boundary at the central axis $L$ of $Cyl(r,h)$.
We assume the unit vectors orthogonal to $L$ for each hyperplane add up to $0$.
We can prove by the coarea formula and the properties of the geometric median for coplanar points that the median of the input currents $H_i \cap \partial Cyl(r,h)$ has to be the central axis of $Cyl(r,h)$.
We call all the hyperplanes inside $Cyl(r,h)$ a book.  

Next we consider the general case where there are $\nCur$ smooth $1$-currents in $\R^3$. The following is an outline of the proof. In the neighborhood around every point $\vx \in \supp(\mdn{T}_\lambda - \partial \mdn{T}_\lambda)$, the median and minimal surfaces will be well approximated by their tangent cones, which are planes.
We will assume that, in order to minimize the sum of flat norms distances, the tangent cones of the minimal surfaces have to form a book, and the tangent cone for the median is where the pages meet.
Otherwise, we may find another $\mdn{T}^{\mbox{new}}_\lambda$ which minimizes the page areas.
Even though, it will add extra areas by connecting things together on the boundary of $Cyl(r,\delta)$, we will show that the extra area will not exceed the total decrease in areas from the pages.
This result is going to be proved using the following steps:

\begin{enumerate}
  \item  Assume the tangent cones of $S_i$'s do not form a book.
    Then we can replace $\mdn{T}_\lambda$ and $S_i$'s around $x$ with their tangent cones $L_{\vx}$ and $T_\vx S_i$'s.
    The error, $E_1$, of the area difference between $S_i$'s and $T_\vx S_i$'s will be very small in the neighborhood of $x$ because of smoothness (see \cref{fig-tangcone}).

      \begin{figure}[htp!] 
        \centering
        \scalebox{0.5}{
          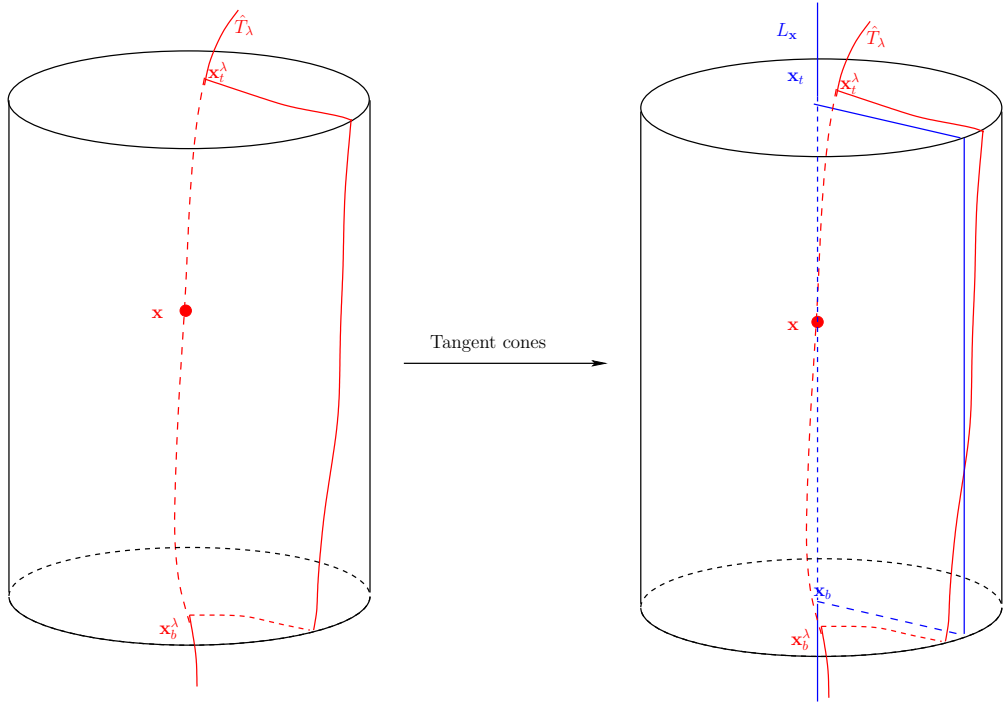}
        \caption{\label{fig-tangcone} Tangent cone inside some cylinder: As assumed, both $\mdn{T}_\lambda$ and $S_i$'s are smooth, and replacing them with their tangent cones will not yield a big difference.
          In fact, the area difference between $S_i$'s and $T_\vx S_i$' in the neighborhood can be proven to be of the order $o(r) \delta$.}
      \end{figure}

    \item  \label{step-moveLx} Under our assumption that $T_\vx S_i$'s do not form a book, we can move $L_{\vx}$ to some other $L'_{\vx}$ such that after the movement, the sum of the areas of the pages will decrease.
      Denote the new pages as $T_\vx S'_i$'s.
      The improvement of this step is of the order $r\delta$. 

    \item The change in Step \ref{step-moveLx} defines a new median $\mdn{T}_\lambda^{\mbox{new}}$ in the following way: 
      \begin{itemize}
        \item $\mdn{T}^{\mbox{new}}_\lambda = \mdn{T}_\lambda$ outside the $Cyl(r,\delta)$. 
        \item On the top and bottom of $Cyl(r,\delta)$, $\mdn{T}^{\mbox{new}}_{\lambda}$ are the line segments connecting $\mdn{T}_{\lambda} \cap Cyl(r,\delta)$ and $L'_{\vx} \cap Cyl(r,\delta)$ on the top and bottom of $Cyl(r,\delta)$, respectively (see \cref{fig-newmdn}).
          \begin{figure}[htp!]
            \centering
            \scalebox{0.5}{
              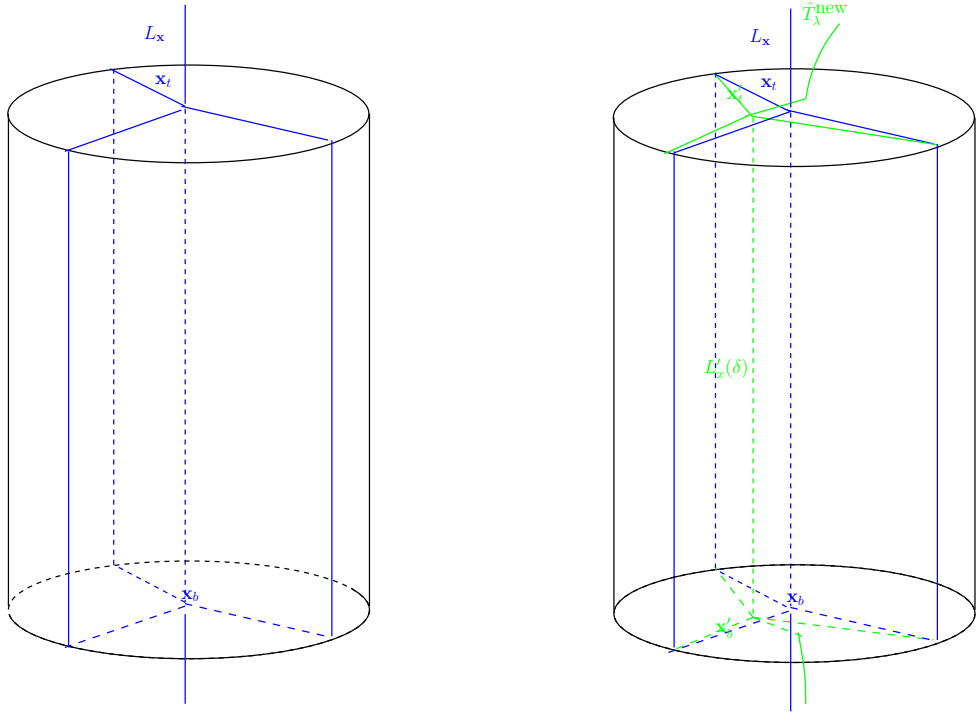}
            \caption{New median \label {fig-newmdn}}
          \end{figure}

        \item $\mdn{T}^{\mbox{new}}_\lambda = L'_{\vx}$ inside $Cyl(r,\delta)$. 
      \end{itemize} 
\end{enumerate}

Compared to $\mdn{T}_\lambda$, $\mdn{T}^{\mbox{new}}_\lambda$ improves the flat norm inside $Cyl(r,\delta)$.
However, it does add extra costs on the top, bottom, and side of $Cyl(r,\delta)$.
If we can show the improvement is greater than the additional cost, then $\mdn{T}^{\mbox{new}}_\lambda$ will be a better choice than $\mdn{T}_\lambda$ for the median.
In particular, the flat norm is calculated by finding a minimizer $S$, and if we are able to construct a different collection of $\{S'_i\}$, whose sum of the areas is still smaller than the flat norms when using $\mdn{T}_\lambda$, then $S$ would not have been a minimizer in the first place.\\

And the way we pick $\{S'_i\}$ is the following:
  \begin{enumerate}
    \item $S'_i = S_i$ outside $Cyl(r,\delta)$. 
    \item $S'_i = T_\vx S'_i$ inside $Cyl(r,\delta)$, where the orientation on $S'_i$ is induced by $S_i,$
    \item On the top and bottom of $Cyl(r,\delta)$, $S'_i$ is defined to be the region generated by swiping $S_i$ to $T_\vx S'_i$ on the top and bottom of $Cyl(r,\delta)$ along $\mdn{T}^{\mbox{new}}_{\lambda}$ and the corresponding arc.  (see \cref{fig-newmdndtls}.)
    \item On the side of the $Cyl(x,\delta)$, $S'_i$ is the region caused by swiping $S_i$ to $T_\vx S'_i$ on the side (see \cref{fig-newmdn}).

      \begin{figure}[htp!]
        \centering
        \scalebox{0.5}{
          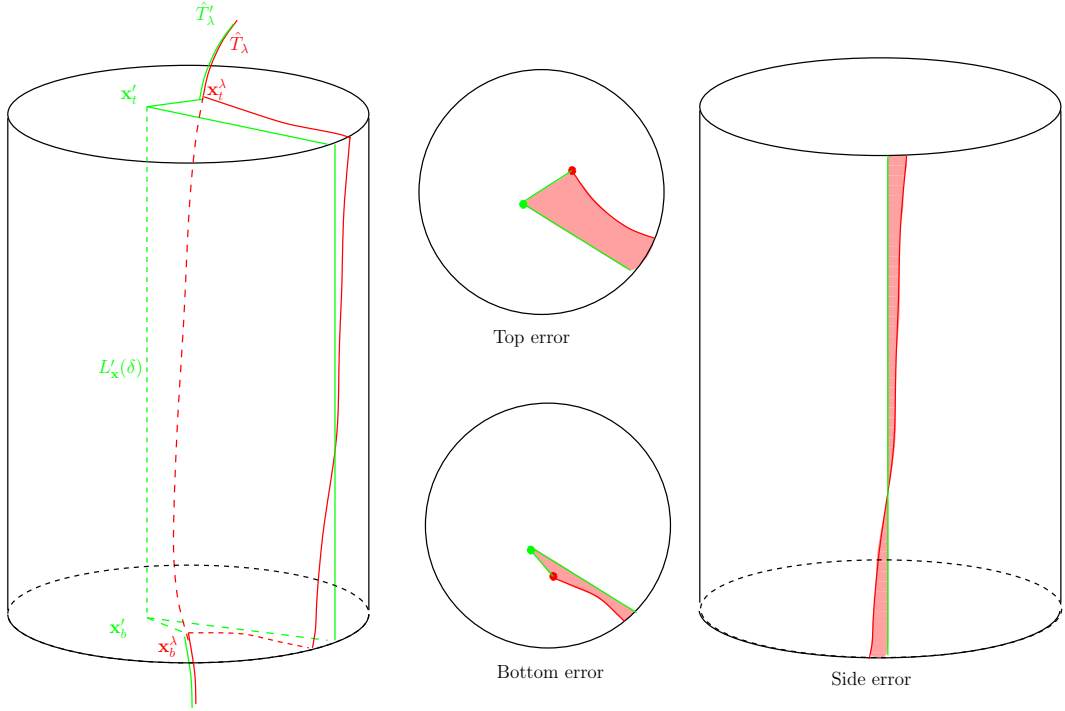}
        \caption{New Median. \label{fig-newmdndtls}}
      \end{figure}

  \end{enumerate}

  The cost for each $S'_i$ on the top and bottom can be bounded by the area of the whole circle, and the cost of the side is $o(r) \delta $ where $\frac{o(r)}{\delta} \rightarrow 0$ as $r\rightarrow 0$.
  Therefore it is sufficient to show: 
  \begin{align*}
    \mbox{Improvement} - \mbox{Cost} ~ \sim ~ r\delta - 2\nCur \pi r^2 - \nCur o(r) \delta ~~> 0,
  \end{align*} 
  and this result can be seen to hold by choosing $\delta / r$ to be a big number.
  Now, we are going to state the problem and give a detailed proof. 

  \begin{thm} \label{thm-tgtconesarebooks}
    Let $\{T_i\}_{i=1}^\nCur$ be $\nCur$ $1$-currents with shared boundaries in $\R^3$, and $\mdn{T}_\lambda$ be their median.
    Then for any $\vx \in \mdn{T}_\lambda \backslash \partial \mdn{T}_\lambda$, there exists a cylinder $Cyl(\vx)$, such that the tangent cone for the minimal surfaces inside $Cyl(\vx)$ is a book, assuming $\mdn{T}_\lambda$ and all spanning currents $S_i$'s between $T_i$'s and $\mdn{T}_\lambda$ are smooth. 
  \end{thm}

  \begin{proof} 
    The proof will be presented in several detailed Steps.

    \medskip
    \textbf{Step 1:} Find an appropriate cylinder $Cyl(\vx,r,\delta)$ centered at $\vx$ with radius $r$ and height $h$, such that $d_H(S_i\cap Cyl(\vx,r,\delta),T_\vx S_i\cap Cyl(\vx,r,\delta))/r$ and $r/\delta$ can be as small as possible, where $d_H$ is the Hausdorff distance, where $T_\vx S_i$ is the tangent cone for the support of $S_i$ at $\vx$.

    Let $L_{\vx}$ be the tangent cones for $\supp(\mdn{T}_\lambda)$ at $\vx$. In the proof, we will suppress the notation $\supp (\cdot)$ to just write $\cdot.$ Since $\mdn{T}_\lambda$ is smooth and $L_{\vx}$ is tangent to $\mdn{T}_\lambda$ at $\vx$, then within some neighborhood of $\vx$, $\mdn{T}_\lambda$ has to stay inside the cone $\Cone_T(h_T, \theta_T)$ with central axis $L_{\vx}$, height $h_T$ and angle $\theta_T$.
    Denote the part of $\mdn{T}_\lambda$  and $L_{\vx}$ that are inside $\Cone_T(h_T, \theta_T)$ to be $\mdn{T}_\lambda(h_T,\theta_T)$ and $L_{\vx}(h_T,\theta_T)$, respectively.
    Then $\mdn{T}_\lambda(h_T,\theta_T)$ can be viewed as a graph over $L_{\vx}(h_T,\theta_T)$ for some smooth Lipschitz function $f$, with Lipschitz constant $Lip(f).$
    Now let us choose $\theta_T$ to satisfy $\tan\theta_T \geq \Lip(f)$ and
    
    \begin{align}\label{tangentlineangle}
      \theta_T \rightarrow 0 \mbox{ and } \frac{d_H(\mdn{T}_\lambda (h_T,\theta_T), L_{\vx}(h_T,\theta_T))}{h_T}\leq \tan\theta_T \rightarrow 0 \mbox{ as } h_T\rightarrow 0. 
    \end{align}

    \begin{figure}[htp!]
      \centering
      \scalebox{0.5}{
        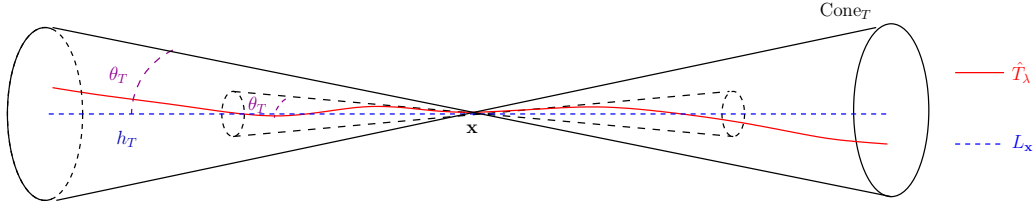}
      \caption{\label{fig-tgtcone} Cone containing $\mdn{T}_\lambda(h_T, \theta_T)$: Since $L_{\vx}(h_T, \theta_T)$ is tangent to $\mdn{T}_\lambda(h_T, \theta_T)$ at $\vx$, the angle of the cone, $\theta_T$, will get smaller as the height $h_T$ of the cone decreases.}
    \end{figure}

    Similarly, since $S_i$ is smooth, within some other neighborhood of $\vx$, $S_i$ must stay inside the cone $\Cone_S(h_S, \theta_S)$ symmetric to $T_\vx S_i$ with height $h_S$ and angle $\theta_S$.
    Denote the parts of $S_i$ and $T_\vx S_i$ that stay inside $\Cone_S(h_S, \theta_S)$ to be $S_i(h_S, \theta_S)$ and $T_\vx S_i(h_S,\theta_S)$, respectively.
    Then $S_i(h_S, \theta_S)$ can be viewed as a graph of a smooth function $g$ over $T_\vx S_i$ for some smooth function $g$.
    Because $S_i$ is tangent to $T_\vx S_i$ at $\vx$, the gradient of $g$ at $\vx$ is 0, and
    \begin{align*}
      \lim_{\vy \rightarrow \vx} |\nabla g(\vy)| = 0.
    \end{align*}
We also know that $|\nabla g|$ is uniformly bounded in any closed subset of $\Proj_{T_\vx S_i} S_i(h_S, \theta_S)$. We may therefore pick $\Cone_S(h_S, \theta_S)$ in the following way:
\begin{itemize}
  \item Let $r_S$ be the radius of $Cyl_S$, and define $\theta_S\in (-\pi/2, \pi/2)$ such that $\tan\theta_S = \Lip(g|_{B(\vx, r_s)})$ where $B(\vx, r_s)\subset \Proj_{T_\vx S_i} S_i(h_S, \theta_S)$, and set $h_S = r_S\theta_S$.
\end{itemize}

\noindent Then $\theta_S = \arctan(\Lip(g))\rightarrow 0$ as $r_S\rightarrow 0$ and 
\begin{align}\label{tangentsufaceangle}
  \frac{d_H(S_i(h_S, \theta_S), T_\vx S_i(h_S,\theta_S))}{r_S} ~ \leq ~ \frac{h_S}{2r_S} ~ = ~
  \frac{1}{2}\tan\theta_S\rightarrow 0 \mbox{ as } h_S \rightarrow 0.
\end{align} 

\begin{figure}[htp!]
  \centering
  \scalebox{0.5}{
    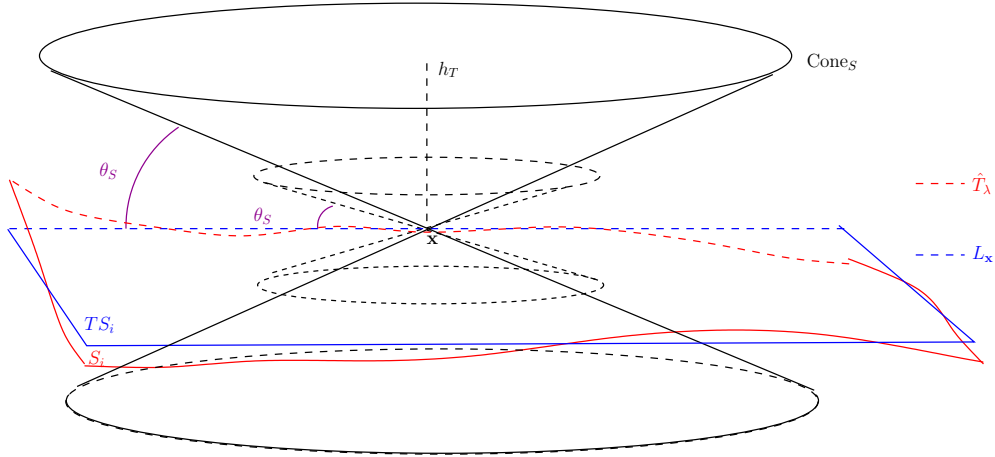}
  \caption{\label{fig-tgtconeS} Cone containing $S_i(h_S,\theta_S)$: Since $T_\vx S_i(h_S,\theta_S)$ is tangent to $S_i(h_S,\theta_S)$ at $\vx$, the angle of the cone, $\theta_S$, will get smaller as the height $h_S$ of the cone decreases.
  This implies both $S_i(h_S,\theta_S)$ and $T_\vx S_i(h_S,\theta_S)$ will stay inside a narrower cone as $h_S$ goes to 0.}
\end{figure}

Now consider a sequence of cylinders $Cyl(\delta_{k'},r_{k'})$ around $\vx$ with central axis $L_{\vx}$, heights $\delta_{k'}$ and radii $r_{k'}$, such that the ratio between radii and heights, $r_{k'}/\delta_{k'} = \epsilon$ for all $k'$. Here $\epsilon$ is a positive constant that will be determined later. Therefore since both $\theta_T$ and $\theta_S$ go to $0$, we may pick $\theta_{k'} = \max\{\theta^{k'}_T, \theta^{k'}_S\}$, where $\theta^{k'}_T, \theta^{k'}_S$ are two angles for the cones $\Cone_T(\delta_{k'},\theta_{k'})$ and $\Cone_S(\delta_{k'},\theta_{k'})$ corresponding to $Cyl( \delta_{k'},r_{k'})$, such that the followings are true:

\begin{align}
  \mdn{T}_\lambda(\delta_{k'},\theta_{k'}), L_{\vx}(\delta_{k'},\theta_{k'}) \subset \Cone_T(\delta_{k'},\theta_{k'}) \subset Cyl(\delta_{k'},r_{k'}) \label{eqn-cdn1} \\ 
S_{k'}(\delta_{k'},\theta_{k'}), T_\vx S_i(\delta_{k'},\theta_{k'})\subset Cyl_S(\theta_{k'}, r_{k'}) \label{eqn-cdn2} \\ 
\frac{d_H(\mdn{T}_\lambda(\delta_{k'},\theta_{k'}),  L_{\vx}(\delta_{k'},\theta_{k'}))}{r_{k'}}\rightarrow 0 \mbox{ as } k' \rightarrow \infty \label{eqn-cdn3} \\ 
\frac{d_H(S_i(\delta_{k'},\theta_{k'}), T_\vx S_i(\delta_{k'},\theta_{k'}))}{r_{k'}}\rightarrow 0 \mbox{ as } k' \rightarrow \infty. \label{eqn-cdn4}
\end{align}

\begin{figure}[htp!]
  \centering
  \scalebox{0.5}{
    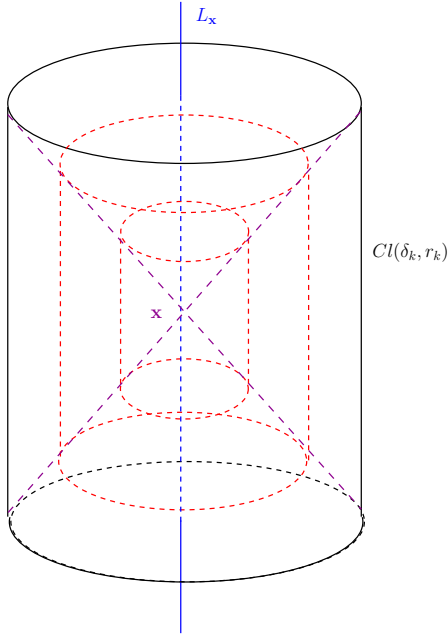}
  \caption{\label{fig-cyl} Ratio preserved cylinder $Cyl(\delta_{k'},r_{k'})$'s.
    For the sequence of cylinders $Cyl(\delta_{k'},r_{k'})$'s, the ratio $r_{k'}/\delta_{k'}$ stays the same, where $r_{k'}$ is the radius and $\delta_{k'}$ is the height.}
\end{figure}

% ----------------------------------------------------------
\textbf{Step 2:} Find the error between $S_i$ and $T_\vx S_i$ inside the cylinder. \\

Similarly as in Step 1, let $\Proj_{T_\vx S_i} S_i(\delta_{k'},\theta_{k'})$ be the image of the orthogonal projection of $S_i(\delta_{k'},\theta_{k'})$ into the plane containing $T_\vx S_i$.
As mentioned before, since $S_i(\delta_{k'},\theta_{k'})$ is smooth, it can be treated as the graph of some smooth function $g$ over $\Proj_{T_\vx S_i} S_i(\delta_{k'})$.
More importantly, $g$ is Lipschitz and $\Lip(g) \leq \tan\theta_{k'}$ since $S_i(\delta_{k'},\theta_{k'}) \subset Cyl_S(\theta_{k'}, r_{k'})$.
Therefore 
\begin{align}\label{projectionerror}
  \begin{split}
    |\Hd^2(S_i(\delta_{k'},\theta_{k'})) - \Hd^2(\Proj_{T_\vx S_i}(S_i(\delta_{k'},\theta_{k'})))| & \leq \left(\sqrt{1+\Lip^2(g)} - 1\right) \Hd^2(\Proj_{T_\vx S_i} S_i(\delta_{k'},\theta_{k'}))\\
& = \left( \frac{1+\Lip^2(g)-1}{\sqrt{1+\Lip^2(g)}+1} \right) \Hd^2(\Proj_{T_\vx S_i} S_i(\delta_{k'},\theta_{k'}))\\
& \leq \frac{\Lip^2 (g)}{2} \Hd^2(\Proj_{T_\vx S_i} S_i(\delta_{k'},\theta_{k'}))\\
& \leq \frac{\Lip^2 (g)}{2} (2r_{k'}\delta_{k'})\\
& = \Lip^2(g) r_{k'}\delta_{k'}\\
& \leq \tan^2\theta_{k'} \cdot r_{k'}\delta_{k'} \, .
  \end{split}
\end{align}

The fourth inequality follows from the observation that area of $\Proj_{T_\vx S_i}S_i(\delta_{k'},\theta_{k'})$ cannot exceed the area of $T_\vx S_i \cap Cyl_S(\theta_{k'}, r_{k'}) = 2r_{k'}\delta_{k'}$.

Next, we will calculate the area difference between $T_\vx S_i(\delta_{k'},\theta_{k'})$ and $\Proj_{T_\vx S_i} S_i(\delta_{k'},\theta_{k'})$. By recalling the definition of $\Proj_{T_\vx S_i} S_i(\delta_{k'},\theta_{k'})$, $T_\vx S_i(\delta_{k'},\theta_{k'})$ and $\Proj_{T_\vx S_i} S_i(\delta_{k'},\theta_{k'})$ are identical except at the places near $L_{\vx}(\delta_{k'},\theta_{k'})$ and $T_\vx S_i(\delta_{k'},\theta_{k'})\cap \partial Cyl(\delta_{k'}, r_{k'})$---see \cref{fig-proj}.
\begin{figure}[htp!]
  \centering
  \scalebox{0.5}{
    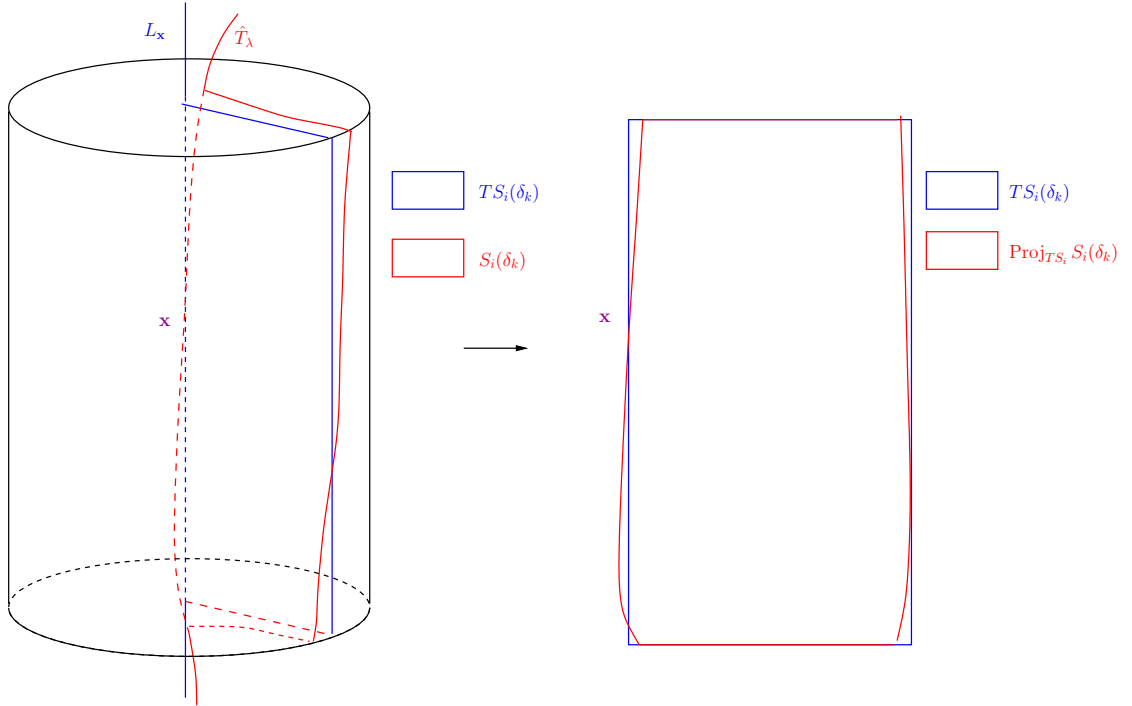}
  \caption{\label{fig-proj} Ratio preserved cylinder $Cyl(\delta_{k'},r_{k'})$: the area differences occur only around $L_{\vx}(\delta_{k'},\theta_{k'})$ and $T_\vx S_i(\delta_{k'},\theta_{k'})\cap \partial Cyl(\delta_{k'}, r_{k'})$, while the other parts are identical.}
\end{figure}

\begin{itemize}
  \item \emph{Area difference near $T_\vx S_i(\delta_{k'},\theta_{k'})$:} \\
  \smallskip
  \\
    This area difference is caused by the deviation from $L_{\vx}(\delta_{k'},\theta_{k'})$ to $\mdn{T}_\lambda(\delta_{k'},\theta_{k'})$, which is controlled by $d_H^{L_{\vx}}(L_{\vx}(\delta_{k'},\theta_{k'}),\mdn{T}_\lambda(\delta_{k'},\theta_{k'}))$, the Hausdorff distance in $L_{\vx}$. Since orthogonal projections into subspaces do not increase distances, 
    \[ d_H^{L_{\vx}}(\Proj_{T_\vx S_i}S_i(\delta_{k'},\theta_{k'}), T_\vx S_i(\delta_{k'},\theta_{k'})) ~ \leq ~
    d_H(L_{\vx}(\delta_{k'},\theta_{k'}),\mdn{T}_\lambda(\delta_{k'},\theta_{k'})).\]
    Hence the area difference near $L_{\vx}(\delta_{k'},\theta_{k'})$, denoted by $AD_1$, is given by 
    \begin{equation}\label{areadiff1}
      \begin{array}{rcl}
        AD_1 & \leq & d_H(\Proj_{T_\vx S_i}S_i(\delta_{k'}),T_\vx S_i(\delta_{k'}))\cdot \delta_{k'} \\
        & \leq & d_H(L_{\vx}(\delta_{k'}),\mdn{T}_\lambda(\delta_{k'}))\cdot \delta_{k'} \\
        & = & \frac{d_H(L_{\vx}(\delta_{k'}),\mdn{T}_\lambda(\delta_{k'}))}{r_{k'}} \, r_{k'} \delta_{k'} \\
        & \leq & \tan\theta_{k'} r_{k'}\delta_{k'}.
      \end{array}
    \end{equation}

  \item \emph{Area difference near $T_\vx S_i(\delta_{k'})\cap \partial Cyl(\delta_{k'}, r_{k'})$:} \\
  \smallskip
  \\
    This area difference is caused by the distance from $\Proj_{T_\vx S_i} S_i(\delta_{k'})\cap \partial Cyl(\delta_{k'}, r_{k'})$ to $T_\vx S_i(\delta_{k'})\cap \boundary Cyl(\delta_{k'}, r_{k'})$:
    \begin{align*}
      &\ \ \ \  d^{\partial Cyl(\delta_{k'},r_{k'})}_H(\Proj_{T_\vx S_i} (S_i(\delta_{k'},\theta_{k'})\cap \partial Cyl(\delta_{k'}, r_{k'})), T_\vx S_i(\delta_{k'},\theta_{k'})\cap \partial Cyl(\delta_{k'}, r_{k'})) \\
      &\leq r_{k'} - \sqrt{r_{k'}^2 - d_H(S_i(\delta_{k'},\theta_{k'})\cap Cyl(\delta_{k'}, r_{k'})), T_\vx S_i(\delta_{k'},\theta_{k'})\cap Cyl(\delta_{k'}, r_{k'})) }\\
      &\leq r_{k'} - \sqrt{r^2_{k'} - (r_{k'}\tan\theta_{k'})^2}\\
      & = \frac{1-(1-\tan^2\theta_{k'})}{1+\sqrt{1-\tan^2\theta_{k'}}} r_{k'} \\
      & \leq (\tan^2\theta_{k'})r_{k'}. 
    \end{align*}

Therefore the area difference near $T_\vx S_i(\delta_{k'})\cap \partial Cyl(\delta_{k'}, r_{k'})$, denoted by $AD_2$, is given by 
  \begin{equation}\label{areadiff2}
    \begin{array}{rcl}
      AD_2 & \leq & d^{\partial Cyl(\delta_{k'},r_{k'})}_H(\Proj_{T_\vx S_i} (S_i(\delta_{k'})\cap \partial Cyl(\delta_{k'}, r_{k'})), T_\vx S_i(\delta_{k'})\cap \partial Cyl(\delta_{k'}, r_{k'}))\delta_{k'} \\
      & \leq &  (\tan^2\theta_{k'})r_{k'} \delta_{k'} \, .
    \end{array}
  \end{equation}

\end{itemize}

Hence, we conclude that the area difference between $T_\vx S_i(\delta_{k'},\theta_{k'})$ and $\Proj_{T_\vx S_i} S_i(\delta_{k'},\theta_{k'})$ is bounded above by the following inequality: 
  \begin{equation}\label{tangentprojectionerror}
    \begin{array}{rcl}
   |\Hd^2(T_\vx S_i(\delta_{k'},\theta_{k'})) - \Hd^2(\Proj_{T_\vx S_i} S_i(\delta_{k'},\theta_{k'}))| & \leq & 
   AD_1 + AD_2 \\
   & \leq & \tan\theta_{k'} r_{k'} \delta_{k'} + \tan^2\theta_{k'} r_{k'} \delta_{k'}.
   \end{array}
  \end{equation}

By triangle inequality, together with \cref{projectionerror,tangentprojectionerror}, we get inequalities: 
\begin{align}\label{tangentsurfacedifference}
  \begin{split}
    &\ \ \ \ |\Hd^2(S_i(\delta_{k'},\theta_{k'})) - \Hd^2(T_\vx S_i(\delta_{k'},\theta_{k'}))|\\
    & \leq |\Hd^2(S_i(\delta_{k'},\delta_{k'})) - \Hd^2(\Proj_{T_\vx S_i}(\delta_{k'},\delta_{k'}))| + |\Hd^2(\Proj_{T_\vx S_i})(\delta_{k'},\theta_{k'}) - \Hd^2(T_\vx S_i(\delta_{k'},\theta_{k'}))|\\
    &\leq \tan^2\theta_{k'}\cdot r_{k'}\delta_{k'} ~+~ \tan\theta_{k'} r_{k'}\delta_{k'} + (\tan^2\theta_{k'})r_{k'}\delta_{k'} = (2\tan^2\theta_{k'} + \tan\theta_{k'}) r_{k'}\delta_{k'}.
  \end{split}
\end{align}

As there are $\nCur <\infty$ input currents, we can find a cylinder $Cyl(\vx,r,\delta)$ with $L_{\vx}$ as its central axis, and $r$ and $\delta$ as its radius and height, respectively, such that \cref{tangentsurfacedifference} and \cref{eqn-cdn1,eqn-cdn2,eqn-cdn3,eqn-cdn4} hold. Therefore, 
\begin{align}\label{totaltangentdifference}
  \left|\sum_{i=1}^\nCur \Hd^2(S_i(\delta,\theta) -\sum_{i=1}^\nCur \Hd^2(T_\vx S_i(\delta,\theta))\right| \leq \nCur [2\tan^2\theta\cdot r\delta + \tan\theta\cdot r\delta].
\end{align}

where $\delta = \epsilon r$, remembering that $\epsilon$ is some positive constant that will be determined later, and $\theta$ is a tiny angle which will also be determined later. \\

% -----------------------------------------------------------
\textbf{Step 3:} Assuming the $T_\vx S_i$'s do not form a book, we will find the improvement between $T_\vx S'_i$'s and $T_\vx S_i$'s inside the cylinder.\\
\smallskip
\\
Next we show that if $\mdn{T}_\lambda$ is the median, $T_\vx S_i(\delta,\theta)$'s must form a book.
Define $\vx_{t}$ and $\vx_{b}$ to be the intersections of $L_{\vx}(\delta,\theta)$ at the top and bottom of the cylinder (See Fig \ref{fig-newmdnTprL}) and $l^{t}_i$'s, $l^{b}_i$'s to be the segments connecting $\vx_{t}, \vx_{b}$ and $p_i^t$'s, $p_i^b$'s, where $p_i^t$'s and $p_i^b$'s are the intersections of $T_\vx S_i(\delta)$'s with the boundaries of the top and bottom of the cylinder $Cyl(r,\delta)$. 

\medskip
If the $T_\vx S_i(\delta)$'s do not form a book, the unit vectors from $\vx_{t}$ to $p_i^{t}$'s and $\vx_{b}$ to $p_i^{b}$'s will not sum up to $0$.
Define $\vx^{\mbox{opt}}_{t}, \vx^{\mbox{opt}}_{b}$ to be the median points for the $p_i^{t}$'s and $p_i^{b}$'s, and define $l^{\mbox{opt},t}_i, l^{\mbox{opt},b}_i$ to be the line segments between $\vx^{\mbox{opt}}_{t}, \vx^{\mbox{opt}}_{b}$ and the $p^{t}_i$'s, $p_i^{b}$'s respectively. By the properties of the median of a collection of points, we get that  
\begin{align*}
  \beta = \sum_{i=1}^\nCur l^{t}_i - \sum_{i=1}^\nCur l_i^{\mbox{opt},t} = \sum_{i=1}^\nCur l^{b}_i - \sum_{i=1}^\nCur l_i^{\mbox{opt},b} > 0.
\end{align*}

Moreover, $\beta$ is comparable to $r$, i.e., $\beta = O(r) \cdot r$ where $O(r) > \alpha > 0$.
Therefore there exists $\vx'_{t}$ and $\vx'_{b}$ such that 
\begin{align*}
  \sum_{i=1}^\nCur l^{t}_i - \sum_{i=1}^\nCur (l')_i^{t} = \sum_{i=1}^\nCur l^{b}_i - \sum_{i=1}^\nCur (l')_i^{b}>\frac{\beta}{2},
\end{align*} 
where the $(l')^{t}_i$'s and $(l')^{b}_i$'s connect $\vx'_{t}, \vx'_{b}$ to the $p^{t}_i$'s and $p_i^{b}$'s respectively. This shows that by replacing $L_{\vx}(\delta,\theta)$ with the segment connecting $\vx'_{t}$ and $\vx'_{b}$, denoted as $L'_{\vx}(\delta,\theta)$, the area improvement is 
\begin{align}\label{areaimprovement}
  \mbox{Area}^- = \left(\sum_{i=1}^\nCur l^{t}_i - \sum_{i=1}^\nCur (l')_i^{t}\right) \delta ~>~ \frac{\beta\delta}{2} \,.
\end{align}

% -----------------------------------------------------------
\textbf{Step 4:} Define the new median. \\

From Step 3, we know that replacing $L_{\vx}(\delta)$ with $L'_{\vx}(\delta)$ can improve the area.
However, that is only the improvement inside the interior of the cylinder $Cyl(\vx,r,\delta)$, and we still need to consider some extra costs when replacing $\mdn{T}_\lambda$ with a new median $\mdn{T}'_\lambda$. Define $\mdn{T}'_\lambda$ as follows:
\begin{itemize}
  \item inside the interior of $Cyl(\vx,r,\delta)$, $\mdn{T}'_\lambda = L'_{\vx}(\delta,\theta)$;
  \item at the top (bottom) of $Cyl(\vx,r,\delta)$, $\mdn{T}'_\lambda = \vx^{\lambda}_t \vx'_{t}$ ($\mdn{T}'_\lambda = \vx^{\lambda}_t\vx'_{b}$), where $\vx_{t}^{\lambda}$ ($\vx_{b}^{\lambda}$) is the intersection of $\mdn{T}_\lambda$ and the top (bottom) of $Cyl(\vx,r,\delta)$, and $\vx^{\lambda}_t\vx'_{t}$ ($\vx^{\lambda}_t\vx'_{b}$) is the line segment from $\vx^{\lambda}_t$ ($\vx^{\lambda}_t$) to $\vx'_{t}$ ($\vx'_{b}$) with orientation from $\vx^{\lambda}_t$ ($\vx^{\lambda}_b$) to $\vx'_{t}$ ($\vx'_{b}$); and
  \item outside $Cyl(\vx,r,\delta)$, $\mdn{T}'_\lambda = \mdn{T}_\lambda$. 
\end{itemize} 

\begin{figure}[htp!]
  \centering
  \scalebox{0.5}{
    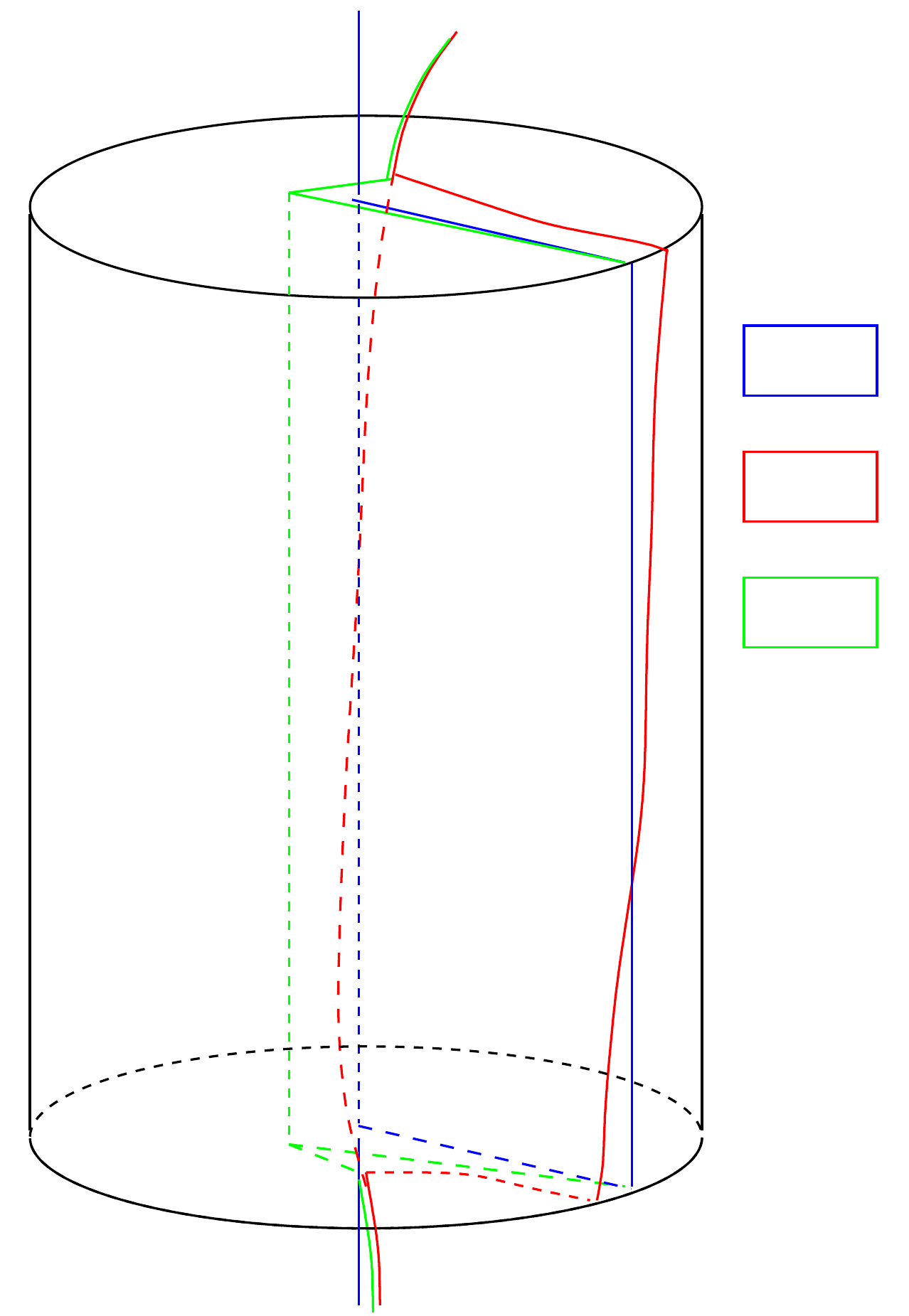}
  \caption{\label{fig-newmdnTprL} New median $\mdn{T}'_\lambda$.}
\end{figure}

After this replacement, the new $S'_i$ that spans $\mdn{T}'_\lambda$ and input currents $T_i$ are defined as follows:
\begin{itemize}
  \item inside the interior of $Cyl(\vx,r,\delta)$, $S'_i$ will be $T_\vx S'_i(\delta,\theta)$, which is the replacement of the $T_\vx S_i(\delta,\theta)$ that has $L'_{\vx}(\delta,\theta)$ and $(l')^t_{i}$ as its height and width respectively;
  \item at the top (bottom) of $Cyl(\vx,r,\delta)$, $S'_i$ is the region enclosed by $S_i(\delta,\theta)\cap Cyl^{t}(\vx,r,\delta)$ ($S_i(\delta)\cap C^{b}(\vx,r,\delta)$), $\vx^{\lambda}_t\vx'_{t}$ ($\vx^{\lambda}_t\vx'_{b}$), $(l')^t_{i}$ ($(l')^b_{i}$) and $\partial Cyl^{t}(\vx,r,\delta)$ ($\partial C^{b}(\vx,r,\delta)$), where $Cyl^{t}(\vx,r,\delta)$ is the top (bottom) circle of $Cyl(\vx,r,\delta)$ (See Fig \ref{fig-newmdndtls});
  \item on the side of $Cyl(\vx,r,\delta)$, $S'_i$ is the region enclosed by the $\partial Cyl^{t}(\vx,r,\delta)$, $\partial C^{b}(\vx,r,\delta)$, $\partial S_i\cap Cyl^{s}(\vx,r,\delta)$ and $\partial T_\vx S'_i(\delta,\theta)\cap Cyl^{t}(\vx,r,\delta)$ where $Cyl^{t}(\vx,r,\delta)$ is the cylindrical side of the $Cyl(\vx,r,\delta)$; and
  \item $S'_i = S_i$ elsewhere. 
\end{itemize}

% ----------------------------------------------------------
\textbf{Step 5:} Find the error for replacement on the top and bottom of the cylinder. \\

For each $S_i$, the error on $Cyl^{t}(\vx,r,\delta)$ is less than the whole area of $Cyl^{t}(\vx,r,\delta)$, and there are $\nCur$ input currents, so the total error is given by $\nCur \pi r^2$.
The same argument works for the bottom.
Therefore, the cost at the top and bottom of $Cyl(\vx,r,\delta)$ together is 
\begin{align}\label{topbottomcost}
 \mbox{Cost}_1 <  2\nCur \pi r^2 \,.
\end{align}

% -----------------------------------------------------------
\textbf{Step 6:} Find the error for replacement on the side of the cylinder.\\

For each $S_i$, because it satisfies \cref{eqn-cdn2}, i.e., it stays inside $Cyl_S(\vx,r,\delta)$, the error is contained in the band centered at $T_\vx S_i(\delta,\theta) \cap Cyl^{s}(\vx,r,\delta) $ with width $2 r \theta$.
Therefore the total cost on the side of $Cyl(\vx,r,\delta)$ satisfy the following bound
\begin{align}\label{sidecost}
  \mbox{Cost}_2 ~\leq~ 2\nCur r\theta \delta ~\leq~ 2\nCur \tan\theta \cdot r\delta \,.
\end{align}

% -----------------------------------------------------------
\textbf{Step 7:} Compare the improvement and the costs.\\

The improvement between $\mdn{T}'_\lambda$ and $\mdn{T}_\lambda$ happens inside $Cyl(\vx,r,\delta)$ (see \cref{areaimprovement,totaltangentdifference}). By the triangle inequality, the total improvement, $I$, is bounded below as follows:
\begin{align}\label{totalimprovement}
  \begin{split}
    I &\geq \left|\sum_{i=1}^\nCur (\Hd^2(S_i(\delta,\theta)) -\Hd^2(S'_i(\delta,\theta)))\right|\\
    & \geq \left|\sum_{i=1}^\nCur (\Hd^2(T_\vx S_i(\delta,\theta)) -\Hd^2(S'_i(\delta,\theta)))\right| - \left|\sum_{i=1}^\nCur \Hd^2(S_i(\delta,\theta) -\sum_{i=1}^\nCur \Hd^2(T_\vx S_i(\delta,\theta))\right| \\
    & \geq \frac{\beta\delta}{2} - \nCur [2\tan^2\theta\cdot r\delta + \tan\theta \cdot r\delta].
  \end{split}
\end{align}

\noindent The total cost, $C$, is the sum of costs in \cref{topbottomcost,sidecost}, which is 
\begin{align}\label{totalcost}
  \begin{split}
    C & = \mbox{Cost}_1 + \mbox{Cost}_2 \leq 2\nCur \pi r^2 + 2\nCur \tan\theta \cdot r \delta.
  \end{split}
\end{align}

\noindent Combining \cref{totalimprovement,totalcost}, the net improvement will be 
\begin{align}\label{netimprovement}
  \mbox{Net}_I = I - C = \frac{\beta\delta}{2} - \nCur [2\tan^2\theta\cdot r\delta + \tan\theta\cdot r\delta] - 2\nCur \pi r^2 - 2\nCur \tan\theta \cdot r \delta \,.
\end{align}

If $\mbox{Net}_I > 0$, then replacing the old median $\mdn{T}_\lambda$ with the new median $\mdn{T}'_\lambda$, will end up reducing the flat norm distance, which contradicts the fact that $\mdn{T}_\lambda$ is the median.
So it is left to show that we may choose the appropriate $r, \delta, \epsilon$ ($r = \epsilon \delta$) and $\theta$ to make $\mbox{Net}_I$ positive. Indeed, 
\begin{align}\label{comparison}
  \begin{split}
    \mbox{Net}_I & \geq \frac{\beta\delta}{2} - \nCur [2\tan^2\theta\cdot r\delta + \tan\theta\cdot r\delta] - 2\nCur \pi r^2 - 2\nCur \tan\theta \cdot r \delta\\
    & = \frac{\beta\delta}{2} - \nCur [2\tan^2\theta \cdot \epsilon \delta^2] -2\nCur\pi (\epsilon \delta)^2 -3\nCur\tan \theta \cdot \epsilon \delta^2\\
    & = \delta \left(\frac{\beta}{2} - \nCur [2\epsilon \delta \tan^2\theta] - 2\nCur\pi \epsilon^2 \delta -3\nCur \epsilon \delta \tan\theta \right)\\
    & > \delta \left(\frac{\beta}{2} - \nCur [2\epsilon \delta \tan^2\theta] - \epsilon^2  -\frac{\epsilon \tan\theta }{\pi} \right)~~\mbox{since } \delta< \frac{1}{3\nCur\pi} <1, \\
    & = \delta \left(\frac{\beta}{2} - \nCur [2 \epsilon \delta \lambda^2] - \epsilon^2  -\frac{\lambda \epsilon}{\pi} \right)~~  \mbox{since } \tan\theta = \lambda,\\
    & = \delta \left(\frac{\beta}{2} - 2\nCur\epsilon \delta \lambda^2 - \epsilon^2  -\frac{\lambda \epsilon}{\pi} \right) \,.
  \end{split}
\end{align}

Define the quadratic function
\begin{align}
  p(\lambda) =  - 2\nCur\epsilon \delta \lambda^2 -\frac{\lambda \epsilon}{\pi} - \epsilon^2   + \frac{\beta}{2} \,.
\end{align}
Its discriminant is 
\begin{align*}
  \Delta & = \left(\frac{\epsilon}{\pi}\right)^2 + 4\left(2\nCur\epsilon\delta\right)\left(\frac{\beta}{2}-\epsilon^2\right)\\
  &> \left(\frac{\epsilon}{\pi}\right)^2 + 4\left(2\nCur\epsilon\delta\right)\left(\frac{cr}{2}-\epsilon^2\right)\\
  & = \left(\frac{\epsilon}{\pi}\right)^2 + 4\left(2\nCur\epsilon\delta\right)\left(\frac{c\epsilon \delta}{2}-\epsilon^2\right) \,.
\end{align*}

Picking $\epsilon < c\delta/2$ gives us that $\Delta> \epsilon/\pi\,$.
And moreover, as long as
\[ 0 ~<~ \lambda ~<~ \frac{\frac{\epsilon}{\pi}-\Delta}{-4n\epsilon} ~=~ \frac{\Delta - \frac{\epsilon}{\pi}}{4n\epsilon} \,,\]
we get $p(\lambda)>0$.
Hence $\mbox{Net}_I>0$, which means $\mdn{T}'_\lambda$ being the median will decrease the flat norm distance  $\mdn{T}_\lambda$.
This contradicts the fact that $\mdn{T}_\lambda$ is the median.  
\end{proof}

%\clearpage
%\begin{center}
%\Huge Median Shapes on Simplicial Complexes
%\end{center}

\section{Median Shapes on Simplicial Complexes: Preliminaries} \label{sec-prelim}
\label{sec:simp-prelim}

We consider the median shape problem under the settings of a {\em finite} simplicial complex. We had previously studied the flat norm under simplicial settings \cite{IbKrVi2013}.
Motivated by this approach, it is natural to consider the problem of defining, and more importantly, efficiently computing average shapes under the simplicial setting.
The input shapes, which are represented as integral $p$-currents in the continuous setting, are now represented as $p$-chains in a simplicial complex $K$ of dimension $q$ (for $q \geq p+1$).
We restrict our attention to the case where $K$ is finite, which also implies that the input chains are finite.

Let $\sigma_i$ for $i=1,\dots,m$ denote the $p$-simplices and $\tau_j$ for $j=1,\dots,n$ denote the $(p+1)$-simplices of $K$.
To compute the simplicial flat norm of the integral current represented by a $p$-chain $\vt = \sum_i t_i \sigma_i$ with $t_i \in \Z$, we consider candidate $(p+1)$-chains $\vs = \sum_j s_j \tau_j$ with $s_j \in \Z$, which defines the corresponding decomposition as $\vx = \sum_i x_i \sigma_i = \vt - \boundary_{p+1} \vs$.
Thus $\vx$ and $\vt$ are homologous $p$-chains, with $\vs$ being the $(p+1)$-chain defining the homology.
The flat norm decomposition is given by the pair of chains $(\vx,\vs)$ that minimizes the sum of weighted volumes of these chains, i.e.,
\[ \sum_{i=1}^m \vol_p(\sigma_i) \, |x_i| ~+~ \lambda \sum_{j=1}^n \vol_{p+1}(\tau_j) \, |s_j|, \]
where $\vol_p(\sigma_i)$ and $\vol_{p+1}(\tau_j)$ are the $p$-dimensional volume of $\sigma_i$ and the $(p+1)$-dimensional volume of $\tau_j$.
We note that $\vol_p(\sigma)$ is equivalent to the mass $\mass(\sigma)$ of the $p$-simplex $\sigma$.
Recall that $\lambda \geq 0$ is the scale parameter.
The boundary operator $\boundary_{p+1}$ is captured by the $(p+1)$-boundary matrix $[\boundary_{p+1}]$ of $K$, which we will denote in brief as $B$.
Notice that $B \in \{-1,0,1\}^{m \times n}$, with $B_{ij} = \pm 1$ when $\sigma_i$ is a face of $\tau_j$ (denoted $\sigma_i \preceq \tau_j$), and is zero otherwise.
This nonzero number is $+1$ if the orientations of $\sigma_i$ and $\tau_j$ agree, and is $-1$ when they are opposite.

We showed that the flat norm problem is NP-hard \cite{IbKrVi2013}.
We cast this problem as an integer linear optimization problem (IP). 
Notice that integer solutions are required, as opposed to real ones, since homology is defined over $\Z$.
Instances of this IP could take exponential time to solve in the worst case.
But an IP can be solved in polynomial time by solving its linear programming (LP) relaxation when its constraint matrix is totally unimodular, i.e., when each of its subdeterminants is in $\{0,\pm 1\}$ \cite{Schrijver1986}.
We showed that the constraint matrix of the flat norm IP is totally unimodular if and only if the boundary matrix $B$ is so.
And $B$ is totally unimodular if and only if $K$ has no relative torsion in dimension $p$.
This condition is satisfied, for instance, when $K$ triangulates a compact, orientable $(p+1)$-manifold, or when it is a $(d+1)$-complex in $\R^{d+1}$ \cite{IbKrVi2013}.

\section{Simplicial Median Shape and Integer Linear Optimization} \label{sec-medshpIP}
\label{sec:simplicial}
Our goal is to study the median shape problem in the simplicial setting, and to formulate it as an integer linear optimization problem.
At the same time, it is not immediately clear whether we would be able to utilize total unimodularity of the boundary matrix $B$, when available.
We present an integer program (IP) for the simplicial median shape problem.
While we are not able to prove that its constraint matrix is totally unimodular when $B$ is so, the LP relaxation of this IP always had an integer optimal solution in all our computational experiments.
Based on this evidence, we believe that the LP relaxation of the median shape IP has integer optimal solution in the case where the volumes of simplices are their default Euclidean masses.

\subsection{Median Shape as an Integer Program} \label{ssec-IPformln}

Let $C_p(K)$ denote the group of $p$-chains of the simplicial complex $K$.
Consider the set of $\nCur$ currents modeled by $p$-chains $\vt_1, \dots, \vt_{\nCur} \in C_p(K)$.
The simplicial median shape $\mdn{\vt}$ is defined as a $p$-chain $\vt \in C_p(K)$ for which the sum of the flat distances between $\vt$ and $\vt_1,\dots, \vt_{\nCur}$, i.e.,
\[ \sum_{h=1}^{\nCur} \rho(\vt,\vt_h) = \sum_{h=1}^{\nCur} \F_{\lambda}(\vt-\vt_h) \]
is minimized:

\begin{align}
  \nonumber
  %\hspace*{-0.0785in}
  \mdn{\vt} = &\argmin_{\vt \in C_p(K)} \left\{ \sum_{h=1}^{\nCur} \F_{\lambda}(\vt,\vt_h) \right\} \\ 
  %   = & \argmin_{\vt \in C_p(K)} \hspace*{-0.03in} \left\{ \hspace*{-0.03in} \sum_{h=1}^{\nCur} \hspace*{-0.05in} \left(\hspace*{-0.02in} \sum_{i=1}^m \vol_p(\sigma_i)|r_{hi}| + \lambda \sum_{j=1}^n \vol_{p+1}(\tau_j)|s_{hj}| \hspace*{-0.02in} \right) \hspace*{-0.04in} \bigg| \vt - \vt_h = \vr_h + \boundary_{p+1} \vs_h; \vt,\vr_h \in \Z^m, \vs_h \in \Z^n,\forall h \right\},
  = & \argmin_{\vt,\vr_h \in C_p(K), \, \vs_h \in  C_{p+1}(K)} \left\{ \sum_{h=1}^{\nCur} \left( \sum_{i=1}^m \vol_p(\sigma_i)|r_{hi}| + \lambda \sum_{j=1}^n \vol_{p+1}(\tau_j)|s_{hj}| \right) ~ \bigg| \right.\mbox{\hspace*{1in}} \\
   & \left. \mbox{\hspace*{1.5in}} \vt - \vt_h = \vr_h + \boundary_{p+1} \vs_h; \vt,\vr_h \in \Z^m, \vs_h \in \Z^n,\forall h ~~\Bigg\}, \right. \nonumber
\end{align}
where $\vt, \vr_h \in C_p(K)$  and $\vs_h \in C_{p+1}(K)$, and the constraints capture the flat norm decomposition of $\vt-\vt_h$ for each $h$.
Note that $r_{hi}$ is the $i$th component of $\vr_h$, with similar notation used for $s_{jh}$ and $\vs_h$.
With the volumes of the simplices taken as $\vol_p(\sigma_i) = w_i$ and $\vol_{p+1}(\tau_j)=v_j$, we can cast the median shape problem as the following integer optimization problem.
\begin{equation}
  \label{eq-optprobMedShp}
  \begin{aligned}
    \mbox{minimize } & \sum_{h=1}^{\nCur} \left( \sum_{i=1}^m w_i \abs{r_{hi}} + \lambda  \sum_{j=1}^n v_j \abs{s_{hj}} \right)  \\
     \mbox{subject to } & ~~\vt - \vt_h = \vr_h + B \vs_h, ~h=1,\dots,\nCur \\
    & ~~\vt \in \Z^m,~\vr_h \in \Z^m, \vs_h \in \Z^n,\,h=1,\dots,\nCur.
  \end{aligned}
\end{equation}
The objective function is piecewise linear, and we can linearize the same using extra variables \cite[Pg.~18]{BeTs1997}, and obtain the following integer {\em linear} optimization problem when $w_i, v_j \geq 0$ for all $i,j$.

\begin{equation}
  \label{eq-IPMedShp}
  \begin{aligned}
    \mbox{minimize } & \sum_{h=1}^{\nCur} \left( \sum_{i=1}^m w_i (r^+_{hi} + r^-_{hi}) + \lambda \sum_{j=1}^n v_j (s^+_{hj}+s^-_{hj}) \right)\\
    \mbox{subject to } & ~~\vt - \vt_h = (\vr^+_h - \vr^-_h) + B (\vs^+_h - \vs^-_h), ~h=1,\dots,\nCur, \\
    & ~~ \vr^+_h, \vr^-_h \geq \vzero,~ \vs^+_h,\vs^-_h \geq \vzero,~~h=1,\dots,\nCur, \\
    & ~~\vt \in \Z^m,~~\vr^+_h, \vr^-_h \in \Z^m, \,\vs^+_h,\vs^-_h \in \Z^n,\,h=1,\dots,\nCur.
  \end{aligned}
\end{equation}
When constructing the integer optimization formulation for the median shape {\em with mass regularization} (\cref{eq:defmsregmdn}), we replace the variable vector $\vt$ with a pair of nonnegative variable vectors $\vt^{\pm}$.
In particular, each occurrence of $\vt$ in the constraints is replaced by $\vt^+ - \vt^-$, and the term $\vw^T (\vt^+ + \vt^-)$ is added to the objective function.
With this extension in mind, we work with this pair $\vt^{\pm}$ in our formulation, but do not include the extra terms in the objective function for the default median shape problem.

We obtain the linear programming relaxation of this integer program by relaxing, i.e., ignoring, the integrality constraints.
We are interested in instances for which this linear program is guaranteed to have an integer optimal solution, in which case we can solve the median shape problem in polynomial time.
To this end, we explore when the constraint matrix $A$ of this linear program  transformed to the standard form $A \vx = \vb$ (with $\vx \geq \vzero$) is totally unimodular.
We rewrite the linear programming relaxation (denoted as LP henceforth) of the integer program in Equation (\ref{eq-IPMedShp}) in this standard form, with the structure of the variable vector $\vx$ detailed in the nonnegativity constraint. Unspecified entries are all zeros.

%\begin{equation}%\LeftEqNo
\begin{align}
  \label{eq-LPStdFm}
  %\begin{aligned}
    %    \mbox{minimize } & \,\begin{bmatrix} ~~~~~\,& ~~~~~\,& [~~\vw ~~~~ \vw ~~~~ \lambda \vv ~~~ \lambda \vv ] & [~~\vw ~~~~ \vw ~~~~ \lambda \vv ~~~ \lambda \vv ] & \cdots & [~~\vw ~~~~ \vw ~~~~ \lambda \vv ~~~ \lambda \vv ] \end{bmatrix} \vx\\
   \mbox{min} & \,\begin{bmatrix} ~~~\,& ~~~\,& ~[~~\vw~~~ \vw ~~~ \lambda \vv ~ \lambda \vv ] & [~~\vw~~~ \vw ~~\lambda \vv ~~ \lambda \vv ] & \cdots & [~\vw ~~~ \vw ~~~ \lambda \vv ~~ \lambda \vv ] \end{bmatrix} \vx \nonumber \\
    \mbox{ s.t.}   \\
    &
    \begin{bmatrix} \nonumber
        \begin{bmatrix} I & -I \end{bmatrix} &  \begin{bmatrix} -I & I & -B & B \end{bmatrix} \\
        \begin{bmatrix} I & -I \end{bmatrix} &  & \begin{bmatrix} -I & I & -B & B \end{bmatrix} \\
        \vdots &  & & \ddots \\
        \begin{bmatrix} I & -I \end{bmatrix} & & & & \begin{bmatrix} -I & I & -B & B \end{bmatrix} \\
      \end{bmatrix}
      \vx =
      \begin{bmatrix} \vt_1 \\ \vt_2 \\ \vdots \\ \vt_{\nCur} \end{bmatrix} \nonumber \\
      \nonumber \\
      %      & \,\begin{bmatrix} ~~\vt^+ & \vt^- & ~~~\vr^+_1 ~~~ \vr^-_1 ~~~~ \vs^+_1 ~~~ \vs^-_1 & ~~~~\vr^+_2 ~~~ \vr^-_2 ~~~~ \vs^+_2 ~~~ \vs^-_2 & \cdots & ~~\vr^+_{\nCur} & \vr^-_{\nCur} & \vs^+_{\nCur} & \vs^-_{\nCur} \end{bmatrix} ~~~~~\geq ~~\vzero. \\
      & \,\begin{bmatrix} ~\vt^+ & \vt^- & ~~~\vr^+_1 ~~ \vr^-_1 ~~~ \vs^+_1 ~~~ \vs^-_1 & ~~\vr^+_2 ~~ \vr^-_2 ~~~ \vs^+_2 ~~~ \vs^-_2 & ~\cdots & ~~\vr^+_{\nCur} & \vr^-_{\nCur} & \vs^+_{\nCur} & \vs^-_{\nCur} \end{bmatrix} ~~\geq ~~\vzero. \nonumber 
  %\end{aligned}
\end{align}
%\end{equation}

So as to avoid clutter, notice that we have avoided transposing the individual component vectors, e.g., $\vw$, $\vt^+$, etc., in both the objective function vector as well as in the variable vector $\vx$ in the nonnegativity constraint.

\subsection{Total Unimodularity and the Median Shape LP} \label{ssec-TUMedShpLP}

We study the structure of the constraint matrix $A$ of the median shape LP in Equation (\ref{eq-LPStdFm}) with respect to the total unimodularity of the boundary matrix $B$.
To this end, we utilize several standard matrix operations that preserve total unimodularity, which we present collectively in Lemma \ref{lem-stdTUops}.
But to construct $A$ from $B$, we have to use a series of these operations along with one other matrix operation, which is not guaranteed to preserve total unimodularity.

%\clearpage
%\smallskip
%\begin{flemma}
\begin{lemma}
  \label{lem-stdTUops}
        {\rm {\bfseries (\cite[Pg.~280]{Schrijver1986})}} 
     Total unimodularity of a matrix is preserved under the following operations.
     \begin{enumerate}
       \item \label{lem-TUswap} Permuting rows or columns.
       \item \label{lem-TUtrsps} Taking the transpose.
       \item \label{lem-TUmltby-1} Multiplying a row or column by $-1$.
       \item \label{lem-TUaddzerosgltn} Adding a row or column of all zeros, or adding a row or column with one nonzero that is $\pm 1$.
       \item \label{lem-TUrep} Repeating a row or column.
     \end{enumerate}
\end{lemma}
%\end{flemma}

\noindent The extra operation we need is a composition involving the identity matrix, which we define as the {\em $I$-sum}.

\begin{definition}
  \label{def-Isum}
  For an integer $\nCur \geq 1$, the {\bf ${\bm \nCur}$-fold ${\bm I}$-sum} of an $m \times n$ matrix $A$ is the $(m\nCur+n) \times n\nCur$ matrix
  \begin{equation} \tag{$I$-sum} \label{eq-Isum}
    \Isum_{\nCur} A :=
    \begin{bmatrix}
      I & I & \cdots & I \\
      A &    \\
      & A  \\
      &   & \ddots \\
      &   &        & A
    \end{bmatrix},
  \end{equation}
  where $I$ is the $n \times n$ identity matrix, $\nCur$ copies of which  are included in the top row. Unspecified entries are zero.
\end{definition}

\smallskip
Several versions of connected sums are already known in the context of total unimodularity.
Schrijver presents $\nCur$-sums for $\nCur=1,2,3$ \cite[Pg.~280]{Schrijver1986}.
In a related context, $\nCur$-sums are used in the decomposition of regular matroids \cite{Truemper1992,Ox2006,Co2001}.
At the same time, our $I$-sum is different from these matrix connected sums, and also from (the matrix equivalents of) the matroid $\nCur$-sums.
But unlike the $\nCur$-sums which preserve total unimodularity, the \ref{eq-Isum} may not do so.

\smallskip
\begin{lemma}
  \label{lem-IsumNotTU}
  The $\nCur$-fold $I$-sum is not guaranteed to preserve total unimodularity.
\end{lemma}
\smallskip

\begin{proof}
  We show by example that an \ref{eq-Isum} of a totally unimodular matrix is itself not totally unimodular.
  Consider the following $3 \times 4$ matrix $A$, and its $2$-fold \ref{eq-Isum} $\Isum_2 A$, which is a $10 \times 8$ matrix.
  The elements of a particular $6 \times 6$ submatrix $S$ of $\Isum_2 A$ are shown in bold.
  $S$ is formed using rows $1,4,5,7,8,9$ and columns $1,2,4,5,7,8$ of its parent matrix.
  We present $S$ after rearranging its rows and columns in the order $1,7,5,4,8,9$ and $1,2,4,8,7,5$, respectively, from $\Isum_2 A$.
  
  \[
    A =
    \begin{bmatrix}
      0 & 1 & \m 1 & 1 \\
      1 & 0 &    1 & 0 \\
      1 & 1 &    0 & 0 
    \end{bmatrix}, ~~~
    \Isum_2 A =
    \begin{bmatrix}
      \bm{1} & \bm{0} & 0 & \bm{0} & \bm{1} & 0 & \bm{0} & \bm{0} \\
      0 & 1 & 0 & 0 & 0 & 1 & 0 & 0 \\
      0 & 0 & 1 & 0 & 0 & 0 & 1 & 0 \\
      \bm{0} & \bm{0} & 0 & \bm{1} & \bm{0} & 0 & \bm{0} & \bm{1} \\
      \vspace*{-0.1in} \\
      \bm{0} & \bm{1} & \m 1 & \bm{1} & \bm{0} & 0 & \bm{0} & \bm{0} \\
      1 & 0 &    1 & 0 & 0 & 0 & 0 & 0 \\
      \bm{1} & \bm{1} &    0 & \bm{0} & \bm{0} & 0 & \bm{0} & \bm{0} \\
      \vspace*{-0.1in} \\
      \bm{0} & \bm{0} & 0 & \bm{0} & \bm{0} & 1 & \m \bm{1} & \bm{1} \\
      \bm{0} & \bm{0} & 0 & \bm{0} & \bm{1} & 0 &    \bm{1} & \bm{0} \\
      0 & 0 & 0 & 0 & 1 & 1 &    0 & 0 
    \end{bmatrix},~~\mbox{ and }~
    S =
    \begin{bmatrix}
      1 & 0 & 0 & 0 &    0 & 1 \\
      1 & 1 & 0 & 0 &    0 & 0 \\
      0 & 1 & 1 & 0 &    0 & 0 \\
      0 & 0 & 1 & 1 &    0 & 0 \\
      0 & 0 & 0 & 1 & \m 1 & 0 \\
      0 & 0 & 0 & 0 &    1 & 1 
    \end{bmatrix}.
  \]
  It can be checked that $A$ is totally unimodular.
  But $\det S = -2$, showing that $\Isum_2 A$ is not totally unimodular.
  In fact, $S$ is a {\em M\"obius cycle matrix} (MCM) of size $6$ (after scaling three rows/columns by $-1$) \cite{DeHiKr2011}, whose determinants are equal to $2$ in absolute value.
\end{proof}

We can construct the constraint matrix $A$ in \cref{eq-LPStdFm} using a sequence of these matrix operations.
First, we construct the matrix $B' = \begin{bmatrix} -I & I & -B & B \end{bmatrix}$ from $B$ by repeating all columns of $B$ and scaling these repeated columns by $-1$ to get $-B$, and then adding the $2m$ columns of $I$ and $-I$.
These are the operations \ref{lem-TUrep}, \ref{lem-TUmltby-1}, and \ref{lem-TUaddzerosgltn} in \cref{lem-stdTUops}.
We then construct the $\nCur$-fold \ref{eq-Isum} of the transpose of $B'$ to get $\Isum_{\nCur} B'^T$, and then take its transpose.
Finally, we repeat the columns formed by the $\nCur$ copies of the $m$-identity matrix, scale these columns by $-1$ to get $\nCur$ copies of $-I$, and swap the columns corresponding to the $\nCur$ copies of $-I$ and those corresponding to the $\nCur$ copies of $I$.
Apart from the \ref{eq-Isum}, we used the operations \ref{lem-TUtrsps} and \ref{lem-TUswap} in \cref{lem-stdTUops} in the previous steps.

All operations used in constructing $A$ from $B$ preserve total unimodularity, except the \ref{eq-Isum}.
As such, we are not guaranteed integer solutions for the median shape LP even when $B$ is totally unimodular.
Nonetheless, we have always obtained integer optimal solutions for all instances of the median shape LP we tried (see Section \ref{sec-compexpr}).

\subsection{Generalizations of the Median Shape LP} \label{ssec-genMedShLP}

We can modify the median shape LP in \cref{eq-LPStdFm} to find a mass-regularized simplicial median shape.
We add $\mu\vw^T (\vt^+ + \vt^-)$ to the objective function, while the rest of the LP remains unchanged.
The scaling factor for the mass of $\vt$ is chosen as $\mu \geq 0$, and is typically taken to be smaller than $\lambda$.
The objective function vector thus gets the additional terms $[\mu\vw~~\mu\vw]$ in the beginning.

Another modification to the objective function lets us formulate the generalized {\em weighted} simplicial median shape problem, where we seek $\vt \in C_p(K)$ that minimizes
\begin{equation} \label{eq-WtdMdn}
 \sum_{h=1}^{\nCur} \alpha_h \rho(\vt,\vt_h) = \sum_{h=1}^{\nCur} \alpha_h \F_{\lambda}(\vt-\vt_h),
~~\mbox{ where } \alpha_h \geq 0 \,\forall h,
 ~\mbox{ and } \sum_{h=1}^{\nCur} \alpha_h = 1.
\end{equation}
Notice that when $\alpha_h=1$ (and the remaining $\alpha_i=0$), we get $\mdn{\vt}=\vt_h$.
As each of the $\alpha_h$'s varies from $0$ to $1$, we obtain each input chain and also a series of ``in between'' chains as the weighted median.

We set the objective function vector in \cref{eq-LPStdFm} as follows (again, we avoid transpose notation to avoid clutter):
\[
\vc = \begin{bmatrix}
  \vzero~ & ~\vzero~ & \alpha_1 [~\vw ~ \vw ~~ \lambda \vv ~~ \lambda \vv] & \alpha_2 [~\vw ~ \vw ~~ \lambda \vv ~~ \lambda \vv]  & \cdots & \alpha_{\nCur} [~\vw ~ \vw ~~ \lambda \vv ~~ \lambda \vv]
\end{bmatrix}.
\]
While we do need each $\alpha_h$ be nonnegative for the formulation to work, the correctness of the LP is independent of the requirement $\sum_h \alpha_h = 1$.
We use the latter observation in analyzing the complexity of the simplicial median shape problem (see below).

We could also compute a mass-regularized weighted simplicial median shape by replacing the first two zero vectors corresponding to $\vt^{\pm}$ with two copies of $\mu\vw$:
\begin{equation}
  \label{eq-MRWSMSP}
  \vc =
  \begin{bmatrix}
    [\mu \vw ~~\mu \vw] & \alpha_1 [~\vw ~ \vw ~~ \lambda \vv ~~ \lambda \vv] & \alpha_2 [~\vw ~ \vw ~~ \lambda \vv ~~ \lambda \vv]  & \cdots & \alpha_{\nCur} [~\vw ~ \vw ~~ \lambda \vv ~~ \lambda \vv]
  \end{bmatrix}.
\end{equation}

\subsubsection{Median shape on generalized spaces}
\label{sssec-medshpgenspc}
Yet another natural generalization permitted by the simplicial approach is to consider median shapes over simplicial complexes that are more general that the corresponding spaces specified in the continuous definition.
With input currents in $\R^d$, the median shape as well as the associated currents could live possibly in all of $\R^d$.
On the other hand, the simplicial median shape could be studied over simplicial complexes $K$ whose underlying spaces are nontrivial subspaces of $\R^d$, i.e., with nontrivial homology.
Notice that we do not have to modify the definition of the simplicial median shape in order to consider such $K$.
For instance, we could study the median shape of chains on the surface of a sphere or a torus, as we illustrate using computations (see \cref{sec-compexpr}).

\subsection{Complexity of Simplicial Median Shape} \label{ssec-cplxtySMS}

To analyze the computational complexity of the simplicial median shape problem (SMSP), we consider a decision version of the most general SMSP we have introduced, which is the mass-regularized weighted simplicial median shape problem (\mrwsmsp) --- see \cref{eq-MRWSMSP}.
We denote this problem as the {\em decision}-MRWSMSP, or \dmrwsmsp.
Consider $\nCur$ input $p$-chains $\vt_1,\dots,\vt_{\nCur}$, the $p$-chain $\vt$, and the $\nCur$ pairs of $p$- and $(p+1)$-chains $(\vr_1, \vs_1), \dots, (\vr_{\nCur}, \vs_{\nCur})$, all in $K$, such that $\vt - \vt_h = \vr_h + [\boundary_{(p+1)}(K)] \vs_h$ for each $h=1,\dots,\nCur$.
Then for given set of parameters $\valpha = [\alpha_1~\dots~\alpha_{\nCur}] \geq \vzero, \lambda \geq 0$, and $\mu \geq 0$, we define the following function:
\begin{equation}
  \label{eq-f}
  \begin{array}{ll}
    f_{(\valpha,\lambda,\mu)}(\vt,\vt_1,\dots,\vt_{\nCur}) & = ~\mu \left( \sum_{i=1}^m w_i |t_i| \right) \\
    & ~~+ \alpha_1 \left(\sum_{i=1}^m w_i |r_{1i}| + \lambda \sum_{j=1}^n v_j |s_{1j}| \right) + ~~ \dots \\
    & ~~+ \alpha_{\nCur} \left(\sum_{i=1}^m w_i |r_{ki}| + \lambda \sum_{j=1}^n v_j |s_{kj}| \right).
  \end{array}
\end{equation} 
 Notice that $f_{(\valpha,\lambda,\mu)}(\vt,\vt_1,\dots,\vt_{\nCur})$ corresponds to the objective function of the median shape LP (\cref{eq-LPStdFm}) with the coefficients for \mrwsmsp \ (\cref{eq-MRWSMSP}).
In particular, we do not require that $\sum_h \alpha_h = 1$.
Also, we assume all parameters involved, i.e., the entries of $\vw, \vv, \valpha$, as well as $\lambda$ and $\mu$, are rational.

In the {\em optimal homologous chain problem} (OHCP), we seek to find a chain with the minimal total weight in the same homology class as the input chain in a finite simplicial complex.
The (decision version of) OHCP is known to be NP-complete \cite[Theorem 1.4]{DuHi2011}.
We reduce OHCP to a special case of \dmrwsmsp \ with a single input chain, thus showing that \dmrwsmsp \ is NP-complete as well.
The default, i.e., optimization, version of MRWSMSP consequently turns out to be NP-hard.

\begin{definition}
  \label{def-dmrwsmsp}
  (\dmrwsmsp)
  Given $\nCur$ $p$-chains $\vt_1, \dots, \vt_{\nCur}$ in a finite $q$-dimensional simplicial complex $K$ (for $p \leq q-1$), nonnegative rational parameters $\valpha = [\alpha_1~\dots~\alpha_{\nCur}],\, \lambda, \mu$, and a rational number $f_0 \geq 0$,
  do there exist $\nCur$ pairs of $p$- and $(p+1)$-chains  $(\vr_1, \vs_1), \dots, (\vr_{\nCur}, \vs_{\nCur})$ and a $p$-chain $\vt$, all in $K$, such that  $f_{(\valpha,\lambda,\mu)}(\vt,\vt_1,\dots,\vt_{\nCur}) \leq f_0$, where $\vt - \vt_h = \vr_h + [\boundary_{p+1}(K)] \vs_h$ for $h=1,\dots,\nCur$?
\end{definition}
  
\smallskip
\begin{lemma}
  \label{lem-npcmplt}
  \dmrwsmsp \ is NP-complete, and \mrwsmsp \ is NP-hard.
\end{lemma}

\begin{proof}
  \dmrwsmsp \ lies in NP as we can compute $f_{(\valpha,\lambda,\mu)}(\vt,\vt_1,\dots,\vt_{\nCur})$ described in \cref{eq-f} in polynomial time when given the vectors $\vt$ and $(\vr_1, \vs_1),$ $\dots, (\vr_{\nCur}, \vs_{\nCur})$, all in $K$, satisfying the specified conditions.
  On the other hand, given an instance of the decision version of OHCP with input $p$-chain $\vt'$, we can reduce it to an instance of \dmrwsmsp \ as follows.
  We set $\nCur=1$, $\vt_1 = \vt'$, $\lambda=0$ and $\mu=1$ for the instance of \dmrwsmsp.
  Let $t'_{\rm max} = \max_{i=1}^m |t'_i|$ be the largest entry in $\vt'$ in absolute value, and let $w_{\rm max} = \max_{i=1}^m w_i$ be the largest weight of any $p$-simplex (we assume $w_i \geq 0$).
  We set $\alpha_1 = 2 m w_{\rm max} t_{\rm max} + 1$.
  This value of $\alpha_1$ insures that $\vr_1 = \vzero$ for nontrivial choices of $f_0$, giving $\vt = \vt_1 +  [\boundary_{(p+1)}(K)] \vs_1$, which is the required homology constraint of the OHCP.
  The result follows since OHCP is NP-complete.
\end{proof}

\begin{remark}
  Even though we have shown that \dmrwsmsp \ is NP-hard in general, the case for particular choices of the parameters $\valpha,\lambda,\mu$ could well be different.
  In fact, when $\mu \gg \lambda \gg 1$ and $\alpha_h = 1$ for all $h$, the problem becomes easy---the median shape is the empty chain in this case.
\end{remark}

\section{Computational Experiments} \label{sec-compexpr}
\label{sec:compute}

We present results from computational experiments on the simplicial median shape problem.
We solve the LP instances using CPLEX \cite{cplex} on a typical laptop machine.
We considered instances where the simplicial complex $K$ is a rectangle in $\R^2$ (i.e., its underlying space is homeomorphic to the closed $2$-disc), the surface of a sphere and a torus in $\R^3$, as well as the closed Euclidean ball in $\R^3$.
The chains considered were $1$-dimensional in these instances.
Thus the problem had a codimension of $1$ in all cases except in that of the $3$-ball, where the codimension was $2$.
The boundary matrix in question ($[\boundary_2(K)]$) is guaranteed to be totally unimodular for the codimension $1$ cases, but is typically not totally unimodular for the codimension $2$ case.
As we observed earlier, the constraint matrix of the median shape LP may not be totally unimodular even when the boundary matrix is so (see \cref{lem-IsumNotTU}).
Nonetheless, we obtained integer optimal solutions for the median shape LPs in {\em each} case.
Solving the median shape LPs took from a few seconds to several minutes, depending on the size of the simplicial complex considered.

\subsection{Instances in 2D} \label{ssec-2dinst}
\cref{fig-2D2crvs,fig-2D3crvs} show a mesh in 2D with $3851$ edges and $2510$ triangles.
We consider one set with two input $1$-chains (\cref{fig-2D2crvs}) and a second set with three $1$-chains (\cref{fig-2D3crvs}).
We show the mass-regularized median shape on the same mesh in each case.
We chose $\lambda=10^{-3}$ and $\mu = 10^{-5}$.

The median shape curve captures the intuition of the average of input curves ($1$-chains) in both cases.
In \cref{fig-2D2crvs}, the median curve stays in the middle of the two input curves all along, and agrees with the inputs in sections where they coincide.
With three input curves (\cref{fig-2D3crvs}), the median is composed of sections of whichever curve is in the middle (of the three) across the domain.
Note that for an odd number of input curves, it is not necessary that the median curve is always composed of pieces of input curves in the middle across the domain---it just happens to be this way for the instance in \cref{fig-2D3crvs} for the specific choices of $\lambda, \mu$, and mesh parameters.

\begin{figure}[ht!]
  \centering
  \includegraphics[width=\textwidth]{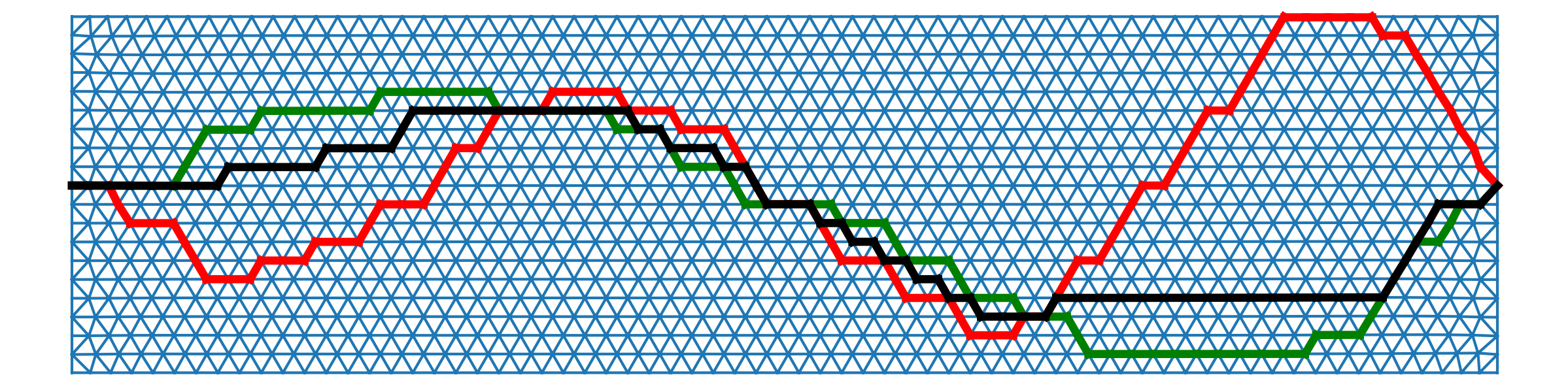}
  \caption{Simplicial median shape of two curves in 2D.
    The input curves are shown in green and red, while the median shape curve is shown in black.   \label{fig-2D2crvs}
    }
\end{figure}

\begin{figure}[ht!]
  \centering
  \includegraphics[width=0.92\textwidth]{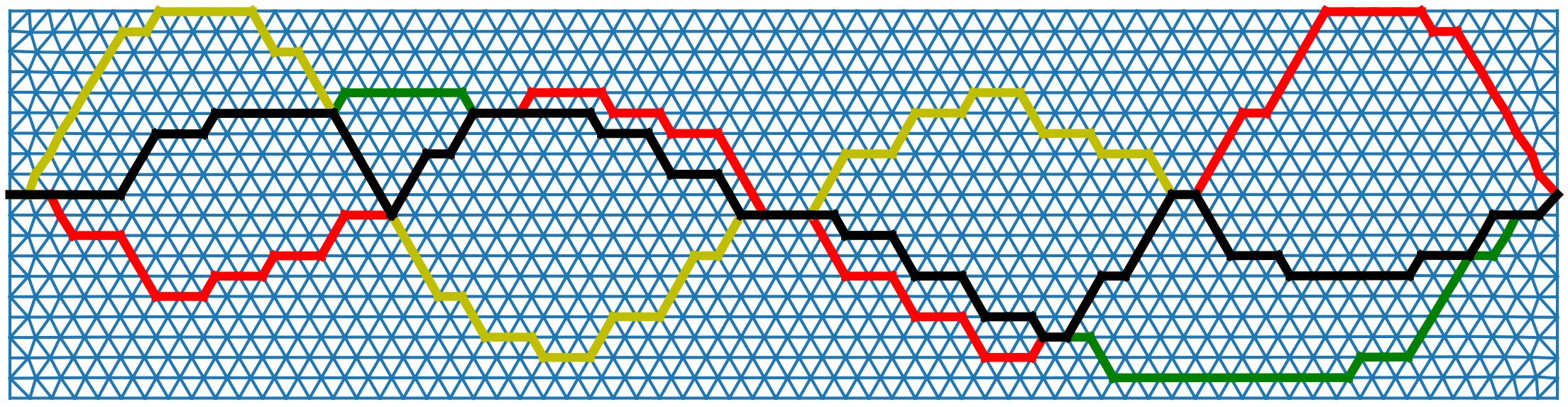}
  \caption{Simplicial median shape of three curves in 2D.
    The third input curve in yellow is added to the two in green and red, which are the input curves in \cref{fig-2D2crvs}.
    The median shape curve is shown in black.   \label{fig-2D3crvs}
  }
\end{figure}

\subsection{Instances in 3D} \label{ssec-3dinst}

We present instances with codimensions $1$ and $2$ in $\R^3$.
In all these instances, the input currents had shared boundaries.
We constructed a $3$-complex tetrahedralizing a $3$-ball, consisting of $45,768$ tetrahedra.
The $2$-skeleton of this complex had $93,149$ triangles and $55,860$ edges.
We considered three input curves, each of which went from the North pole to the South pole, meeting roughly at $120$ degree angles at both poles.
We started with the curves living on the surface, i.e., on the $2$-sphere, and added some noise so that they wiggled into the interior of the $3$-ball at places.
One would expect the median shape to be the diameter connecting the North and south poles, and our computations agreed with this intuition.
\begin{figure}[ht!]
  \centering
  \includegraphics[scale=0.36]{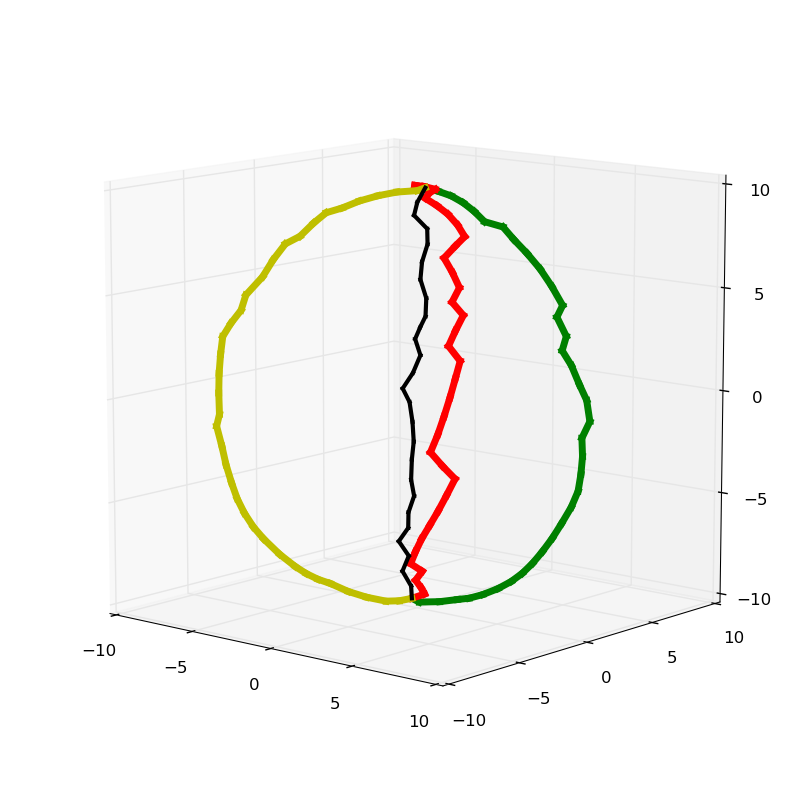}
  \includegraphics[scale=0.36]{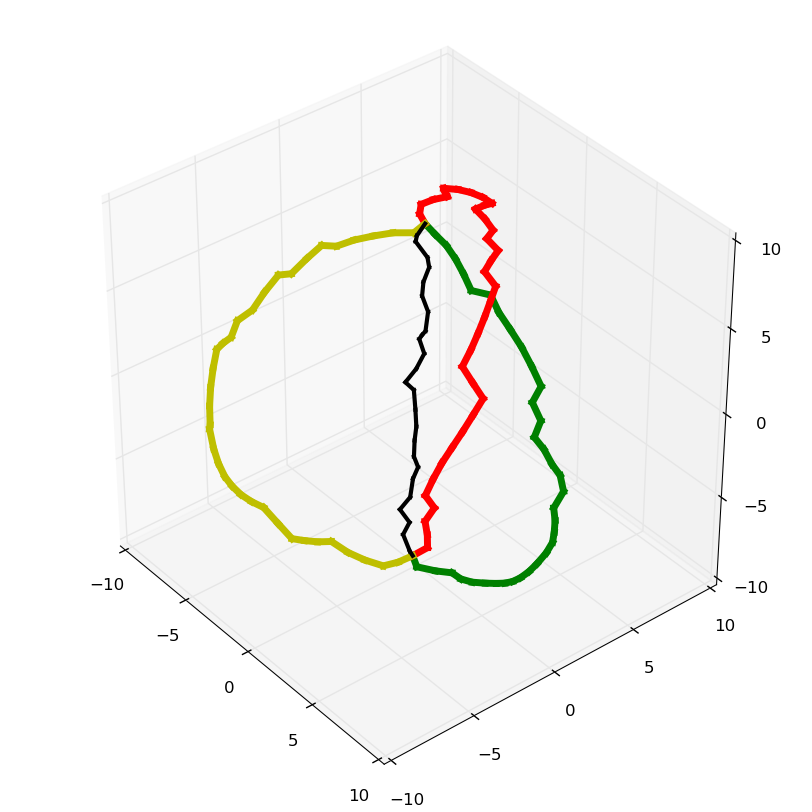}
  \caption{Tow views of the simplicial median shape of three curves in a $3$-ball in $\R^3$.
    The input curves are shown in green, red, and yellow, while the median shape curve is shown in black.   \label{fig-3ball3crvs}
    }
\end{figure}

We next considered two similar input curves (between the poles), and solved the generalized weighted simplicial median shape problem over the $2$-sphere, i.e., the surface of the $3$-ball.
The $2$-complex triangulating the $2$-sphere had $8,695$ edges and $5,788$ triangles.
As we vary the weights $[\alpha_1 ~ \alpha_2]$ from $[1~~0]$ to $[0~~1]$, the median shape changes from the first to the second input curve, all along the surface of the sphere (see \cref{fig-2sphrdeform}).

\begin{figure}[ht!]
  \centering
  \includegraphics[scale=0.38]{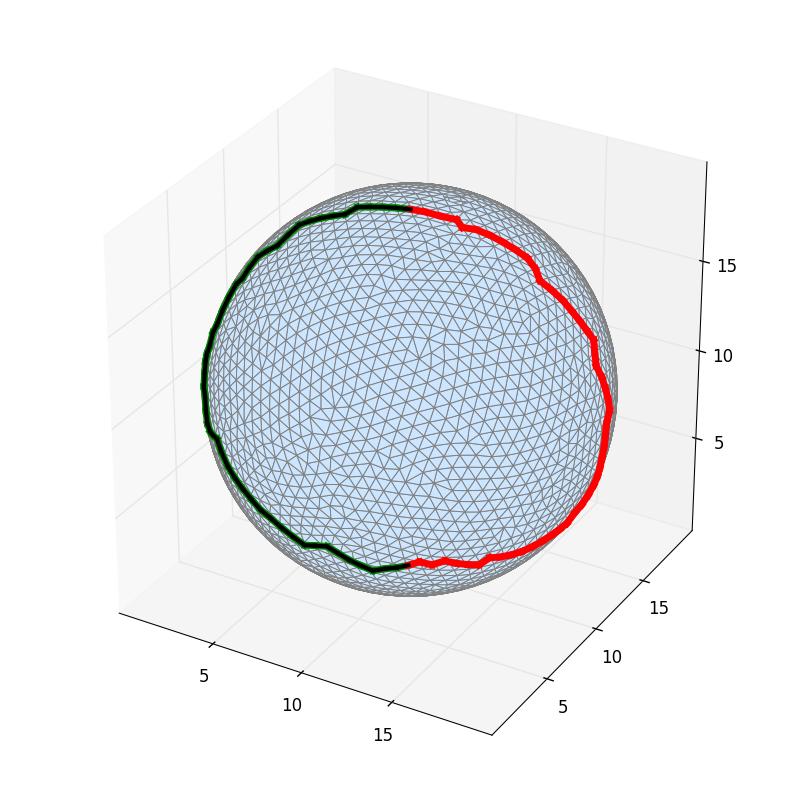}
  \includegraphics[scale=0.38]{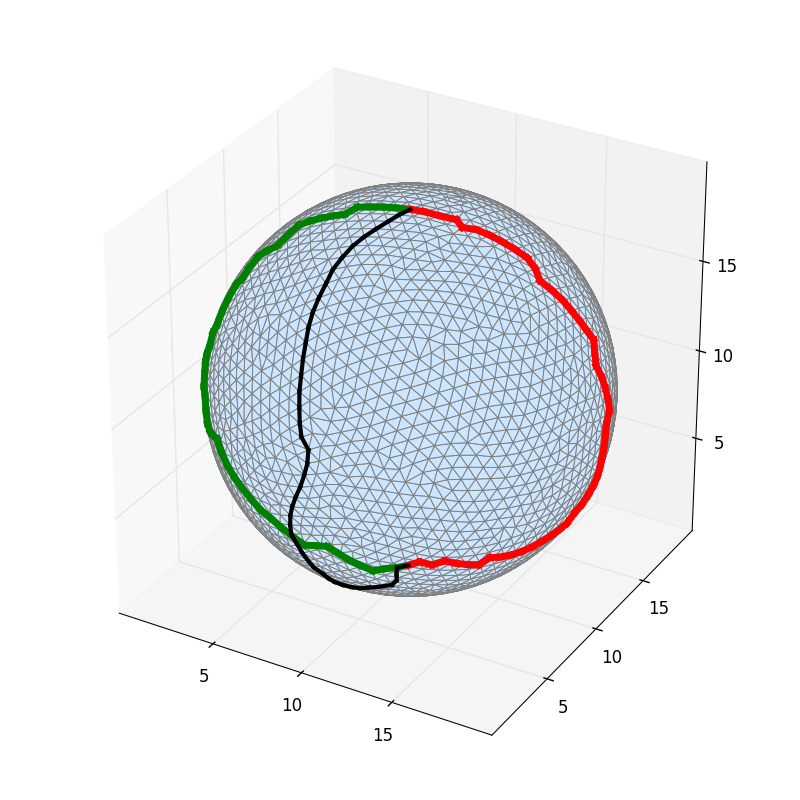}
  \includegraphics[scale=0.38]{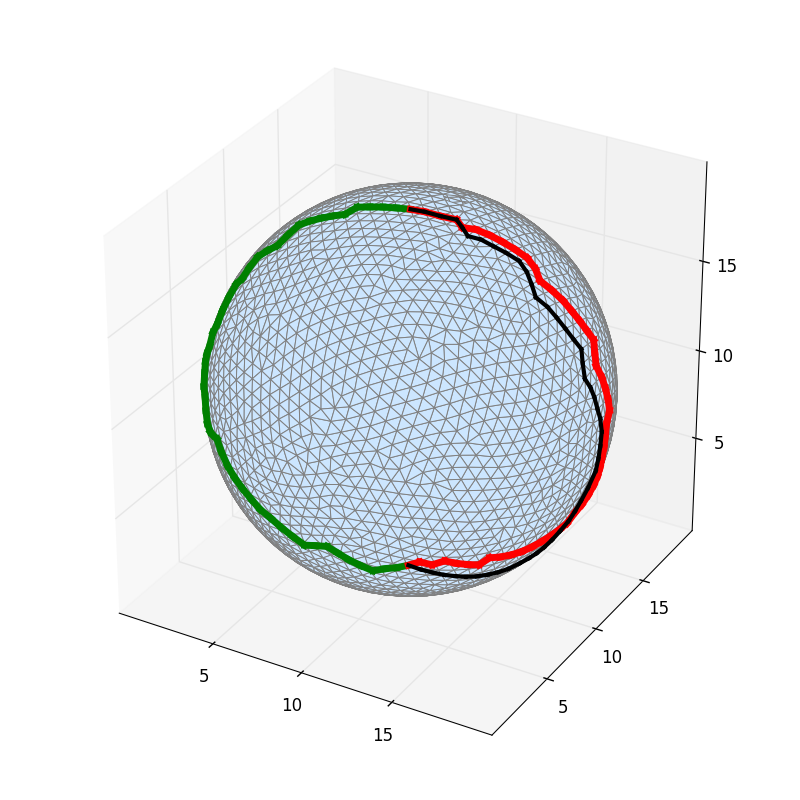}
  \includegraphics[scale=0.38]{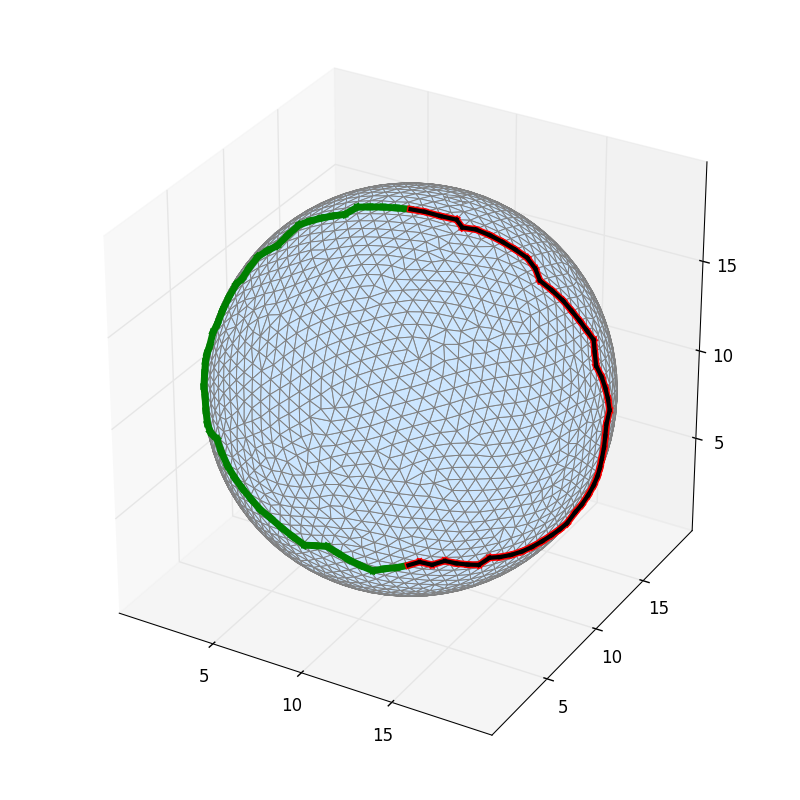}
  \caption{Weighted median shape of $2$ pole-to-pole curves on a $2$-sphere.
    The input curves are shown in green and red, while the median shape curve is shown in black.
    The weights $[\alpha_!~\alpha_2]$ for the two input curves were chosen as $[1~0]$ in the top left figure and $[0~1]$ in the bottom right figure.
    The weights are $[0.5~0.5]$ in the top right figure and $[0.1, 0.9]$ in the bottom left figure.   \label{fig-2sphrdeform}
    }
\end{figure}

To further demonstrate the versatility of simplicial flat norm (as described in \cref{sssec-medshpgenspc}), we present computations on a torus in \cref{fig-torus2crvs}.
The triangulation of the torus had $5,007$ edges and $3,336$ triangles.
We consider two cases each with a pair of input currents---a pair of handle loops and another pair of tunnel loops.

\begin{figure}[ht!]
  \centering
  \includegraphics[scale=0.36]{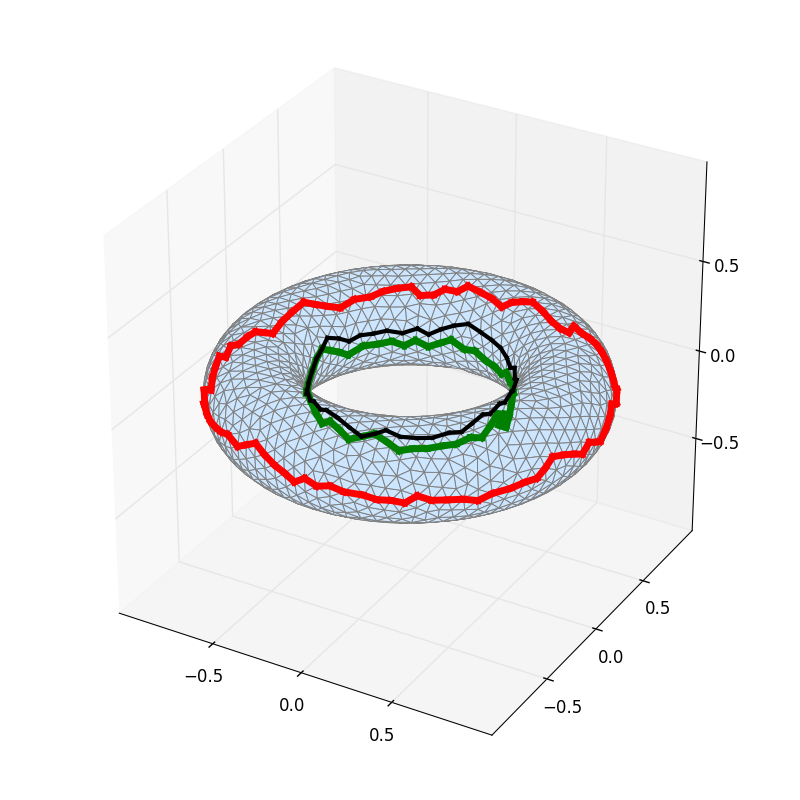}
  \includegraphics[scale=0.36]{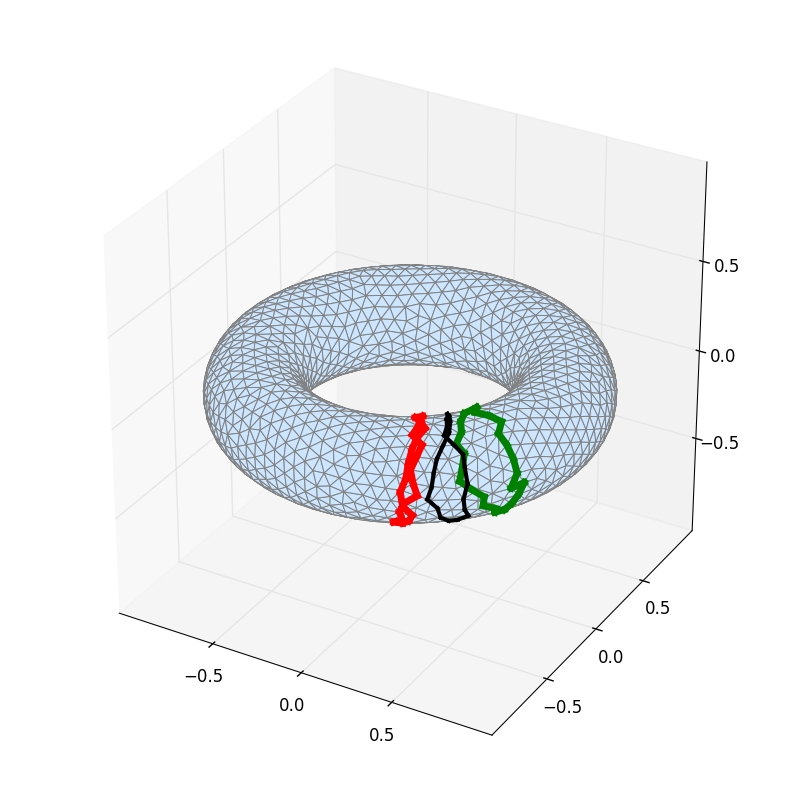}
  \caption{Simplicial median shapes of two tunnel curves and two handle curves on a torus.
    The input curves are shown in green and red, while the median shape curve is shown in black.
    \label{fig-torus2crvs}
    }
\end{figure}

\medskip

It is worth noting that the boundary matrix in question ($[\partial_2]$) is not totally unimodular in the case of the $3$-ball (for computations shown in \cref{fig-3ball3crvs}).
On the other hand, the boundary matrix of the $2$-sphere as well as the torus are indeed totally unimodular (for computations shown in \cref{fig-2sphrdeform} and \cref{fig-torus2crvs}).
At the same time, the constraint matrix of the LPs in all these cases need not be totally unimodular even if the boundary matrix is totally unimodular.   

We have also worked with instances of surfaces in $\R^3$ (i.e., codimension $1$ in 3D).
One such instance is made available as part of our open source software repository available at \\\href{https://github.com/tbtraltaa/medianshape}{https://github.com/tbtraltaa/medianshape}.

\section{Discussion} 
\label{sec:discuss}

\subsection{Theory}\label{sec:theory}

The theory we have presented in the first four sections of this paper
is just a beginning.  We list a few of the directions inviting further
work:
\begin{description}
\item[Big $\lambda$:] We have explored the codimension 1 case of the
  median problem fairly thoroughly, though the study of that problem
  for $\lambda$ that is not small, remains. (By small $\lambda$, we
  mean $\lambda$ small enough to guarantee that for each $i$, the
  multiscale flat norm of $\mdn{T} - \T_i$ equals the mass of $S_i$
  where $\partial S_i = (\mdn{T} - \T_i)$.)
\item[Non-shared Boundaries:] The case of codimension 1 input currents that do
  not share boundaries is completely open. It seems that
  there is a sort of soft transition from shared boundaries to
  non-shared boundaries that could be studied first. By this we mean
  input currents that almost share boundaries in the sense that the
  Hausdorff distance between the supports of the boundaries of all the
  input currents is much smaller than the diameter of any of supports
  the input currents, which in turn are comparable to the diameter of
  the supports of the boundaries. In other words, first study the case
  in which:
\begin{enumerate}
\item Boundaries are close: $\haus(\supp(\partial T_i),\supp(\partial T_j)) \leq \delta$ for all $i$ and $j$
\item we have $\delta << \diam (\supp(T_i))$ for all $i$,
\item and  $\diam (\supp(T_i)) \approx \diam (\supp(\partial T_i))$ for all $i$.
\end{enumerate}
where $\haus(E,F)$ is the Hausdorff distance between the sets $E$ and $F$.
\item[Higher Codimension:] We just scratched the surface of the case of input currents with
  higher codimension. In general higher codimension increases the
  difficulty of studies---see for example the technicalities
  in the study of the regularity of minimizing currents in higher
  codimension by Almgren~\cite{almgren-2000-1}, which were recently illuminated
  by the work of De Lellis and Spodaro, which by itself is still impressively
  large; see De Lellis' overview here~\cite{de-lellis-2010-almgren}
  as well
  as~\cite{de-lellis-2013-regularity-II,de-lellis-2013-regularity-III,de-lellis-2014-regularity-I,de-lellis-2011-q-valued}. See
  also the work they did with Spolaor
  here~\cite{de-lellis-2015-regularity-2d-III,de-lellis-2017-regularity-2d-II,de-lellis-2018-regularity-2d-I}.
\item[Means:] We left the entire subject of flat norm based means open, due to
  the difficulty in computing the means. It is also the case that the
  medians seem a bit more natural geometrically. On the other hand,
  means are closer to unique and their study would almost certainly
  raise interesting theoretical questions
\item[Interpolation:] What sorts of paths in the
  space of currents would be traversed if we introduced time evolving
  $\lambda_{i(t)}$'s for each $T_i$ so that the resulting objective
  function becomes $\sum_{i=1}^{\nCur}\F_{\lambda_{i(t)}}(T-T_i)$? How
  we can smoothly interpolate between shapes is of practical interest
  if the computation of those paths could be made tractable.
\end{description}

\subsection{Computation}\label{sec:compute-discuss}
It is rather surprising that we are obtaining integer optimal solutions for the median shape LPs even when the constraint matrices are not guaranteed to be totally unimodular.
Could we characterize the classes of simplicial complexes for which this property holds?
Previously, we had presented a class of simplicial complexes that are {\em non total-unimodularity neutralized} \cite{KrSm2013}, on which instances of the optimal homologous chain problem (OHCP) linear program are guaranteed to have integer optimal solutions even when the boundary matrix is not totally unimodular.
At the same time, this characterization depended critically on the coefficients of the $(p+1)$-dimensional simplices, e.g., triangles in the edge-triangle case, being all zero in the objective function.
We do not have this condition satisfied in the simplicial flat norm LP or in the median shape LP.

While we are able to solve the simplicial median shape problem as a linear program, the LPs themselves could be quite large in size, and take a long time to solve in practice.
For instance, the median shape LPs in the $3$-ball examples (shown in \cref{fig-3ball3crvs}) had more than a million columns ($1,005,774$ to be exact).
Could we design an algorithm that solves the median shape LP much faster than general LPs, the most efficient algorithms for which take time that grows as the cube of the number of columns?

% \bibliography{flatnorm,homology}

\bibliographystyle{plain}
%\bibliography{the_bib,yh-median-shapes,flatnorm,homology,mybib}

\end{document}